\documentclass[11pt]{amsart}
\pdfoutput=1

\usepackage{latexsym}
\usepackage{graphicx} 
\usepackage{amsmath} 
\usepackage{amsfonts}
\usepackage{amssymb} 
\usepackage{mathtools} 
\usepackage{bbm}
\usepackage{mathrsfs}
\usepackage{extarrows}
\usepackage{color}

\usepackage{multirow}
\usepackage{diagbox}
\usepackage{threeparttable}
\usepackage{booktabs} 
\usepackage{tabularx}
\usepackage{array}
\usepackage{ragged2e} 
\usepackage{siunitx} 

\usepackage{enumitem}
\usepackage[percent]{overpic}
\usepackage{algorithm}
\usepackage{algorithmicx}
\usepackage{algpseudocode}
\usepackage{caption} 
\usepackage{subcaption}

\usepackage[left=1.25in,right=1.25in,top=1in]{geometry}
\usepackage{amsthm}
\usepackage{tikz}

\usepackage{hyperref}

\floatname{algorithm}{Algorithm}

\newcolumntype{L}{>{$}l<{$}} 
\newcolumntype{Y}{>{\RaggedRight\arraybackslash}X} 
\newcolumntype{M}{>{$}c<{$}} 

\DeclareCaptionLabelFormat{subfiglabel}{#1~(#2)} 


\usepackage{soul}

\newcounter{Rownumber} 

\newtheorem{theorem}{Theorem}[section]

\newtheorem{remark}{Remark}
\newtheorem{Proposition}{Proposition}[section]

\makeatletter

\newtheoremstyle{condstyle}%
  {\abovedisplayskip}{\belowdisplayskip}
  {}
  {}
  {\bfseries}
  {}
  {0.5em}
  {\thmname{#1}\ \thmnumber{#2.}\thmnote{\ \textnormal{(#3)}}}

\theoremstyle{condstyle}
\newtheorem{condition}{Condition}

\theoremstyle{plain} 
\title{Inverse acoustic scattering for random obstacles with multi-frequency data }

 \author{zhiqi sun}
 \thanks{School of Mathematical Sciences, Shanghai Jiao Tong University, Shanghai 200240, People’s Republic of China. Email: sunzhq1016@sjtu.edu.cn}

 \author{xiang xu}
\thanks{School of Mathematical Sciences, Zhejiang University, Hangzhou 310027,  People’s Republic of China. Email: xxu@zju.edu.cn}

 \author{yiwen lin}
 \thanks{School of Mathematical Sciences, Shanghai Jiao Tong University, Shanghai 200240,   People’s Republic of China. Email: linyiwen@sjtu.edu.cn}

\begin{document}
\graphicspath{figures/}
\setstcolor{red}  
\maketitle

\begin{abstract}

We study an inverse random obstacle scattering problems in $\mathbb{R}^2$ where  the scatterer is formulated by a Gaussian process defined on the angular parameter domain. Equipped with a modified covariance function which is mathematically well-defined and physically consistent, the Gaussian process admits a   parameterization via Karhunen--Lo\`eve (KL) expansion. Based on observed multi-frequency  data, we  develop a two-stage inversion method: the first stage reconstructs the baseline shape of the random scatterer and the second stage estimates  the statistical characteristics of  the boundary fluctuations, including KL eigenvalues and covariance hyperparameters. We further provide theoretical  justifications for the modeling and  inversion pipeline, covering well-definedness of the Gaussian-process model, convergence for the two-stage procedure and  a brief discussion on uniqueness. Numerical experiments demonstrate stable recovery of both geometric and statistical information for obstacles with simple and more complex shapes.

\end{abstract}


\pagestyle{myheadings}
\thispagestyle{plain}
\markboth{Z. Sun, X. Xu, Y. Lin}
{RANDOM INVERSE OBSTACLE PROBLEM}

\section{Introduction}

Direct and inverse scattering problems are of great importance in mathematical physics, particularly the \emph{inverse scattering problem} (ISP) has attracted growing attention in recent years due to its wide range of important applications across numerous fields such as non-destructive testing \cite{bao2001mathematical}, medical imaging \cite{kuchment2013radon} and beyond. A classic and significant subclass within ISP are the \emph{inverse obstacle scattering problems} (IOSP). In this paper we focus on the IOSP in $\mathbb{R}^2$ where the goal is to recover the boundary of an unknown scatterer from noisy far-field measurements. 

The inverse obstacle scattering problem  is an essential topic in inverse scattering research, with extensive studies devoted to its modeling, analysis and numerical reconstruction in both two and three dimensions. Depending on the underlying physical setting, the scattering measurements may correspond to different wave models including acoustic waves \cite{colton1986novel,zhang2022direct}, elastic waves \cite{yue2019numerical,dong2019inverse} and electromagnetic waves \cite{cakoni2004electromagnetic,li2019inverse}. From an algorithmic standpoint, existing approaches can be broadly divided into classical computational methods and deep learning–based methods: classical methods are often grouped into iterative and direct methods, representative iterative strategies include gradient-based optimization \cite{audibert2022accelerated} and Newton \cite{chang2024novel} (or Gauss–Newton type \cite{sini2012inverse}) methods, while direct methods include sampling methods \cite{colton1996simple,ji2018direct} and factorization methods \cite{kirsch2007factorization}, among others. Beyond these two paradigms, hybrid approaches that combine iterative updates with direct reconstruction ideas have also been proposed,  e.g., see \cite{li2020extended}. Fueled by rapid development of deep learning,  an increasing number of studies have begun employing the neural network approaches to investigate IOSP,  we refer to some representative works including  \cite{khoo2019switchnet,chen2024solving,yin2025tddm} and the review paper \cite{chen2020review}. Integrated schemes that combine classical computational methods with deep learning methods have also been proposed in \cite{ning2025direct}. The method developed in this work falls within the class of classical iterative reconstruction algorithms. 

In many practical applications the true geometry and material properties of a scatterer are rarely known exactly and may vary from one realization to another. Such variability can originate from manufacturing tolerances, small cracks or wear accumulated during operation and environmental effects such as temperature-dependent deformation and contamination. These factors introduce intrinsic uncertainty in  scattering mechanism which cannot be adequately captured by purely deterministic models. To account for these uncertainties, it is natural to introduce the random   model into the  problems of our interests. Indeed, this uncertainty-aware viewpoint has been adopted in several inverse problem settings, for instance, in inverse problems with random sources, the source term  is modeled by uncertain models including white-noise-type sources \cite{bao2014inverse} or the  generalized Gaussian random fields \cite{li2020inverse}; in inverse problems with random potentials the potential term is modeled by the Gaussian random field \cite{li2024inverse}. As far as we know,  most existing  inverse  studies have focused on randomness in the source, medium or potential terms whereas the  random obstacle scattering appears to remain largely open.  In a closely related direction,  \cite{bao2020inverse}  considered an inverse random periodic surface problem in  $\mathbb{R}^1$ where they modeled the  random grating  as a  Gaussian process parameterized by a Karhunen--Lo\`eve (KL) expansion, and recovered its statistical parameters via an MCCUQ-based approach. On the other hand, there is also a body of work devoted to forward scattering with uncertainties  including random gratings in $\mathbb{R}^1$ \cite{bao2018robust} and random obstacles in $\mathbb{R}^2$ \cite{xiu2007efficient}. These studies primarily target the computation of statistical quantities of the wave field, i.e., forward uncertainty quantification, but do not address the  inverse recovery.
In this paper we try to extend  the model framework of \cite{bao2020inverse} to a two-dimensional setting, formulate an inverse random obstacle scattering problem  and develop  a reconstruction scheme which is capable of  recovering the obstacle shape and inferring the statistical characterization of the random scatterer. 

 More specifically, we consider a two-dimensional star-shaped obstacle whose boundary is parameterized in polar coordinates by a radius function $r(\theta)$ with $\theta\in[0,2\pi):=\Theta$. We decompose  $r(\theta)$  into a deterministic baseline component and a random fluctuation where the latter random part will be  modeled as a one-parameter Gaussian process indexed by $\theta$. To reflect the geometry of the angular domain,  we construct a $2\pi$-periodic covariance function in which the correlations decay with the minimal angular separation on $\Theta$ and then apply a mild correction to ensure that the resulting kernel is mathematically well-defined. The random fluctuation is  subsequently represented by  a truncated Karhunen--Lo\`eve (KL) expansion, yielding  a finite-dimensional parametrization of the random geometry. On the computational side, we develop a two-stage reconstruction framework inspired by the multi-frequency recursive linearization approach (RLA) of Sini et al. \cite{sini2012inverse}: in the first stage, we reconstruct the  mean geometry  from multi-frequency far-field data; in the second stage, taking the first-stage reconstructions across realizations as input, we infer the statistical descriptors of the geometric uncertainty, including the KL spectrum and the covariance hyperparameters (the variance $\sigma^2$ and correlation length $\ell$). We also provide  theoretical analysis of the proposed modeling  and inversion algorithm. Numerical results are reported in the experiments section, while additional proofs and derivations are deferred to the appendix. The numerical experiments demonstrate that the proposed method can recover both the mean geometry and the associated  statistical quantities  with good accuracy for a range of test cases, including simple and relatively complex geometries. We summarize the main contributions of this work as follows:
\begin{itemize}
\item \textbf{Random obstacle model and covariance design in $\mathbb{R}^2$}. We propose a random geometric model for two-dimensional inverse obstacle scattering problems where the random boundary of the scatterer  is modeled as a Gaussian process on the angular parameter domain. We construct a $2\pi$-periodic covariance kernel adapted to this setting, ensuring that the induced Gaussian process is mathematically well-defined, physically consistent and compatible with a truncated KL parameterization.
\vspace{0.15cm}
\item \textbf{Two-stage reconstruction algorithm}. We develop a two-stage  framework that (i) reconstructs the mean geometry from multi-frequency data and (ii) estimates statistical descriptors of the boundary fluctuations, including the KL spectrum and covariance hyperparameters.
\vspace{0.15cm}
\item \textbf{Theoretical analysis}. We provide theoretical analysis of the proposed framework, including well-definedness of the Gaussian process model, convergence estimates for the two-stage algorithm and a brief  discussion on uniqueness.
 \end{itemize}
\vspace{0.25cm}

The rest of this paper is organized as follows. In \autoref{Problem setup and preliminaries}, we introduce the random obstacle model, including the Gaussian-process representation of the random obstacle, the covariance construction tailored to our periodic angular parameterization and the  Karhunen--Lo\`eve parameterization. In \autoref{Reconstruction method}, we present the proposed two-stage reconstruction algorithm: the first stage recovers the mean geometry based on multi-frequency recursive linearization scheme  while the second stage estimates the statistical parameters of the random scatterer. \autoref{Theoretical analysis} is devoted to the theoretical analysis where we mainly derive convergence estimates for the two-stage procedure, together with a brief  discussion on uniqueness. In \autoref{Numerical_experiments}, we report numerical experiments to validate the proposed approach on both simple and complex geometries. Additional proofs and other corresponding derivations are listed in \autoref{appendix}.

\section{Problem setup and preliminaries}\label{Problem setup and preliminaries}

\subsection{Deterministic direct and inverse scattering problem}
As illustrated in \autoref{schematic}, we consider a bounded obstacle $D \subset \mathbb{R}^2$  with $C^3$-smooth boundary $\partial D$.  The incident field is taken to be a time-harmonic acoustic plane  wave
\begin{equation}\label{01}
    u^i(x, d) = \mathrm{e}^{\mathrm{i} k x \cdot d}, \quad x \in \mathbb{R}^2,
\end{equation}
where $k > 0$ represents the wavenumber and $d \in \mathbb{S}^1$ denotes the direction of propagation.  The total field $u$ is defined as  $u(x) = u^i(x) + u^s(x)$ in $\mathbb{R}^2 \setminus \overline{D}$. Since the obstacle is assumed to be sound-soft, $u$ satisfies a homogeneous Dirichlet boundary condition on $\partial D$. It turns out the total field $u$ and the scattered field $u^s$ satisfy the exterior Helmholtz problem:
\begin{subequations}
\begin{align}
    \Delta u + k^2 u &= 0, \quad \text{in } \mathbb{R}^2 \setminus \overline{D}, \label{eq:helmholtz} \\
    u = u^s + u^i &= 0, \quad \text{on } \partial D, \label{eq:bc} \\
    \lim_{r\to\infty} \sqrt{r} \left( \frac{\partial u^s}{\partial r} - \mathrm{i} k u^s \right) &= 0, \quad r=|x|, \label{eq:src}
    \end{align}
\end{subequations}
 equation \eqref{eq:src} is the celebrated Sommerfeld radiation condition which guarantees the uniqueness of the solution and characterizes the outgoing nature of the scattered waves. It is a classical result that the boundary value problem \eqref{eq:helmholtz}--\eqref{eq:src} is well-posed provided that the boundary $\partial D$ is Lipschitz continuous \cite{mclean2000strongly}. Furthermore, the scattered field $u^s$ exhibits the following asymptotic behavior at infinity:
\begin{equation}\label{eq:farfield_expansion}
    u^s(x) = \frac{\mathrm{e}^{\mathrm{i}k|x|}}{\sqrt{|x|}} u^\infty(\hat{x},d,k) + O\left(|x|^{-3/2}\right), \quad |x| \to \infty, 
\end{equation}
where $\hat{x} := x/|x|$ denotes the observation direction. The function $u^\infty(\hat{x})$, which is analytic on the unit circle $\mathbb{S}^1$, is referred to as the \emph{far field pattern} of the scattered field \cite{colton1998inverse}.

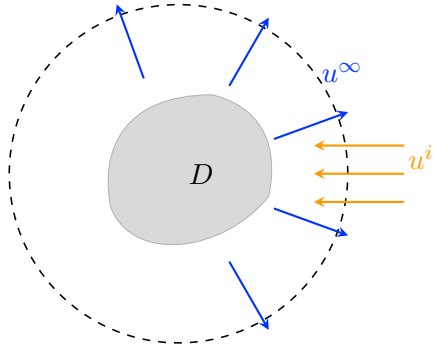
\begin{figure}[ht]
  \centering
  \begin{tikzpicture}[scale=1.5,>=stealth]

    \draw[black!60!black, dashed, line width=0.6pt] (0,0) circle (1.5);

    \begin{scope}[shift={(0.1,0)}]
      \draw[fill=gray!30, draw=gray!60]
        (0.2,0.7) .. controls (-0.6,0.7) and (-0.8,0.2) .. (-0.7,-0.3)
                  .. controls (-0.5,-0.8) and (0.3,-0.7) .. (0.7,-0.2)
                  .. controls (0.8,0.3) and (0.6,0.6) .. (0.2,0.7);
    \end{scope}
    \node at (0.2,0) {$D$};

    \foreach \y in {-0.25,0,0.25}{
      \draw[->, thick, orange!80!yellow]
        (2,\y) -- (1.2,\y);
    }
    \node[orange!80!yellow] at (2.15,0.15) {$u^i$};

    \foreach \ang in {-60,-20,20,60,110}{
      \draw[->, thick, blue!80!cyan]
        ({0.9*cos(\ang)},{0.9*sin(\ang)}) --
        ({1.6*cos(\ang)},{1.6*sin(\ang)});
    }
    \node[blue!80!cyan] at (1.45,0.9) {$u^\infty$};

  \end{tikzpicture}
 
  \caption{schematic of the inverse obstacle scattering problem with a single incident direction.}
   \label{schematic}
\end{figure}

The two-dimensional inverse obstacle scattering problem has been studied extensively both theoretically and computationally. Reconstruction methods are often distinguished by their data-acquisition paradigms. In single-frequency settings, measurements may be full-aperture with a single incident direction (e.g., \cite{kress2003newton}) or with multiple incident directions (e.g., \cite{colton1986novel}). Limited-aperture data, where both incident directions and the observation aperture are restricted, are substantially more challenging but practically relevant (see, e.g., \cite{yin2025physics}); a common remedy is to complete the missing full-aperture data numerically (e.g., \cite{liu2019data,dou2022data}), albeit at a high computational cost. Multi-frequency formulations have also been widely investigated as they can alleviate the ill-posedness of fixed-frequency reconstructions and improve stability \cite{bao2015inverse}. For multi-frequency shape reconstruction in two dimensions, one may use single-incident-direction data (e.g., \cite{sini2012inverse}) or multiple-incident-direction data (e.g., \cite{colton1985novel}). Motivated by operational feasibility when acquiring data from many incident angles is impractical or prohibitively time-consuming, in this work we focus on the single-incident-direction, multi-frequency setting and leave extensions to multiple incident directions for future study.  Accordingly, we assume that the available data consist of multi-frequency far-field measurements
\[
    \left\{ u^\infty(\hat{x}_m, d, k_j) \right\}_{m=1,\dots,M}^{j=1,\dots,n_K}
\]
generated by a fixed incident direction $d$, observation directions $\{\hat{x}_m\}_{m=1}^M\subset \mathbb S^1$ and wavenumbers $\{k_j\}_{j=1}^{n_K}\subset\mathbb Z_+$. Our goal is to reconstruct the obstacle boundary (geometry) from these measurements.


\subsection{Star–shaped random obstacle model and KL parameterization}\label{ranmo}

As noted earlier, in many practical applications the scatterer's geometry exhibits uncertainty due to manufacturing or environmental variations, hence introducing an appropriate random geometric model to study the associated forward and inverse problems is of significant theoretical and practical value. To make the resulting scattering problem amenable to iterative shape-reconstruction methods, we restrict our attention to star-shaped obstacles. More precisely, we assume that the deterministic scatterer $D$ is star-shaped with respect to  point $x_0$ so that its boundary can be parameterized in polar coordinates as 
\begin{equation}{\label{orsur}}
\partial D=\left\{x_0+r(\theta)(\cos \theta, \sin \theta): \theta \in\Theta:=[0,2 \pi)\right\},
\end{equation}
the radial function \( r \) is assumed to be positive on \( \Theta \) and satisfies \( r(0) = r(2\pi) \). To incorporate geometric uncertainties,  we introduce the following random perturbation model for star–shaped obstacles which can be represented as:
\begin{equation}\label{eq:star-shaped}
    \partial D(\omega)
    = \Big\{ x(\theta,\omega)
        = x_0+r(\theta,\omega)\,(\cos\theta,\sin\theta)
        : \theta\in\Theta
      \Big\},
\end{equation}
here $r(\theta,\omega)>0$ is a positive random radius function in direction $\theta$ and $\omega$  belongs to a given probability space  $(\Omega,\mathscr{F},\mathbb{P})$. The random radius $r(\theta,\omega)$ is characterized by the sum of a deterministic component $ r_{\mathrm{true}}(\theta) $ and a random fluctuation $\delta r(\theta,\omega)$: 
\begin{equation}\label{eq:r-decomp}
    r(\theta,\omega)
    =  r_{\mathrm{true}}(\theta) + \delta r(\theta,\omega).
\end{equation} 
We assume that  $\delta r(\theta,\omega)$ in \eqref{eq:r-decomp} is  a real-valued, zero-mean Gaussian process on $\Theta$ with a constant variance $\sigma^2$: 
\begin{gather}\label{eq:delta r}
\delta r : \Theta \times \Omega \to \mathbb{R}, \quad \theta \in\Theta,\quad \omega\in \Omega ,\\
\mathbb{E}[\delta r(\theta, \cdot)] = 0, \quad \mathbb{V}[\delta r(\theta, \cdot)] = \sigma^2, \quad \theta \in [0, 2\pi). \label{condi}
\end{gather}
In the regime of interest, we also assume the random perturbation $\delta r$ has sufficiently small amplitude relative to the true radius $r_{\text{true}}$ so that $r(\theta,\omega)$ stays positive and defines an admissible boundary for all $\theta$. Since a vast array of Gaussian processes meet the conditions \eqref{eq:delta r},\eqref{condi}, in order to tailor a process that aligns with our specific physical and statistical requirements of the two-dimensional scatterers, we therefore proceed to focus on the construction of covariance function. 

In our setting, the covariance function  of the Gaussian process should possess the following properties: first, it is required to be stationary on the angular domain, i.e., it depends only on the angular separation $d\left(\theta,\theta^{\prime}\right)$ rather than on the absolute angles $\theta$ and $\theta'$ themselves:
\begin{condition}[\textbf{Stationarity}]\label{sta1} 
    \begin{equation}\label{sta}
\operatorname{Cov}\left(\delta r(\theta), \delta r(\theta^{\prime})\right)=C\left(d(\theta,\theta^{\prime})\right)\xlongequal{d(\theta,\theta^{\prime})=t}{}C(t).
\end{equation}
\end{condition}
\noindent  To reflect the underlying physics on a circular geometry where the correlation between $\delta r(\theta)$ and $\delta r(\theta')$ should be  measured by the geodesic distance (i.e., the shortest arc distance), $d\left(\theta,\theta^{\prime}\right)$ in \eqref{sta} is defined through the  geodesic distance  by
\begin{condition}[\textbf{Shortest arc dependence}]\label{law}
\begin{equation}\label{dd}
     d(\theta,\theta'):=\min\{|\theta-\theta'|,\; 2\pi-|\theta-\theta'|\}\in[0,\pi].
\end{equation}
\end{condition}
\noindent Thus the angular difference $|\theta-\theta^{\prime}|>\pi$ are interpreted through the complementary shorter arc $2\pi-|\theta-\theta'|$. 
 Moreover, the covariance function is also required to satisfy the symmetry condition: 
\begin{condition}[\textbf{Symmetry}]\label{symm1}
    \begin{equation}\label{symm}
   C(t)=\mathbb{E}[\delta r(0) \delta r(t)]=\mathbb{E}[\delta r(t) \delta r(0)]=C(-t).
\end{equation}
\end{condition}

    \noindent Under Condition \autoref{sta1}, the requisite Condition \autoref{symm1} is automatically fulfilled. Condition \autoref{symm1} also implies that $C(t)$ is an even function (when viewed as a $2\pi$-periodic function on $\mathbb{R}$). Furthermore, since the underlying geometry is invariant under rotations by $2n\pi$ ($n \in \mathbb{Z}$), the statistical structure remains unchanged under such transformations; consequently, it is natural and indeed necessary for covariance function $C(t)$ to be $2\pi$-periodic as a function of the angular difference:
\begin{condition}[\textbf{Periodicity}]\label{peo1}
    \begin{equation}\label{peo}
    C(t)=C(2\pi+t).
\end{equation}
\end{condition}

\noindent  Finally, according to \emph{Daniell--Kolmogorov theorem} \cite{oksendal2013stochastic,baudoin2012lecture5}, a Gaussian process with a prescribed covariance function $\operatorname{Cov}\left(\delta r(\theta_i), \delta r\left(\theta_k\right)\right)=C(\theta_i-\theta_k)$ exists provided that the covariance kernel  $C$ is shown to be symmetric and positive semi-definite, therefore  in addition to Conditions \autoref{sta1}-\autoref{peo1}, we also require that:
\begin{condition}[\textbf{Positive semi-definiteness}]\label{pd1}
    \begin{equation}\label{pd}
     \sum_{i,k=1}^{n} b_i b_k C(\theta_i - \theta_k) \geq 0,
\end{equation}
\end{condition}\noindent
 for any finite set of angles $\{\theta_1, \dots, \theta_n\}$ and real coefficients $\{b_1, \dots, b_n\}$. 

A classical covariance model introduced in \cite{xiu2007efficient} for random scatterers on the angular domain $\Theta=[0,2\pi)$  is defined by
\begin{equation}\label{eq:cov-1d}
    C_{\mathrm{geod}}(t)
   =\sigma^2 \exp\!\Big(
        -\frac{1}{\ell^{2}}\min\{t,\,2\pi-t\}^{2}
      \Big),
    \quad  t\in\Theta,
\end{equation}
which can be viewed as applying the \emph{standard Gaussian 
covariance function}: 
\begin{equation}\label{standard} 
     C_f(\widetilde{t})
    = \sigma^2 \exp\!\Big(-\frac{\widetilde{t}^2}{\ell^2}\Big), \qquad \widetilde{t}\ge 0,
\end{equation}
to the shortest-arc distance on the circle, i.e., $\widetilde{t}=\min\{t,2\pi-t\}\in[0,\pi]$.
Here $\sigma^2$ is the variance and $\ell$ is the correlation length. Simple verification demonstrates that $C_{\mathrm{geod}}(t)$ complies with the Conditions \autoref{sta1}-\autoref{peo1} but fails to satisfy the constraint imposed by Condition \autoref{pd1}.
In fact, \cite{da2023gaussian} pointed out that if a Gaussian process defined on $\Theta$ has a kernel of the form 
\begin{equation}\label{shit}
     k(\theta_i, \theta_j) = \exp(-\zeta d(\theta_i,\theta_j)^2), \  \zeta > 0,
\end{equation}
with $  d(\theta_i, \theta_j)$ defined in \eqref{dd}, then for any fixed $\zeta > 0$ there exists a sufficiently large $N_\theta$ for which the Gram matrix generated from $N_\theta$ equally spaced points on $\Theta$ with kernel in \eqref{shit} must contain at least one negative eigenvalue. This implies that the positive semi-definiteness condition mentioned in \eqref{pd} for $C_{\mathrm{geod}}(t)$ is violated. To restore positive semi-definiteness and thereby obtain a valid covariance function, we modify \eqref{eq:cov-1d} and construct a corrected covariance function that satisfies Condition \autoref{pd1}. 
For a covariance function $C(t)$ satisfying Conditions \autoref{sta1}-\autoref{peo1}, it is $2\pi$-periodic and satisfies  $C(t)=C(-t)$ upon periodic extension to $\mathbb R$. Consequently, it admits a cosine Fourier series:
\begin{equation}\label{Fourier}
    C(t)=a_0+\sum_{j=1}^{\infty} a_j \cos (j t), \quad t \in[0,2 \pi] ,
\end{equation}
where
\begin{equation}\label{coeff}
    a_0=\frac{1}{2 \pi} \int_0^{2 \pi} C(t) \mathrm{d} t, \quad a_j=\frac{1}{\pi} \int_0^{2 \pi} C(t) \cos (j t) \mathrm{d} t, \quad j \geq 1.
\end{equation}
Meanwhile, we take a uniform grid on $\Theta$ given by $\theta_{m} = \frac{2\pi m}{N_{\theta}}$ for any $N_{\theta}$ and construct the Gram matrix $K_{mn}=C(\theta_m-\theta_n)$. As $C(t)$ depends only on the difference of the angles, $K_{mn}$ forms a circulant matrix. The following theorem will be of great utility:
\begin{theorem}\label{allpo}
Let $C(t)$ be a covariance function defined on $\Theta$ satisfying Conditions \autoref{sta1}-\autoref{peo1}. Suppose that $C(t)$ admits a cosine–series representation as in \eqref{Fourier} with real coefficients $a_j$ given by \eqref{coeff} and $\sum_{j\geq0}|a_j|<\infty$, then the following statements are equivalent:
\begin{enumerate}[label=(\roman*), ref=(\roman*)] 
    \item $C(t)$ is positive semi-definite in the sense of \eqref{pd}. \label{stmt:i}
    \item For every $N_\theta\geq 1$, letting $\theta_m = 2\pi m/N_\theta$ and $K=K^{(N_\theta)}_{mn}:=C(\theta_m-\theta_n)$, the circulant matrix $K^{(N_\theta)}$ is positive semi-definite.  \label{stmt:ii} 
    \item  All cosine coefficients in \eqref{Fourier} are nonnegative, i.e. \[a_j\geq 0, \quad \forall j\geq 0.\]  \label{stmt:iii}
\end{enumerate}  
In particular, under Conditions \ref{sta1}-\ref{peo1} together with any one of \ref{stmt:i}–\ref{stmt:iii}, $C(t)$ is a valid covariance function on $\Theta$ and hence there exists a zero-mean stationary Gaussian process $\{\delta r(\theta)\}_{\theta\in\Theta}$ with covariance $C(t)$.
\end{theorem}
\begin{proof}
    See \autoref{appen:Gaussexi}.
\end{proof}
\noindent From \autoref{allpo}, the failure of $C_{\mathrm{geod}}(t)$ to satisfy Condition~\autoref{pd1} can be traced to the fact that some cosine-series  coefficients in its expansion are negative. Since, under the uniform  discretization, these coefficients correspond to the eigenvalues of the associated circulant Gram matrix up to a constant scaling factor, this also implies that the discrete spectrum may contain a few small-magnitude eigenvalues that become negative.  At present, there is no closed-form characterization of how the first index $j$ with $\lambda_j^{\mathrm{geod}}<0$ depends on the grid size $N_\theta$ ,the conclusion in \cite{da2023gaussian} only states that $N_{\theta}$ needs to be large enough. To obtain a positive semi-definite covariance while preserving the qualitative shape of $ C_{\mathrm{geod}}(t)$, we work in the spectral domain. Compute the cosine coefficients of $ C_{\mathrm{geod}}(t)$ according to \eqref{coeff} yields:
\begin{equation}\label{origcoe}
    a_0^{\mathrm{orig}} = \frac{1}{2\pi} \int_0^{2\pi} C_{\mathrm{geod}}(t) \, dt,
    \qquad
    a_j^{\mathrm{orig}} = \frac{1}{\pi} \int_0^{2\pi} C_{\mathrm{geod}}(t) \cos(j t) \, dt, \quad j \geq 1.
\end{equation}
We then define a modified sequence $\{\widetilde{a}_j\}$ by projecting the negative entries to zero:
\begin{equation}\label{pro}
    \widetilde{a}_j := \max\left\{ a_j^{\mathrm{orig}},\, 0 \right\}, \quad j \geq 0.
\end{equation}
The corrected covariance function is finally defined as
\begin{equation}\label{correc}
     {C}_{\delta r}(t) = \widetilde{a}_0 + \sum_{j=1}^{\infty} \widetilde{a}_j \cos(j t), \quad t \in [0, 2\pi).
\end{equation}
By \autoref{allpo}, ${C}_{\delta r}(t)$ is positive semi-definite and defines a valid covariance function on $\Theta$. Furthermore, the correction does not compromise the convergence of the cosine-series representation in \eqref{Fourier} since it only replaces the negative coefficients by $0$ (i.e., applies the projection $a_j\mapsto \max\{a_j,0\}$) which  only decreases the coefficient magnitudes and  preserves  the absolute summability required for series convergence. Hence after the modification \eqref{origcoe}–\eqref{correc} the covariance kernel ${C}_{\delta r}(t)$  satisfies all the required  Conditions \autoref{sta1}-\autoref{pd1},  which ensures the existence of the associated Gaussian process according to  Daniell–Kolmogorov theorem.

With the definition of the covariance function ${C}_{\delta r}(t)$,  we now turn to the definition of corresponding covariance operator $\mathcal{C}$. Define a linear operator $\mathcal{C}$ on the space of functions
$\varphi(\theta)$,  $\theta\in\Theta$ which is  given by the integral:
\begin{equation}{\label{cooperator}}
    \mathcal{C} \varphi  ( \theta ) := \int_0^{2 \pi}  {C}_{\delta r} ( | \theta - \theta' | ) \, \varphi ( \theta ' ) \, d \theta '.
\end{equation}
Then we have the following eigenvalue problem for the operator $\mathcal{C}$:
\begin{equation}\label{eigenvalue problem}
   \mathcal{C}\varphi_j(\theta)= \lambda_j \varphi_j(\theta), \qquad \varphi_j\in L^2([0,2\pi)),
\end{equation}
where $\lambda_j$ is  eigenvalue of $\varphi_j(\theta)$. A straightforward computation shows that (details see \autoref{relationn})
\begin{gather}
\lambda_0 = 2\pi \, \widetilde{a}_0 = \int_{0}^{2\pi} {C}_{\delta r}(t) \, dt,\quad j=0,\label{guanxi1}\\
\mathcal{C} \varphi_{j, c}=\lambda_j \varphi_{j, c}, \quad \mathcal{C} \varphi_{j, s}=\lambda_j \varphi_{j, s}, \quad \lambda_j=\pi \widetilde{a}_j=\int_0^{2 \pi} {C}_{\delta r}(t) \cos (j t) \mathrm{d} t, \quad j\geq 1. 
    \label{coe}
\end{gather}
In order to parameterize the Gaussian process with covariance function $ _{\delta r}$, we employ the well-known Karhunen--Lo\`eve (KL) expansion \cite{sullivan2015introduction,feng2018efficient} to represent this stationary process $\delta r$:
\begin{equation}\label{KLexpan}
    \delta r(\theta, \omega)=\sqrt{\lambda_0} \xi_0(\omega) \phi_0(\theta)+\sum_{j=1}^{\infty} \sqrt{\lambda_j}\Big(\xi_{j, s}(\omega) \phi_{j, s}(\theta)+\xi_{j, c}(\omega) \phi_{j, c}(\theta)\Big),
\end{equation}
where $\varphi_0(\theta)$,\ $\varphi_{j,s}(\theta)$ and $\varphi_{j,c}(\theta)$ are the orthogonal eigenfunctions of \eqref{eigenvalue problem} with the form:
\begin{align}
\varphi_0(\theta)&=\frac{1}{\sqrt{2 \pi}}, \quad  j = 0 \label{var0}   \\ 
\varphi_{j, c}(\theta)&=\frac{1}{\sqrt{\pi}} \cos (j \theta), \quad j \geq 1, \\ 
\varphi_{j, s}(\theta)&=\frac{1}{\sqrt{\pi}} \sin (j \theta), \quad j \geq 1 ,\label{varjs}
\end{align}
and $\xi_0(\omega)$, $ \xi_{j, c}(\omega)$,  $\xi_{j, s}(\omega)$ are i.i.d. Gaussian random variables with zero mean and unit covariance. Owing to the rapid decay of the KL coefficients (see \autoref{hinvv}) and to simplify numerical implementation, it is common practice to truncate the infinite expansion \eqref{KLexpan} to a finite series:
\begin{equation}\label{trunKLexpan}
     \delta r(\theta, \omega)=\sqrt{\lambda_0} \xi_0(\omega) \phi_0(\theta)+\sum_{j=1}^{N_{\mathrm{KL}}} \sqrt{\lambda_j}\Big(\xi_{j, s}(\omega) \phi_{j, s}(\theta)+\xi_{j, c}(\omega) \phi_{j, c}(\theta)\Big),
\end{equation}
where $N_{\mathrm{KL}}$ is the truncation number. 

\begin{remark}\label{remark2}
Although the modification in \eqref{pro}--\eqref{correc} enforces positive semi-definiteness of the covariance kernel and thus renders the associated Gaussian process mathematically well-defined, its numerical impact is negligible in practice: the corrected kernel ${C}_{\delta r}(t)$ differs from the target kernel $C_{\mathrm{geod}}(t)$ only by a few high-frequency Fourier coefficients, while the low- and medium-frequency modes that dominate the boundary fluctuations remain essentially unchanged. 
This is consistent with the fact that potential violations of positive semi-definiteness arise primarily in the high-frequency regime: heuristically, since $C_{\mathrm{geod}}(t)$ is nonnegative and concentrates most of its mass near $t=0$ (decaying rapidly as $t$ increases), the integrals in \eqref{coe} are dominated by positive contributions for small/moderate $j$, whereas for large $j$ the rapid oscillations of $\cos(jt)$ lead to strong cancellations and hence to small (and more sign-sensitive) eigenvalues, which may even become negative.

This high-frequency sensitivity is further mitigated by the rapid spectral decay of the covariance operator induced by the Gaussian-type kernel. 
Indeed, the fast decay of the KL eigenvalues (see \autoref{hinvv}) implies that, in practice, the KL expansion in \eqref{trunKLexpan} can be safely truncated after only a small number of modes. 
Since our modification affects only a few tiny high-frequency coefficients, the target kernel $C_{\mathrm{geod}}(t)$ and its projected counterpart ${C}_{\delta r}(t)$ therefore produce practically identical truncated KL representations within the parameter regimes used in \eqref{trunKLexpan} (see \autoref{fig:covariance} and \autoref{fig:covariance10000}).

\end{remark}

\begin{figure}[htbp]
    \centering
    \begin{subfigure}{0.33\textwidth}
        \includegraphics[width=\linewidth]{ 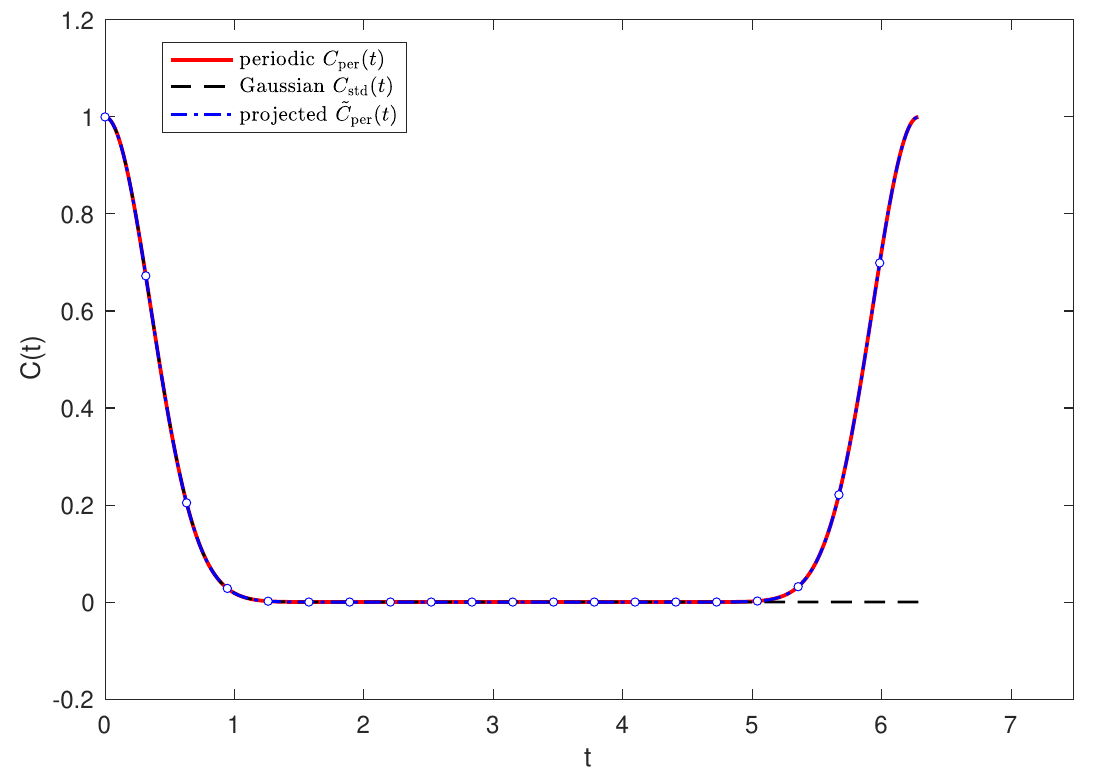}
        \caption{\small  covariance (50 terms)}
        \label{fig:covariance}
    \end{subfigure}
     \begin{subfigure}{0.33\textwidth}
        \includegraphics[width=\linewidth]{ 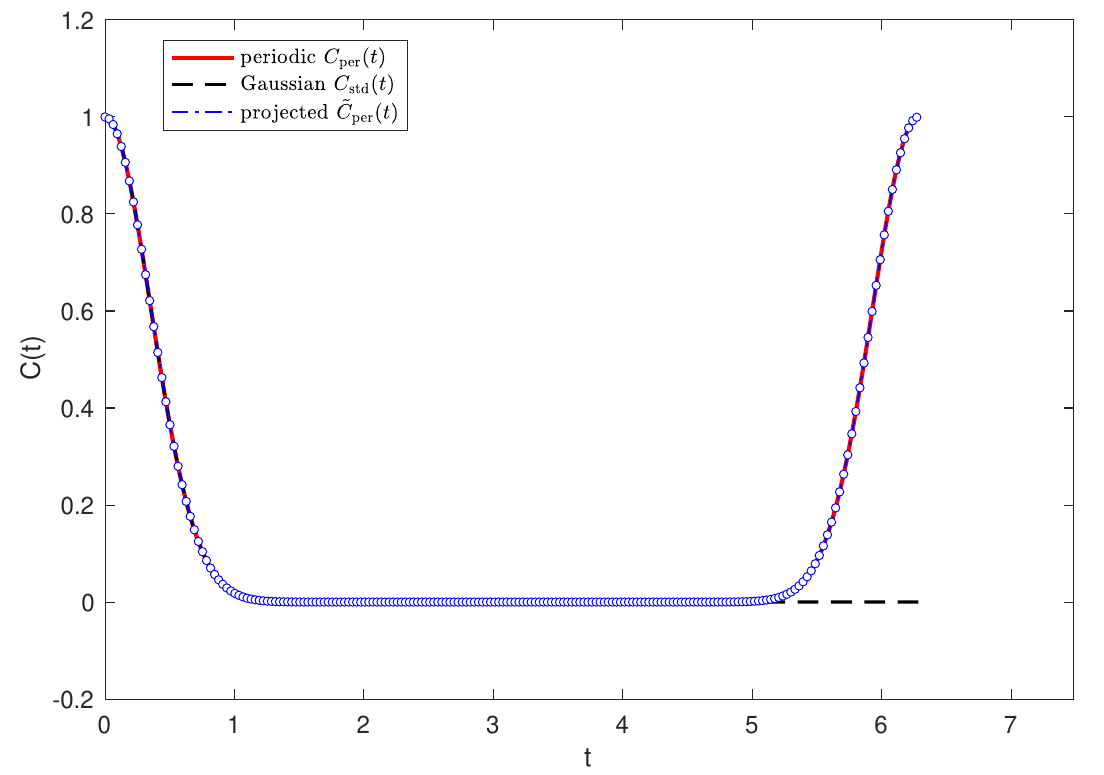}
        \caption{\small  covariance (10000 terms)}
        \label{fig:covariance10000}
    \end{subfigure}
     \begin{subfigure}{0.33\textwidth}
        \includegraphics[width=\linewidth]{ 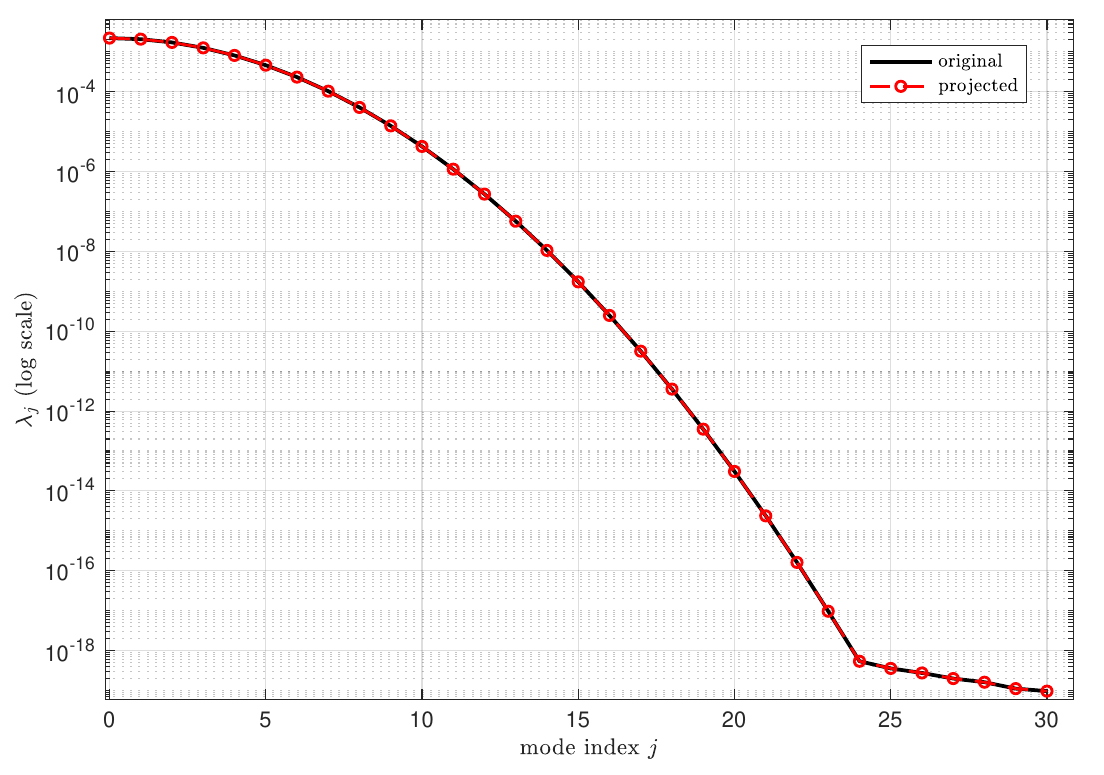}
        \caption{\small  $\lambda_j$ for operator $\mathcal{C}$} 
        \label{hinvv}
    \end{subfigure}
    \caption{\small \subref{fig:covariance}: 3 periodic covariance functions  with $\sigma=1$, $\ell = 0.5$, $N_{\theta}=400$ and 50 Fourier expansion coefficients $\{a_j\}$:  $C_{\mathrm{geod}}(t)$ (red solid line) defined in \eqref{eq:cov-1d};  $C_f(t)$ (black dashed line) defined in \eqref{standard}; $ {C}_{\delta r}(t)$ (blue dashed line with ``o" markers) defined in \eqref{correc}. \subref{fig:covariance10000}: 3 periodic covariance functions  with $\sigma=1$, $\ell = 0.5$, $N_{\theta}=40004$ and 10000 Fourier expansion coefficients $\{a_j\}$:  $C_{\mathrm{geod}}(t)$ (red solid line);  $C_f(t)$ (black dashed line); $ {C}_{\delta r}(t)$ (blue dashed line with ``o" markers).  \subref{hinvv}: decay of  first 30 $\lambda_j$ of the covariance operator $\mathcal{C}$ when $\sigma=0.05$, $\ell=0.5$ and $N_{\theta}=400$: $\lambda_j$   under $C_{\mathrm{geod}}(t)$ (black line); $\lambda_j$ under $ {C}_{\delta r}(t)$ (red line with ``o" markers).}
    \label{fig:KLcoe}
    \end{figure}




\begin{remark}
\label{remark1}
    According to the construction of the random scatterer in  \eqref{eq:star-shaped} and \eqref{eq:r-decomp}, our current  random modeling framework is temporarily confined to a specific subclass of star-shaped scatterers, that is, the boundary $\partial D$ must strictly satisfy the governing expression given in \eqref{orsur} which serves to ensure the mathematical consistency of the additive perturbation form in \eqref{eq:r-decomp}. Additionally, we also require the boundary admits a radial parameterization where the parameter $t$ equals the polar angle $\theta$ (or generally $t$ is linearly related to $\theta$). Typical examples include   \textbf{circles}:
\begin{equation}\label{circle}
    \partial D_{\mathrm{circle}}=\{x_0+R(cos(t),sin(t)),\qquad t \in [0, 2\pi)\}
    \end{equation}
     \textbf{pears}:
\begin{equation}\label{pear}
    \partial D_{\mathrm{pear}} = \left\{ x_0 + r_{\mathrm{pear}}(t) (\cos t, \sin t) : t \in [0, 2\pi) \right\},\quad r_{\mathrm{pear}}(t) = a + b \cos t + c \sin t,
    \end{equation}
    and \textbf{peanuts}:
    \begin{equation}
        \partial D_{\mathrm{peanut}} =\{ x_0+r_{\mathrm{peanut}}  (\cos t, \sin t) : t \in [0, 2\pi) \}, \quad r_{\mathrm{peanut}}(t)=\sqrt{\cos^2 t + 0.25 \sin^2 t}.
    \end{equation}
    etc. While some other star-shaped scatterers like \textbf{rounded square-shaped} obstacle:
\begin{equation}
    \partial D_{\mathrm{rounded\  square}} = \frac{3}{4} \left( \cos^3t + \cos t,\ \sin^3t  + \sin t \right), \quad t\in [0, 2\pi).
\end{equation}
    or \textbf{kites}:
   \begin{equation}
       \partial D_{\mathrm{kite}} = \left( \cos t+ 0.65(\cos 2t - 1),\ \sin t \right), \quad t \in [0, 2\pi).
   \end{equation}
    Their geometric representations in polar coordinates involve a nonlinear relationship between $t$ and $\theta$ which falls outside our present scope. Extending this inversion methodology to such nonlinear random scatterers represents an important direction for our future work.
    \end{remark}

\section{Reconstruction method}
\label{Reconstruction method}

In this section, we present a two-stage methodology for reconstructing the random scatterer's geometry.  The goal of Stage 1 is to aggregate these sample-wise reconstructions to obtain the mean shape, Stage 2 then estimates the statistical descriptors of the boundary fluctuations, namely the KL eigenvalues ${\lambda_j}$ in \eqref{trunKLexpan} and the covariance hyperparameters $(\sigma,\ell)$.


\subsection{The reconstruction of mean shape} \label{mean shape}
To recover the mean shape in the first stage, we adopt the reconstruction framework of \cite{sini2012inverse} which combines a low-frequency nonlinear least-squares step with a high-frequency recursive linearization approach (RLA). This hybrid strategy exploits complementary frequency regimes: the low-frequency data admit uniqueness guarantees (see, e.g., \cite{colton1983uniqueness}) while the high-frequency  refines fine geometric features and improves numerical stability. From a methodological viewpoint, Stage 1 only aims to realize  reasonably accurate sample-wise reconstructions to enable estimation of the mean geometry, thus our two-stage pipeline is not tied to the specific first-stage solver used here and  other shape reconstruction methods (e.g., continuation-based gradient methods \cite{li2016inverse}) can also be adopted. Once these reconstructions are available, the  baseline geometry for the scatterer $D$ is then obtained by   averaging  these sample-wise reconstructed  radial functions over the whole random realizations.

For a fixed $\omega$, we consider a star-shaped scatterer $D\subset\mathbb{R}^2$ whose boundary admits the radial parametrization \eqref{eq:star-shaped} with $\partial D\in C^3$. 
For each wavenumber $k$, we define the \emph{far-field operator} $\mathcal{F}(\cdot,k):\partial D \rightarrow C(\mathbb{S}^1)$ 
which assigns a radial function $r$ (equivalently, the boundary $\partial D$) to the corresponding far-field pattern $u^\infty$ of $u^s$ in  \eqref{eq:helmholtz}--\eqref{eq:src}. 
Its domain (Fr\'echet) derivative at $r$ in the  direction $h$ is defined by
\begin{equation}\label{Frechet}
\mathcal{F}'(r,k)(h):=\lim_{\varepsilon\to 0}\frac{\mathcal{F}(r+\varepsilon h,k)-\mathcal{F}(r,k)}{\varepsilon},
\end{equation}
where $h$ is a $2\pi$-periodic perturbation of the radial function. 
It is known that $\mathcal{F}'(r,k)(h)$ is a bounded injective linear operator and evaluating $\mathcal{F}'(r,k)(h)$ reduces to solving an auxiliary exterior Helmholtz problem with a specific boundary condition \cite{kirsch1993domain}. Let $u^{\infty,\delta}(\cdot,k)$ denote the measured far-field data corrupted by additive noise of level $\delta\ge 0$. 
Given $u^{\infty,\delta}(\cdot,k)$, the reconstruction at frequency $k$ is then formulated as a nonlinear least-squares problem
\begin{equation}\label{target}
G(r,k):=\frac{1}{2}\Big\|\mathcal{F}(r,k)-u^{\infty,\delta}(\cdot,k)\Big\|^2_{L^2(\mathbb{S}^1)}.
\end{equation}
A natural way to tackle \eqref{target} with multi-frequency data is through a \emph{frequency-continuation} strategy: we start from the lowest available wavenumber $k_l$, compute an initial reconstruction by minimizing the data-misfit functional \eqref{target} at $k_l$ and then use the obtained minimizer as the initial guess for the next higher wavenumber until reaching the maximum wavenumber $k_h$. In practice, however, objective $G(\cdot,k_l)$ are always highly nonconvex and expensive to minimize, so in order  
to mitigate these issues one often introduces Tikhonov regularization to \eqref{target}:
\begin{equation}\label{Tikhonov}
G_{\gamma}(r,k_l):=\frac{1}{2}\Big\|\mathcal{F}(r,k_l)-u^{\infty,\delta}(\cdot,k_l)\Big\|^2_{L^2(\mathbb{S}^1)} + \frac{1}{2}\gamma\|r\|^2_{\partial D},
\end{equation}
where $\gamma>0$ is a regularization parameter. 
At the low frequency $k_l$, the regularized objective $G_{\gamma}(r,k_l)$ in  \eqref{Tikhonov} is locally better behaved and local uniqueness results can be established under mild initialization conditions \cite{gintides2005local}.

However, the minimizer of \eqref{Tikhonov} typically captures only the coarse features of the boundary and is not expected to be highly accurate. We therefore refine this low-frequency estimate using higher-frequency data through the continuation procedure. At high-frequency we employ the Recursive Linearization Approach (RLA), originally developed in \cite{chen1997inverse} for medium scattering, we replace the nonlinear optimization at each new frequency in \eqref{target} by a locally linearized update. Specifically, RLA considers a sequence of wavenumbers $k_j = k_l + j\Delta k$, $j=0,1,\ldots,N$ with $\Delta k=(k_h-k_l)/N$. Starting from an initial reconstruction $r_0$ at $k_l$ (e.g., the solution of \eqref{Tikhonov}), the method advances from $k_j$ to $k_{j+1}$ by linearizing the far-field map $\mathcal{F}(\cdot,k_{j+1})$ around the current estimate $r_j$ through its first-order Taylor expansion:
\begin{equation}\label{Taylor}
    \mathcal{F}(r, k_{j+1}) = \mathcal{F}(r_j, k_{j+1}) +  \mathcal{F}'(r_j, k_{j+1})(r - r_j) + O(\|r - r_j\|_{\partial D}^2).
\end{equation}
Substituting this expansion into \eqref{target} linearizes the objective functional $G(r,k)$, the updated approximation is then sought as $r_{j+1}=r_j+\Delta r$ where the increment $\Delta r$ solves the following linear least-squares problem:
 \begin{equation}\label{RLAtarget}
     \Delta r_j := \arg\min_{\Delta r} \frac{1}{2} \left\| \mathcal{F}(r_j, k_{j+1}) + \mathcal{F}'(r_j, k_{j+1})( \Delta r)- u^{\infty,\delta}(\cdot,k_{j+1},r_j)\right\|_{L^2(\mathbb{S}^1)}^2.
 \end{equation}
To mitigate the inherent ill-posedness, we solve the regularized counterpart of \eqref{RLAtarget}which is defined by
\begin{equation}\label{RLATikhonov}
       \Delta r_j := \arg\min_{\Delta r} \bigg\{\frac{1}{2} \left\| \mathcal{F}(r_j, k_{j+1}) + \mathcal{F}'(r_j, k_{j+1})( \Delta r)- u^{\infty,\delta}(\cdot,k_{j+1},r_j)\right\|_{L^2(\mathbb{S}^1)}^2+\frac{1}{2}\alpha\|\Delta r\|^2_{\partial D}\bigg\}.
\end{equation}
Here $\alpha$ is the  regularization parameter. Applying the Levenberg-Marquardt method to solve the regularized problem \eqref{RLATikhonov} we obtain the following normal equation for $\Delta r_j$:
\begin{equation}\label{normal}
    (A_j^*A_j + \alpha I)\,\Delta r_j = - A_j^* g_j,
\end{equation}
with $ A_j := \mathcal{F}'(r_j,k_{j+1})$, $g_j := \mathcal{F}(r_j,k_{j+1})
           - u^{\infty,\delta}(\cdot,k_{j+1},r_j)$ and $A_j^*$ denoting the adjoint operator of $A_j$.
Then \eqref{normal} gives the  following increment $\Delta r_j$ as:
\begin{equation}\label{update}
     \Delta r_j
    = -(\alpha I + A_j^*A_j)^{-1}
      A_j^*g_j.
\end{equation}
Thus, the radius $r_{j+1}$ will be updated via
\begin{equation}\label{update}
    r_{j+1} = r_j +\Delta r_j=r_j- (\alpha I + A_j^* A_j)^{-1} A_j^* g_j.
\end{equation}
This recursive process continues until $k_j$ reaches  the highest wavenumber $k_h$. Note that the inversion process described above is conducted for a fixed $\omega$. For a set of random realizations $\{\omega_s\}_{s=1}^{N_s}$, we execute this process independently for each sample and record the corresponding results, then the empirical mean is obtained by taking the sample average over all realizations.

To transform the radial function into a finite set of unknown parameters, we expand $r$ into a Fourier series and truncate it at a finite order. Concretely, any radial function $r(\theta)$ satisfying $r(0)=r(2\pi)$ can be regarded as a $2\pi$-periodic function and thus admits the following Fourier expansion:
\begin{equation}\label{Foure}
    r(\theta) = a_0 + \sum_{m=1}^{\infty} \left( a_m \cos m\theta + b_m \sin m\theta \right),\qquad \theta\in\Theta,
\end{equation}
where $\{a_m,b_m\}$ are the Fourier coefficients.  In practical inversion, we retain only a finite number of modes $m=1,2,\cdots,N_r$, yielding the truncated approximation
\begin{equation}\label{eq:r-fourier-trunc}
r^{N_r}(\theta)
:= a_0+\sum_{m=1}^{N_r}\bigl(a_m\cos m\theta + b_m\sin m\theta\bigr),
\qquad \theta\in\Theta.
\end{equation}
For notational convenience,  by collecting all coefficients into a parameter vector 
\begin{equation}\label{fou}
    p = (p_1, \ldots, p_{n_r})^\top \in \mathbb{R}^{n_r}, \quad n_r = 2N_r + 1,
\end{equation}
with the convention  $p_1 = a_0$, $p_{2m} = a_m,$ $p_{2m+1} = b_m$, $ m = 1, \ldots, N_r$, then the truncated radius can then be written as  
\begin{equation} \label{eq:param-radius}
r(\theta, p) = p_1 + \sum_{m=1}^{N_r} \bigl( p_{2m} \cos m\theta + p_{2m+1} \sin m\theta \bigr), \qquad \theta \in \Theta.
\end{equation}
In this way, the inverse problem is transformed from finding the function \( r(\theta) \) into finding the parameter vector \( p \) in a finite-dimensional coefficient space.
The complete procedure is outlined in  Algorithm \ref{alg:mean-shape-random}. In Algorithm \ref{alg:mean-shape-random}, all updates, whether the low-frequency nonlinear iterations or the high-frequency recursive linearizations, act directly on \( p \) and the corresponding geometric shape is uniquely determined by \( p \) via equation \eqref{eq:param-radius}.

\begin{algorithm}[htbp]
\caption{First-stage RLA-based multi-frequency baseline shape reconstructing method}
\label{alg:mean-shape-random}
\begin{algorithmic}[1]
 \Require   noisy multi–frequency far–field data 
  $\{u^{\infty,\delta}(\hat x_m,d,k_j)\}_{m=1,\dots,M}^{j=1,\dots,n_K}$,  samples number $N_s$, incident direction $d$, observation directions $\{\hat x_m\}_{m=1}^M$, wavenumbers $\{k_j\}_{j=1}^{n_K}$ with $k_1<\dots<k_{n_K}$, shape parameterization $r(\theta,p)$,  forward map $\mathcal F(p,k)$ and its Fr\'echet derivative
  $\mathcal F'(p,k)$, regularization parameters $\gamma>0$, $\alpha>0$;
  tolerances \texttt{tol} and maximal inner iterations $N_{\mathrm{it}}$.

\State \textbf{Initialization.}
Choose an initial guess $p^{(0)}\in\mathbb R^{n_r}$ (e.g.\ choose a circle of radius $R_0$).

\Statex
\State \textbf{For each sample \(\omega_s\in\Omega\), reconstruct the deterministic geometry shape.}
\For{$s = 1,\dots,N_s$}
  \State Denote the far--field data of this sample by
  \[
    u^{\infty,\delta}_{j}(\hat x_m,\omega_s)
    := u^{\infty,\delta}(\hat x_m,d,k_j,\omega_s),
    \quad j=1,\dots,n_K,\; m=1,\dots,M.
  \]

  \State \textbf{Step 1: non-linear optimization at low frequency $k_1$.}
    \State  Start from $p^{(0)}$, minimize $G_\gamma(p,k_1)$ in \eqref{Tikhonov}
by any optimization method (e.g.\ Levenberg–Marquardt or
a gradient–based method) until $
  \|\nabla G_\gamma(p,k_1)\|
  \le \texttt{tol}
  \ \text{or}\ 
  \text{iteration} \ge N_{\mathrm{it}}.$
    \State Denote the resulting minimizer by $p_{1}^{(s)}$ and set
      $r_{1}(\theta,\omega_s):=r(\theta,p_{1}^{(s)})$.

  \Statex
  \State \textbf{Step 2: recursive linearization at higher frequencies.}
  \For{$j=1,\dots,n_K-1$}
    \State Set $k:=k_{j+1}$ and current parameter $p_{j}^{(s)}$.
    \State Compute the residual
    $ g_j:= \mathcal F(p_j,k)
        - u^{\infty,\delta}(\cdot,k).$
    \State Compute the update
      $\Delta p_j^{(s)}$ by solving the normal equation \eqref{normal}, update
      $p_{j+1}^{(s)} := p_{j}^{(s)} + \Delta p_j^{(s)}$ by \eqref{update} and set
      $r_{j+1}(\theta,\omega_s):= r(\theta,p_{j+1}^{(s)})$.
  \EndFor

  \State \textbf{Store final reconstruction for sample $\omega_s$.}
  \State Set $p^{\mathrm{rec}}(\omega_s):=p_{n_K}^{(s)}$ and
         $r^{\mathrm{rec}}(\theta,\omega_s)
           := r(\theta,p^{\mathrm{rec}}(\omega_s))$.
\EndFor

\Statex
\State \textbf{Sample mean shape.}
\State Define the reconstructed empirical mean radius by
\[
  \bar r^{\mathrm{rec}}(\theta)
  := \frac{1}{N_s}\sum_{s=1}^{N_s}
     r^{\mathrm{rec}}(\theta,\omega_s),
  \qquad \theta\in[0,2\pi).
\]

\State \textbf{return}
$\bar r^{\mathrm{rec}}(\theta)$ and
$\{r^{\mathrm{rec}}(\theta,\omega_s)\}_{s=1}^{N_s}$.
\end{algorithmic}
\end{algorithm}


\subsection{The reconstruction of random statistics}
\label{random statistics}

Following the inversion of  mean shape $\bar r^{\mathrm{rec}}(\theta)$ from the first stage, we can now advance to the second stage:  estimating the stochastic parameters based on the reconstructed results from  \autoref{mean shape}. To begin with,  we focus on the reconstruction of the coefficients $\{\lambda_j\}_{j=0}^{N_{KL}}$ in the KL expansion defined in \eqref{trunKLexpan}. According to the definition of $\lambda_j$ in \eqref{eigenvalue problem}, reconstructing these eigenvalues hinges on obtaining a representation of the integral operator $\mathcal{C}$ in \eqref{cooperator}. We begin this inversion by discretizing operator $\mathcal{C}$ on a chosen grid of suitable observation angles $\{\theta_l\}_{l=1}^{N_{\theta}}$, which is constructed by evaluating  the reconstructed shapes at these angles $\{\theta_l\}$ to form the reconstructed empirical covariance matrix. The eigenvalues are then retrieved by computing the singular value decomposition (SVD) of this matrix. Since Algorithm \ref{alg:mean-shape-random} is executed independently for each sample $\{\omega_s\}_{s=1}^{N_s}$, upon completion we obtain the inverted shape parameters $r^{\mathrm{rec}}(\theta,\omega_s)$ (or $p^{\mathrm{rec}}(\omega_s)$) at the highest wavenumber $k_{n_{K}}$ for each sample and their empirical mean reconstructing geometry $\bar r^{\mathrm{rec}}(\theta)$. Define the \emph{reconstructed stochastic fluctuation} $ \delta r^{\mathrm{rec}}(\theta, \omega_s)$ as
\begin{equation}\label{refluc}
    \delta r^{\mathrm{rec}}(\theta, \omega_s) := r^{\mathrm{rec}}(\theta, \omega_s) - \bar{r}^{\mathrm{rec}}(\theta).
\end{equation}
Given a set of discrete angles $\{\theta_l\}_{l=1}^{N_{\theta}}$, $ \delta r^{\mathrm{rec}}(\theta, \omega_s)$ is naturally cast into a matrix form:
\begin{equation}\label{reflucmatrix}
  R_{\mathrm{rec}} = \left[ \delta r^{\mathrm{rec}}(\theta_\ell, \omega_s) \right]_{\ell,s} \in \mathbb{R}^{N_\theta \times N_{s}}.
\end{equation}
Then the empirical covariance matrix is estimated  by
\begin{equation}\label{ecm}
   C_{\mathrm{rec}} := \frac{1}{N_{s} - 1} \, R_{\mathrm{rec}} \, R_{\mathrm{rec}}^\top\in\mathbb{R}^{N_\theta\times N_\theta}.
\end{equation}
The factor $1/(N_s-1)$ is the usual unbiased normalization so that $\mathbb{E}[C_{\mathrm{rec}}]$ coincides with the covariance matrix of the discretized fluctuations. We then solve the eigenvalue problem $C_{\mathrm{rec}} f_j=\mu_j f_j$ by eigendecomposition. In order to relate $\mu_j$ to the eigenvalues $\lambda_j$ of the continuous covariance operator $\mathcal{C}$, the eigenvalues $\mu_j$ of the discrete covariance matrix $C_{\mathrm{rec}}$ should be multiplied  by  the quadrature weight  $w=2\pi/N_\theta$.  This quadrature weight $w$ arises from the Nystr\"om discretization on a uniform grid which replaces the integral over $[0,2\pi]$ by a quadrature sum with weight $w$, hence we define the reconstructed discrete eigenvalues by
$ \lambda_j^{\mathrm{rec}} := w\,\mu_j$ 
and then proceed to utilize  $\{\lambda_j^{\mathrm{rec}}\}$ throughout the subsequent inversion procedure.



To fully characterize the stochastic Gaussian process, the final step of the inversion process involves estimating the covariance kernel parameters $\ell$ and $\sigma$. As established by the Fourier expansion of the covariance function $ C_{\delta_r}(t)$ in \eqref{Fourier}-\eqref{coeff} and the quantitative relationship \eqref{guanxi1}-\eqref{coe} between its coefficients $a_j$ and the eigenvalues $\lambda_j$, the complete information about the statistical parameters $(\ell,\sigma)$ is encoded within the eigenvalue spectrum $\{\lambda_j\}$. Therefore, we can infer these two parameters by utilizing the reconstructed eigenvalues $\{\lambda^{\mathrm{rec}}_j\}$. With the covariance function $C_f(t)$ defined in \eqref{standard} in stochastic periodic gratings in \cite{bao2020inverse}, an approximate relation between the eigenvalues $\lambda_j^{\mathrm{mod}}$ and the parameters $\sigma,\ell$ can be derived as follows:
\begin{equation}\label{relation}
\lambda_j^{\mathrm{mod}} \approx \sqrt{\pi} \sigma^2 \ell e^{-j^2 l^2 / 4}, \quad j = 0, 1, \dots
\end{equation}
Precisely, 
\begin{equation}\label{exactrelatio}
    \lambda_j^{\mathrm{mod}}
  = \sqrt{\pi}\,\sigma^2\ell\,e^{-\ell^2 j^2/4}
    + \mathcal{O}\!\bigl(\sigma^2\ell^2 e^{-\pi^2/\ell^2}\bigr).
\end{equation}
The detailed derivation of \eqref{exactrelatio} is shown in \autoref{appen:relation}. It is noted that this approximating relationship  \eqref{relation} or \eqref{exactrelatio} also holds for the covariance function $C_{\delta_r}$ in our setting, as ensured by the following Proposition:
\begin{Proposition}\label{relatio}
    The estimated  relation \eqref{exactrelatio} still holds when covariance function is $C_{\delta_r}(t)$.
\end{Proposition}
\begin{proof}
    See \autoref{appen:Justification}.
\end{proof}
\noindent Thus, in the following process we will invert $(\sigma, \ell)$  based on  the relation \eqref{relation} with the reconstructed set $\{\lambda^{\mathrm{rec}}_j\}$. 

Due to the ill-posedness of the  two-dimensional inverse scattering problem, unlike the one-dimensional case where only the first two modes $\{\lambda_0^{\mathrm{mod}},\lambda_1^{\mathrm{mod}}\}$ are considered, we utilize the first $N_{\mathrm{KL}}$ eigenvalues jointly  to invert for the parameters $\sigma$ and $\ell$ where $N_{\mathrm{KL}}$ is a prescribed selected number. By Proposition \autoref{relatio}, relation \eqref{relation} still holds in our problem settings. Taking the logarithm on both sides of \eqref{relation} yields
\begin{equation}\label{loggre}
    \log \lambda_j^{\mathrm{mod}} \approx \log(\sqrt{\pi} \sigma^2 \ell) - \frac{\ell^2}{4} j^2,
\end{equation}
With a slight abuse of notation, we can regard \eqref{loggre} as a linear model of the form $y=A-Bx$  where $y_j:= \log \lambda_j^{\mathrm{mod}}$ and $x_j:=j^2$. The coefficients are given by the intercept $A=\log(\sqrt{\pi} \sigma^2 \ell)$ and the slope $B= \frac{\ell^2}{4}$. To estimate these coefficients  $A$ and $B$ from  the available data points $(x_j,y_j)$, we employ a \emph{linear least-squares approach} which involves minimizing the sum of squared residuals defined by the objective function: 
\begin{equation}\label{fitting}
\mathcal{J}(A, B) := \sum_{j =0}^{N_{\mathrm{KL}}} \left( y_j - (A - B x_j) \right)^2 = \sum_{j =0}^{N_{\mathrm{KL}}} \left( \log \lambda_j^{\mathrm{mod}} - A + B j^2 \right)^2,
\end{equation}
which is to determine
\begin{equation}\label{AB}
    (A_{\mathrm{fit}},B_{\mathrm{fit}}) := \arg\min_{A,B} \mathcal{J}(A, B).
\end{equation}
 Once we get $A_{\mathrm{fit}}$ and $B_{\mathrm{fit}}$, the correlation length $\ell$ is first recovered from the slope $B_{\mathrm{fit}}$:
\begin{equation}\label{ell}
  B_{\mathrm{fit}} = \frac{\ell^2}{4} \quad \Longrightarrow \quad \ell_{\mathrm{est}} = \sqrt{\max\{4 B_{\mathrm{fit}},0\}},
\end{equation}
and the strength parameter $\sigma$ is determined from the intercept $A_{\mathrm{fit}}$ using the relation
\begin{equation}\label{sigma}
    A_{\mathrm{fit}} = \log(\sqrt{\pi} \, \sigma^2 \ell_{\mathrm{est}}) \quad \Longrightarrow \quad \sigma_{\mathrm{est}}^2 = \frac{e^{A_{\mathrm{fit}}}}{\sqrt{\pi} \, \ell_{\mathrm{est}}}, \qquad
\sigma_{\mathrm{est}} = \sqrt{ \frac{e^{A_{\mathrm{fit}}}}{\sqrt{\pi} \, \ell_{\mathrm{est}}} }.
\end{equation}
This completes the specification of the second-stage inversion, with Algorithm \ref{alg:random-stats} encapsulating the entire workflow from the eigenvalue reconstruction to the estimation of $(\sigma_{\mathrm{est}},\ell_{\mathrm{est}})$. These results set the stage for the theoretical analysis in the next section and for the numerical experiments reported thereafter.

\begin{algorithm}[htbp]
  \caption{Second-stage  random statistics reconstructing method}
  \label{alg:random-stats}
  \begin{algorithmic}[1]
    \Require
      Number of samples $N_s$; angular grid $\{\theta_i\}_{i=1}^{N_\theta}$;
      RLA reconstructions $\{p^{(s)}\}_{s=1}^{N_s}$ at the highest wavenumber $k_{\max}$ (Stage~1 output); $N_{\mathrm{KL}}$ for the number of KL modes to be used in the fitting process.

    \State \textbf{Step 1: define the discrete random fluctuation.}
    \For{$s = 1,\dots,N_s$}
       \State Compute the reconstructed radius
              $r_{\mathrm{rec}}^{(s)}(\theta_i)
                := r(\theta_i,p^{(s)})$ for all $i=1,\dots,N_\theta$
              using the Fourier parameterization~\eqref{eq:param-radius}.
       \State Store $\mathbf r^{(s)} = \bigl(r_{\mathrm{rec}}^{(s)}(\theta_1),\dots,
                           r_{\mathrm{rec}}^{(s)}(\theta_{N_\theta})\bigr)^\top$.
    \EndFor
    \State Compute the empirical mean shape
           $\displaystyle \bar{\mathbf r}_{\mathrm{rec}}
              = \frac{1}{N_s}\sum_{s=1}^{N_s}\mathbf r^{(s)}$.
    \State Form the fluctuation matrix
           $R_{\mathrm{rec}} = [\mathbf r^{(1)}-\bar{\mathbf r}_{\mathrm{rec}},\dots,
                                \mathbf r^{(N_s)}-\bar{\mathbf r}_{\mathrm{rec}}]
              \in\mathbb{R}^{N_\theta\times N_s}$ as defined in \eqref{refluc}.

    \Statex
    \State \textbf{Step 2: form empirical covariance and invert KL coefficients.}
    \State Define the empirical covariance matrix \eqref{ecm}.
          
    \State Utilize SVD to compute eigenvalues $\{\mu_j\}$ and obtain $\{\lambda^{\mathrm{rec}}_j\}$.
    \State Sort the eigenvalues in descending order and reorder the eigenvectors accordingly:
           $\lambda^{\mathrm{rec}}_1 \ge \lambda^{\mathrm{rec}}_2 \ge \cdots$.
    \State (Optional) Group eigenvalues corresponding to the same angular frequency
           (e.g.\ $\cos(j\theta)$ and $\sin(j\theta)$) by averaging the
           pairs $\{\lambda^{\mathrm{rec}}_{2j},\lambda^{\mathrm{rec}}_{2j+1}\}$.

    \Statex
    \State \textbf{Step 3: linear least–squares fitting of $(\ell,\sigma)$ from KL eigenvalues.}
    \State Select the first $N_{\mathrm{KL}}$ grouped empirical eigenvalues
           $\{\lambda^{\mathrm{rec}}_j\}_{j=0}^{N_{\mathrm{KL}}}$.
    \State Define the objective function \eqref{fitting}.
    \State Find $ (A_{\mathrm{fit}},B_{\mathrm{fit}})$ by solving \eqref{AB} and estimate \eqref{ell} and \eqref{sigma} with solution $ (A_{\mathrm{fit}},B_{\mathrm{fit}})$.
    
    \State \textbf{return} reconstructed eigenvalues
           $\{\lambda^{\mathrm{rec}}_j\}$ and
           parameter estimates $(\sigma_{\mathrm{est}},\ell_{\mathrm{est}})$.
  \end{algorithmic}
\end{algorithm}

\section{Theoretical analysis}\label{Theoretical analysis}

In this section, we present a theoretical analysis of the proposed random model and the associated two-stage reconstruction algorithm. We start with a brief discussion on uniqueness, clarifying how  deterministic theory supports the  mean-shape recovery in Stage 1 and how  statistical parameters are (locally) identifiable in Stage 2. Subsequently we focus on the convergence analysis:  we first derive error bounds for the second-stage recovery of the covariance spectrum and the associated hyperparameter estimates, then for context   we  review the RLA convergence result of \cite{sini2012inverse} for the first-stage reconstruction and in the end we  discuss several remaining challenges and open issues in establishing an 	overall coupled convergence theory for our two-stage pipeline. 

\subsection{Uniqueness}
Concerning uniqueness of the mean geometry, we note that the first-stage reconstruction is carried out sample-wise: for each realization $\omega$ it reduces to a deterministic inverse obstacle scattering problem. Uniqueness in  deterministic setting has been extensively studied:  for sound-soft obstacles, multi-frequency data for a single incident direction uniquely determine the scatterer both for a continuum of wavenumbers \cite{ramm1999multidimensional} and for finitely many wavenumbers \cite{colton1983uniqueness}; at a fixed wavenumber, single-frequency uniqueness  also holds for important geometric classes such as polygonal \cite{cheng2004global} and polyhedral \cite{alessandrini2005determining} scatterers. These results cover the adopted measurements in this paper and provide a deterministic uniqueness backdrop for the mean-shape reconstruction in the first stage.


For the uniqueness of the second-stage KL eigenvalues and the covariance parameters, under the random obstacle model introduced  in \autoref{ranmo}  the Gaussian random field is uniquely determined in law once its mean function \eqref{condi} and covariance function $C_{\delta r}$ in \eqref{correc} are specified. These specifications in turn uniquely determine the covariance operator  $\mathcal{C}$ in \eqref{cooperator} whose spectrum is thus uniquely determined up to ordering (counting multiplicities). The  covariance hyperparameters $(\sigma,\ell)$ are recovered from the estimated spectral quantities through a deterministic spectrum-to-parameter map and thus are uniquely determined whenever this map is locally  invertible on its range. These required nondegeneracy conditions are precisely those assumptions that will be introduced in the subsequent convergence analysis. 


\subsection{Convergence of eigenvalues}
 Next we turn to the  convergence analysis of the proposed two-stage inversion algorithm.  To begin with, we first provide a convergence estimate for the second-stage reconstruction of the KL eigenvalues. 
 Let $N_\theta$ be the number of discrete grid points $\{\theta\}_{i=1}^{N_{\theta}}$ on the scatterer and $N$ be the number of available samples. For each sample $s=1,\dots,N$, denote  
     the \emph{true fluctuation (discretised on the grid) } by $x^{(s)} \in \mathbb{R}^{N_{\theta}}$ and  the corresponding \emph{fluctuation obtained from the first stage reconstruction} by $\widehat x^{(s)} \in \mathbb{R}^{N_{\theta}}$. Define the reconstruction error $ \widetilde e^{(s)} := \widehat x^{(s)} - x^{(s)} \in \mathbb{R}^{N_{\theta}}$ and two empirical covariance matrices:
    \begin{equation}\label{twoco}
        \underbrace{\widehat{\Sigma}_N:=\frac{1}{N-1} \sum_{s=1}^N x^{(s)} x^{(s)\top}}_{\mathrm{empirical\ \textbf{true} \ covariance \ matrix}}, \quad \underbrace{\widehat{\Sigma}_N^{\mathrm{rec}}:=\frac{1}{N-1} \sum_{s=1}^N \widehat{x}^{(s)} \widehat{x}^{(s) \top}}_{\mathrm{empirical \ \textbf{reconstructed} \  covariance \ matrix}},
    \end{equation}
    let $\mu_j(\widehat{\Sigma}_N):=\widehat{\mu}_j$, $\mu_j(\widehat{\Sigma}_N^{\mathrm{rec}}):=\widehat{\mu}_j^{\mathrm{rec}}$ are their eigenvalues, respectively.  The result below indicates that the high-performance inversion of the eigenvalues for the empirical matrices in \eqref{twoco} relies on a sufficiently accurate geometric reconstruction from  Stage 1:
\begin{theorem}
\label{convergence}
     Assume there exist constants $M_x>0 $ and $  \varepsilon>0$ such that 
    \begin{equation}\label{assumption1}
        \frac1N\sum_{s=1}^N\|x^{(s)}\|_2^2 \le M_x,\qquad
\frac1N\sum_{s=1}^N\|\widetilde e^{(s)}\|_2^2 \le \varepsilon^2,
    \end{equation}
   then the difference of sample empirical covariance matrices in \eqref{twoco} admit the following error bound:
   \begin{equation}\label{coexti}
       \Bigl\|\widehat{\Sigma}_N^{\mathrm{rec}}-\widehat{\Sigma}_N\Bigr\|_2
\;\le\;\frac{N}{N-1}\Bigl(2\sqrt{M_x}\,\varepsilon+\varepsilon^2\Bigr),
   \end{equation} consequently for all $j$, their eigenvalues  satisfy:
   \begin{equation}\label{coeest}
       \Bigl|\widehat{\mu}_j^{\mathrm{rec}}-\widehat{\mu}_j\Bigr|
\;\le\;\frac{N}{N-1}\Bigl(2\sqrt{M_x}\,\varepsilon+\varepsilon^2\Bigr).
   \end{equation}
\end{theorem}
\begin{proof}
    See \autoref{appen:proofconvergence}.
\end{proof}
\noindent  For later use we denote  the total error bound $\frac{N}{N-1}\Bigl(2\sqrt{M_x}\,\varepsilon+\varepsilon^2\Bigr)$ on the right hand side of \eqref{coeest}  as 
   \begin{equation}
       \frac{N}{N-1}\Bigl(2\sqrt{M_x}\,\varepsilon+\varepsilon^2\Bigr):=E^{(\mu)}_{\mathrm{eig}} (N,\varepsilon,M_x),
   \end{equation}
 $E^{(\mu)}_{\mathrm{eig}}(N,\varepsilon,M_x)$ represents the error $E^{(\mu)}_{\mathrm{eig}}$ related to $N$,  $\varepsilon$ and $M_x$. \autoref{convergence} states that the accuracy of the  KL eigenvalues at Stage 2 is automatically ensured  provided the first-stage shape reconstruction meets the prescribed constraints with sufficient fidelity: the first constraint in \eqref{assumption1} specifies that each scatterer is a bounded obstacle while the second  constraint requires that the geometry reconstructed in the first stage is in close proximity to the true shape. 

 To further quantify the discrepancy between the discrete spectrum and its continuous counterpart, we now incorporate the   uniform quadrature weight $w:=2\pi/N_\theta$ and define
    \begin{equation}\label{lamb}
\lambda_j(\widehat{\Sigma}_N):=\widehat{\lambda}_j=w\,\widehat{\mu}_j,\qquad \lambda_j(\widehat{\Sigma}_N^{\mathrm{rec}}):=\widehat{\lambda}_j^{\mathrm{rec}}=w\,\widehat{\mu}_j^{\mathrm{rec}},
    \end{equation}
 through multiplication of $\widehat{\mu}_j$ and $\widehat{\mu}_j^{\mathrm{rec}}$ by $w$.  Combining \eqref{coeest} and \eqref{lamb}, naturally we get the following error estimate:
\begin{equation}\label{bound1}
    \left|\widehat{\lambda}_j^{\mathrm{rec}}-\widehat{\lambda}_j\right|\leq wE_{\mathrm{eig}}^{(\mu)}(N,\varepsilon,M_x)=:E_{\mathrm{eig}}(N,\varepsilon,M_x).
\end{equation} 
As  stated in \eqref{exactrelatio}, we denote  
\begin{equation}\label{lammod}
    \lambda_{j}^{\mathrm{mod}}=\sqrt{\pi}\,{\sigma}^2\ell\,e^{-{\ell}^2 j^2/4},\qquad \lambda_{j}^{\mathrm{mod}}(\sigma^\ast,\ell^\ast)=\sqrt{\pi}\,{\sigma^\ast}^2\ell^\ast\,e^{-{\ell^\ast}^2 j^2/4},
\end{equation}
 where $(\sigma^\ast,\ell^\ast)$ represents the true covariance parameters. Note that the definitions of the empirical covariance $\widehat{\Sigma}_N$, $\widehat{\Sigma}_N^{\mathrm{rec}}$ in \eqref{twoco} are both dependent on the sample size $N$ (the same holds for $\widehat{\lambda}_j$, $\widehat{\lambda}_j^{\mathrm{rec}}$). For the true discrete covariance Gram matrix $K=\{K_{mn}\}\in\mathbb{R}^{N_\theta\times N_\theta}$ appearing in \autoref{allpo} which is independent of the sample size $N$,  its eigenvalues  are given by the DFT of its first column since $K$ is circulant, we then define its corresponding eigenvalues $\{\mu_j^{\mathrm{samp}}\}$ and $\{\lambda^{\mathrm{samp}}_j\}$ as  
\begin{equation}\label{lamsam}
    \mu^{\mathrm{samp}}_j=\sum_{m=0}^{N_\theta-1} C(t_m)\,e^{-2\pi i jm/N_\theta},\quad
\lambda^{\mathrm{samp}}_j:=w\,\mu^{\mathrm{samp}}_j,\ \ w=\frac{2\pi}{N_\theta},\ \ t_m=\frac{2\pi m}{N_\theta},\ \forall j=0,\ldots,N_\theta-1.
\end{equation}

With these notations, we now aim to show the error estimates between $\widehat{\lambda}_j$ in \eqref{twoco} and $\lambda_j^{\mathrm{mod}}(\sigma^*,\ell^*)$ in \eqref{lammod}, which will play an important role in our subsequent analysis. For the fixed number of the selected modes $N_{\mathrm{KL}}$, we claim that there exists a finite modeling error $E_{\mathrm{mod}}>0$ such that
\begin{equation}\label{bound2}
    |\widehat{\lambda}_j-\lambda_j^{\mathrm{mod}}(\sigma^*,\ell^*)|\leq E_{\mathrm{mod}},\quad j=0,\ldots,N_{\mathrm{KL}}.
\end{equation}
\noindent In fact, the total error $E_{\mathrm{mod}}$ consists of three parts: 
 \begin{equation}\label{Emod}
\begin{split}
 \Big|\widehat{\lambda}_j-\lambda_j^{\mathrm{mod}}(\sigma^\ast,\ell^\ast)\Big|&\leq\Big|\lambda_j^{\mathrm{samp}}-\lambda_j\Big|+\Big|\lambda_j-\lambda_j^{\mathrm{mod}}(\sigma^\ast,\ell^\ast)\Big|+\Big|\widehat{\lambda}_j-\lambda_j^{\mathrm{samp}}\Big|\\
    &\leq E_{\mathrm{quad}} + E_{\mathrm{asym}} + E_{\mathrm{samp}} \\
    &:= \max_{0\le j\le N_{\mathrm{KL}}} E_{\mathrm{quad},j} + E_{\mathrm{asym}} +  E_{\mathrm{samp}}\\&:= E_{\mathrm{mod}}(N,N_{\mathrm{KL}},N_{\theta},\sigma^\ast,\ell^\ast), 
\end{split}
\end{equation}
 where
 \begin{itemize}
     \item $E_{\mathrm{quad}}$: maximum quadrature error across retained modes $E_{\mathrm{quad},j}$ where $E_{\mathrm{quad},j}$ denotes quadrature discretization error between the $j$-th continuous theoretical $\lambda_j$ from covariance operator $\mathcal C$ (defined via the integrals in \eqref{guanxi1}–\eqref{coe}) and its uniform-grid discretization $\lambda_j^{\mathrm{samp}}$ in \eqref{lamsam};
     
     \item  $E_{\mathrm{asym}}$: the asymptotic approximation error  incurred when replacing the exact periodic eigenvalues $\lambda_j$ by the asymptotic Gaussian formula $\lambda_j^{\mathrm{mod}}$ in \eqref{exactrelatio};
     
     \item $E_{\mathrm{samp}}$:  finite-sample  error caused by estimating the covariance from only $N$ realizations, i.e.,   $E_{\mathrm{samp}}$  measures the discrepancy between the  empirical covariance  eigenvalue $\widehat{\lambda}_j$ in \eqref{lamb} and its true (discrete) counterpart $\lambda_j^{\mathrm{samp}}$.
 \end{itemize} 
\noindent For any fixed configuration of 
$N,N_{\mathrm{KL}},N_{\theta}$ and fixed true parameters $(\sigma^\ast,\ell^\ast)$,  all  error terms $E_{\mathrm{quad}}$,  $E_{\mathrm{asym}}$ and $E_{\mathrm{samp}}$ in \eqref{Emod} are finite and admit explicit formulations: specifically,  $|\lambda_j^{\mathrm{samp}}-\lambda_j|$ can be estimated by 
\begin{equation}\label{quad}
      |\lambda_j^{\mathrm{samp}}-\lambda_j|
      =\left|\sum_{m=0}^{N_\theta-1}wC(t_m)e^{ijt_m}-\int_{-\pi}^{\pi}C(t)e^{ijt}dt\right|, 
\end{equation}
with $\lambda_j=\int_{-\pi}^{\pi}C(t)e^{ijt}dt$ denoting the theoretical eigenvalue of the covariance operator $\mathcal C$ \big(for detailed  derivation of  $\lambda_j$, see \eqref{eq:lambda-fourier}\big) and $\lambda_j^{\mathrm{samp}}=\sum_{m=0}^{N_\theta-1}wC(t_m)e^{ijt_m}$  derived with  the grid points set $\{t_m\ |\  t_m=-\pi+mw,m=0,1,\cdots N_{\theta}-1\}$. Using the composite trapezoidal rule, for each fixed mode $j$ we set $f_j(t):=C(t)e^{ijt}$. Since $f_j\in C^2([-\pi,\pi])$, the quadrature error  in \eqref{quad} gives a finite upper bound:
\begin{equation}\label{qu}
\begin{split}
       \left|\sum_{m=0}^{N_\theta-1}wC(t_m)e^{ijt_m}-\int_{-\pi}^\pi f_j(t)dt\right|&\leq\sum_{m=0}^{N_\theta-1}\frac{w^3}{12}\max_{t\in[t_m,t_{m+1}]}|f_j^{\prime\prime}(t)|\\&\leq\frac{N_\theta w^3}{12}\max_{t\in[-\pi,\pi]}|f_j^{\prime\prime}(t)|
      \\&\leq\frac{2\pi}{12}.w^2\max_{t\in[-\pi,\pi]}\left|\partial_t^2(C(t)e^{ijt})\right| :=E_{\mathrm{quad},j}.
\end{split}
\end{equation}
Hence  taking the maximum over $j=0,\dots,N_{\rm KL}$ of $E_{\mathrm{quad},j}$ yields the  bounded error $E_{\mathrm{quad}}$ in \eqref{quad}:
\begin{equation}
  E_{\mathrm{quad}}:=\max_{0\leq j\leq N_{\mathrm{KL}}}E_{\mathrm{quad},j}.
\end{equation}
As for  $E_{\mathrm{asym}}$, simple calculations together with Proposition \autoref{relatio} shows that
\begin{equation}\label{asy1}
    \big|\lambda_j-\lambda_j^{\mathrm{mod}}\big|=\sigma^2\left| \int_{-\pi}^{\pi} e^{-t^2/\ell^2} e^{ijt}\,dt-\int_{-\infty}^{\infty} e^{-t^2/\ell^2} e^{ijt}\,dt\right|\leq C_{\mathrm{asym}}\,\sigma^2\ell^2 e^{-\pi^2/\ell^2},
\end{equation}
for some finite constant $C_{\mathrm{asym}}>0$ independent of $j$ (for detailed  derivation, see \autoref{appen:relation}).
Therefore, 
\begin{equation}\label{asy2}
    \left|\lambda_j-\lambda_j^{\mathrm{mod}}(\sigma^\ast,\ell^\ast)\right|\leq C_{\mathrm{asym}}\sigma^{*2}\ell^{*2}\exp\left(-\pi^2/\ell^{*2}\right):=E_{\mathrm{asym}}.
\end{equation}
The finite bound for $|\widehat{\lambda}_j-\lambda_j^{\mathrm{samp}}|$ follows from the covariance estimation result  in \cite{koltchinskii2017concentration}: for any $\epsilon\in(0,1)$ with probability at least $1-\epsilon$, there exists a  finite constant $C_{\mathrm{samp}}$ such that  for empirical matrix $\widehat{\Sigma}_N$ and the  matrix $K$, we have
\begin{equation}\label{prob}
  \|\widehat{\Sigma}_N-K\|_2 \le C_{\mathrm{samp}}\,\|K\|_2\,
\max\left\{
\sqrt{\frac{r(K)}{N}},\;
\frac{r(K)}{N},\;
\sqrt{\frac{\log(1/\epsilon)}{N}},\;
\frac{\log(1/\epsilon)}{N}
\right\}:=\mathcal{E}_{\mathrm{samp}}(N,\epsilon),
\end{equation}
where $r(K) := \frac{\mathrm{tr}(K)}{\|K\|_2}$. Then with Weyl’s inequality \cite{franklin2000matrix} the finite error estimate of $|\widehat{\lambda}_j-\lambda_j^{\mathrm{samp}}|$ is obtained:
\begin{equation}\label{Esamp}
    |\widehat{\lambda}_j-\lambda_j^{\mathrm{samp}}|\leq w\|\widehat{\Sigma}_N-K\|_2\leq w\mathcal{E}_{\mathrm{samp}}(N,\epsilon):=E_{\mathrm{samp}}(N,\epsilon).
\end{equation}
Thus \eqref{bound2} holds with finite error bound $E_{\mathrm{mod}}$.  

Combining \eqref{bound1} and \eqref{bound2}, we derive the total finite error bound $E_\lambda$ for  KL eigenvalues:
\begin{equation}\label{totalerror}
   \eta_j:=  \left|\widehat{\lambda}_j^{\mathrm{rec}}-\lambda_j^{\mathrm{mod}}(\sigma^\ast,\ell^\ast)\right|\leq E_{\mathrm{eig}}+E_{\mathrm{mod}}:=E_\lambda,\quad \forall \ j=0,\ldots,N_{\mathrm{KL}}.
\end{equation}

\subsection{Convergence of fitting process}
We now proceed to the second-stage least-squares fitting of the covariance hyperparameters.  First we convert the eigenvalue perturbation bound \eqref{totalerror} into a corresponding bound in the log scale.
Denote $y_j:=\log\widehat{\lambda}_j^{\mathrm{rec}},\ y_j^\ast:=\log\lambda_j^{\mathrm{mod}}(\sigma^\ast,\ell^\ast)$, then under several assumptions we will arrive at the logarithmic version of the estimate towards the eigenvalues:
\begin{Proposition}\label{loglam}
Suppose there exist constants $c_0>0$ such that for all $j=0,1,\ldots,N_{\mathrm{KL}}$, $\lambda_j^{\mathrm{mod}}(\sigma^\ast,\ell^\ast)\ge c_0$ and  the errors
$\eta_j$ for each $j$ defined in \eqref{totalerror} can be bounded by $E_\lambda$. 
    If $E_\lambda\leq\frac{c_0}{2}$,  then the following estimate holds:
    \begin{equation}
|y_{j}-y_{j}^{\ast}|=|\log\widehat{\lambda}_j^{\mathrm{rec}}-\log\lambda_j^{\mathrm{mod}}(\sigma^\ast,\ell^\ast)|\leq C^{\prime}E_{\lambda}.
    \end{equation} where $C^\prime>0$ is a constant depending only on $c_0$.
\end{Proposition}
\begin{proof}
    Since $ |\eta_j|\leq E_\lambda$ and $\lambda_j^{\mathrm{mod}}(\sigma^\ast,\ell^\ast)\geq c_0$ (for simplicity, we write $\lambda_j^{\mathrm{mod}}$ for $\lambda_j^{\mathrm{mod}}(\sigma^\ast,\ell^\ast)$ in the following proof), we have $\left|\frac{\eta_{j}}{\lambda_{j}^{\mathrm{mod}}}\right|\leq\frac{E_{\lambda}}{c_{0}}$.  Given the condition $E_\lambda\leq \frac{c_0}{2}$, then $\left|\frac{\eta_{j}}{\lambda_{j}^{\mathrm{mod}}}\right|\leq\frac{E_{\lambda}}{c_{0}}\leq\frac{1}{2}$. It follows that
    \[
   \Big| y_{j}-y_{j}^{\ast}\Big|=\Big|\log\widehat{\lambda}_{j}-\log\lambda_{j}^{\mathrm{mod}}\Big|=\Big|\log(\lambda_{j}^{\mathrm{mod}}+\eta_{j})-\log\lambda_{j}^{\mathrm{mod}}\Big|=\bigg|\log{\left(1+\frac{\eta_{j}}{\lambda_{j}^{\mathrm{mod}}}\right)}\bigg|.
    \]
    With the inequality $|\log(1+t)|\leq2|t|,\ |t|\leq\frac{1}{2}$, finally we get
    \[
    |y_j-y_j^*|=\left|\log\left(1+\frac{\eta_{j}}{\lambda_{j}^{\mathrm{mod}}}\right)\right|\leq2\left|\frac{\eta_{j}}{\lambda_{j}^{\mathrm{mod}}}\right|\leq\frac{2}{c_0}E_\lambda:=C^\prime E_\lambda. \qedhere 
    \] \end{proof}
\noindent The uniform lower bound $c_0$ introduced in the Proposition \autoref{loglam} serves to ensure a consistent bounded error $C^\prime E_\lambda$ for the data used in the subsequent least-squares fitting. Additionally, this restriction ensures that high-frequency eigenvalues, whose magnitude may be comparable to the error $E_{\lambda}$, are excluded from the  fitting process  as their information content is likely to be severely corrupted by noise. 

For $y_{j}^{\ast}$, taking the expression for $\lambda_{j}^{\mathrm{mod}}$ in \eqref{lammod} into the logarithm formulation yields the following linear relation in terms of the true covariance parameters $(\sigma^\ast,\ell^\ast)$:
\begin{equation}\label{bless}
    y_{j}^{\ast}:=\log\lambda_{j}^{\mathrm{mod}}(\sigma^\ast,\ell^\ast)=\underbrace{\log\left(\sqrt{\pi}\frac{{\sigma^\ast}^2}{\ell^\ast}\right)}_{=:A^\ast}-\underbrace{\frac{{\ell^\ast}^2}{4}}_{=:B^\ast}j^{2}\xlongequal{x_j:=j^2}A^\ast-B^\ast x_j.
\end{equation}
By \autoref{convergence} and Proposition \autoref{loglam},   data $y_j$,  $y_j^\ast$  formulate a  linear least-squares regression problem with additive noise up to a uniformly bounded residual term:
\begin{equation}\label{bless2}
    y_j=A^\ast-B^\ast x_j+\eta_j,\quad|\eta_j|\leq C^{\prime}E_\lambda,\qquad \forall\ j=0,\dots,N_{\mathrm{KL}}-1,
\end{equation}
therefore, a standard error analysis for linear regression yields the convergence statement below for $A_{\mathrm{fit}},\ B_{\mathrm{fit}}$ obtained from \eqref{AB}:
\begin{theorem}\label{bless5}
  Let $A_{\mathrm{fit}},B_{\mathrm{fit}}$ are the solutions to \eqref{AB}, $y^\ast_j,A^\ast:=\log\left(\sqrt{\pi}\frac{{\sigma^\ast}^2}{\ell^\ast}\right)$ and  $B^\ast=\frac{{\ell^\ast}^2}{4}$ satisfy \eqref{bless}. If the assumptions in \autoref{convergence} and Proposition \autoref{loglam} hold,  then there exists a constant $C_1:=C_1(N_{\mathrm{KL}},\{x_j\}_{j=0}^{N_{\mathrm{KL}}})>0$ depending only on $N_{\mathrm{KL}}$ and $\{x_j\}_{j=0}^{N_{\mathrm{KL}}}$ s.t.
  \begin{equation}
     |A_{\mathrm{fit}}-A^\ast|+ |B_{\mathrm{fit}}-B^\ast|\leq C_1E_\lambda.
  \end{equation}
 Moreover,  let $\ell_{\mathrm{est}}$ and $\sigma_{\mathrm{est}}$ are obtained from $A_{\mathrm{fit}}$ and $B_{\mathrm{fit}}$ by \eqref{ell} and \eqref{sigma}, respectively. If the true covariance parameters $(\sigma^\ast,\ell^\ast)$ are contained within a compact set $I\subset(0,\infty)^2$ and $A_{\mathrm{fit}},B_{\mathrm{fit}}$ computed via \eqref{fitting}-\eqref{AB} also lie in some bounded sets $I_{AB}:=\{(A,B)\big|\|(A,B)\|<\infty, \ B>0\}$,  then $\exists \ C_2>0$(depending on these compact sets) such that
 \begin{equation}\label{bless10}
     |\sigma_{\mathrm{est}}-\sigma^{\ast}|+|\ell_{\mathrm{est}}-\ell^{\ast}|\leq C_{2}{\left(|A_{\mathrm{fit}}-A^{\ast}|+|B_{\mathrm{fit}}-B^{\ast}|\right)}\leq C_1C_2E_\lambda.
 \end{equation}
\end{theorem}
 

\begin{proof}
    See \autoref{proofofbless5}.
\end{proof}

\subsection{Convergence of the RLA-based shape reconstruction}
Earlier we have investigated the convergence properties of the KL eigenvalues and statistics for  second Stage 2: \autoref{convergence} derives an error upper bound \eqref{coeest} for the reconstructed eigenvalues against the true ones, Proposition \autoref{loglam} and \autoref{bless5} yield an overall error bound \eqref{bless10} between the estimated and true covariance parameters. The validity of all the theoretical convergence estimates established above is predicated on a precise characterization of the convergence behavior in the first-stage geometric reconstruction. 
 
 However, in practical computations  it is often difficult to directly quantify $\widetilde e^{(s)}$ in \autoref{convergence} since the $\varepsilon$ in \eqref{assumption1}  is not computable  due to the unknown true parameters of the scatterer's  geometry. In practice, a heuristic alternative is to use 
\begin{equation}\label{etas}
    \nu_s
:=\frac{1}{m}\sum_{i=1}^{m}\bigl(\widehat r^{(s)}(\theta_i,k_{n_K})-\widehat r^{(s)}(\theta_i,k_{n_{K-1}})\bigr)^2,
\quad \text{or equivalently,} \quad
\nu_s := \|\widehat x_{n_K}^{(s)}-\widehat x_{n_{K-1}}^{(s)}\|_2^2
\end{equation}
as a proxy for $\|\widetilde e^{(s)}\|^2 $ since an accurate reconstruction in the first stage yields a shape that stabilizes well at the highest frequency and exhibits low sensitivity to frequency increments, leading to a gradually decreasing difference between consecutive frequency reconstructions quantified by $\nu_s$. But we should note that the $\nu_s$ in \eqref{etas} is merely an a posteriori diagnostic and does not constitute a rigorous error estimator:  the intrinsic ill-posedness of IOSP implies that weak sensitivity to frequency increments does not necessarily imply proximity to the true geometry, hence a rigorous convergence analysis based on $\nu_s$ is currently lacking.

In fact, for the  shape inversion method we adopt in this paper at Stage 1, Sini et al. established the following convergence result in \cite{sini2012inverse}:
\begin{theorem}[\cite{sini2012inverse}]\label{Siniconver}
    Let $\delta$ be the noise level of the measured far field patterns and $\widetilde{\delta}>\delta$. Denote $X_{M_j}$ the $M_j$-dimensional subspace  spanned by $\{1,\cos t,\sin t,\cdots,\cos M_jt,\sin M_jt\}$. A finite-dimensional observable shape $D(\widetilde{r}(k))$ is defined as a domain of the radial function $\widetilde{r}(k)\in X_M^+$ for some $M\in\mathbb{N}$  and the corresponding far field pattern $\mathcal{F}(\widetilde{r}(k),k)$ satisfies the condition $\left\| \mathcal{F}(\widetilde{r}(k), k) - u^{\infty, \delta}(\cdot, k) \right\|_{L^2(\mathbb{S}^1)}^2 \leq \widetilde{\delta}$. The far field pattern is measured at the discrete set of frequencies $k_j := k_1 + j\Delta k,\ j = 0,1,\dots,n_K,\ \text{with}\ \Delta k = \frac{k_h - k_l}{n_K}=\frac{k_{n_K}-k_1}{n_K}$. Assume that the smallest singular value $\sigma_{min}$ of the linear operators $\widetilde{A}_j \big|_{x_{M_j}},\ j \in \{0, \dots, n_K\}$ satisfies  the condition (45) in \cite{sini2012inverse} and the regularization parameter $\alpha$ satisfies (47) in \cite{sini2012inverse} for some fixed values $\epsilon,\xi\in(0,1)$, then there exists a constant $N_0=N_0(\alpha)$ such that if $\|\widetilde{r}(k_l) - r_0\|_{\partial D} \leq c_0 \alpha$, then the following estimate holds
    \begin{equation}\label{Sinire}
        \|\widetilde{r}(k_h) - r_N\|_{\partial D} \leq \frac{C}{n_K\sqrt{\alpha}} + \frac{\widetilde{\delta}}{2\sqrt{\alpha}(1 - \epsilon)}, \quad \forall N \geq N_0,
    \end{equation}
    where $c_0$ is given by (36) in \cite{sini2012inverse} and $C$ is a constant independent of $\widetilde{\delta},\alpha$ and $N$.
\end{theorem}
\noindent This theorem provides a quantitative convergence bound between the finitely parameterized reconstruction $r_N$ and the observable shape variable $\widetilde r$ within the RLA framework. Beyond the hypotheses in \autoref{Siniconver}, they also  assume another Lipschitz-type condition in wavenumber in \cite{sini2012inverse}: there exists $d_0>0$ such that $\|\widetilde r(k_{j+1})-\widetilde r(k_j)\|_{\partial D} \le d_0|k_{j+1}-k_j|$ for $j=0,\dots,N-1$. As noted in \cite{sini2012inverse}, the existence of such an observable limit $\widetilde r$ is itself not rigorously established.
This naturally raises the question of whether our second-stage statistical analysis, from convergence of discrete covariance spectra to consistency of the fitted $(\sigma,\ell)$, can be formulated directly with respect to $\widetilde r$, thereby avoiding the untestable proximity-to-truth assumption in \autoref{convergence}. 

Conceptually, one may treat $\widetilde r$ as the data-dependent limit from the first-stage procedure and define the “reference” fluctuation field around $\widetilde r$. The reconstructed empirical covariance can then be compared to the covariance operator associated with $\widetilde r$, so that the eigenvalue-based inference (including the log-linear fit for $\sigma$ and $\ell$) is interpreted as estimating statistical parameters at the observable resolution. But turning this interpretation into a fully rigorous coupled convergence theory is considerably more subtle: first, the observable shape $\widetilde r$ is not fully characterized from a theoretical viewpoint, its existence is not guaranteed in general and even when it exists it may fail to be unique as pointed out in \cite{sini2012inverse}; second, due to the intrinsic ill-posedness of  IOSP,  the reconstruction process may amplify small perturbations in  data, the discretization or the algorithmic setup into pronounced  variations in the recovered high-frequency geometry. 
Therefore, obtaining quantitative bounds that (i) transfer the geometric error  into a sharp spectral error, especially in higher modes, and (ii) further translate that spectral error into explicit consistency estimates for $(\sigma,\ell)$ would  be left for the future study.

\section{Numerical experiments} \label{Numerical_experiments}

In this section, we present several numerical experiments to verify the effectiveness of our two-stage algorithm. As illustrated in Remark \ref{remark1}, the current random geometric framework for scatterers of interest is confined to the smooth star-shaped category where the parameter $t$ exhibits a linear relationship with the polar angle $\theta$. Common scatterers satisfying this requirement include the circle, pear and flower shapes. Therefore, the numerical examples in the remainder of this paper will focus on these three structures as depicted in \autoref{pic:scatterer}:

\begin{itemize}
 \item circle: the boundary is parameterized as   in \eqref{circle}.
\item pear: the boundary is parameterized as  in \eqref{pear}.
\item flower: the boundary is parameterized as:
\begin{equation}
        \partial D_{\mathrm{flower}} = \left\{ x_0 + r_{\mathrm{flower}}(\theta) (\cos\theta, \sin\theta) : \theta \in [0, 2\pi) \right\},\quad r_{\mathrm{flower}}(\theta) = R \left(1 + \varepsilon \cos(j\theta) \right).
\end{equation}
\end{itemize}
For flowers, $R$ controls the overall size of the obstacle,
$\varepsilon$ is the relative amplitude of the radial oscillation
and $j$ determines the number of petals.

\begin{figure}[htbp]
    \centering
    \begin{subfigure}{0.3\textwidth}
        \includegraphics[width=\linewidth]{ 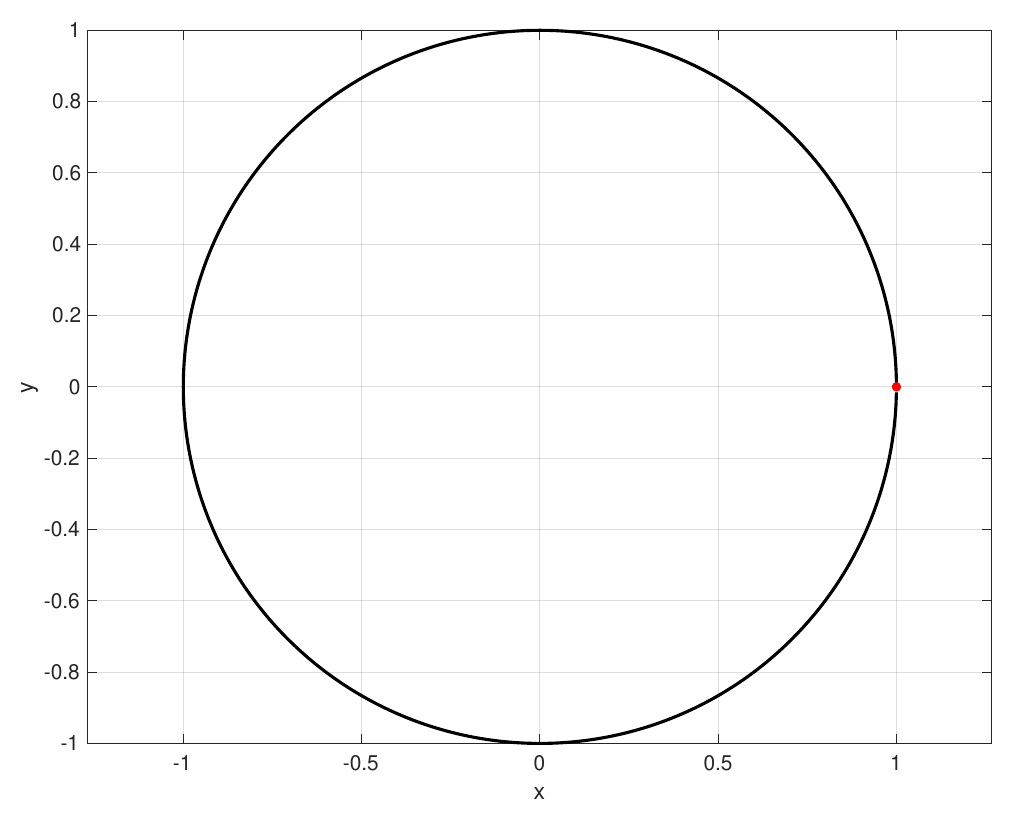}
        \caption{\small circle}
        \label{pic:circle}
    \end{subfigure}%
    \hspace{0.5cm} 
     \begin{subfigure}{0.3\textwidth}
        \includegraphics[width=\linewidth]{ 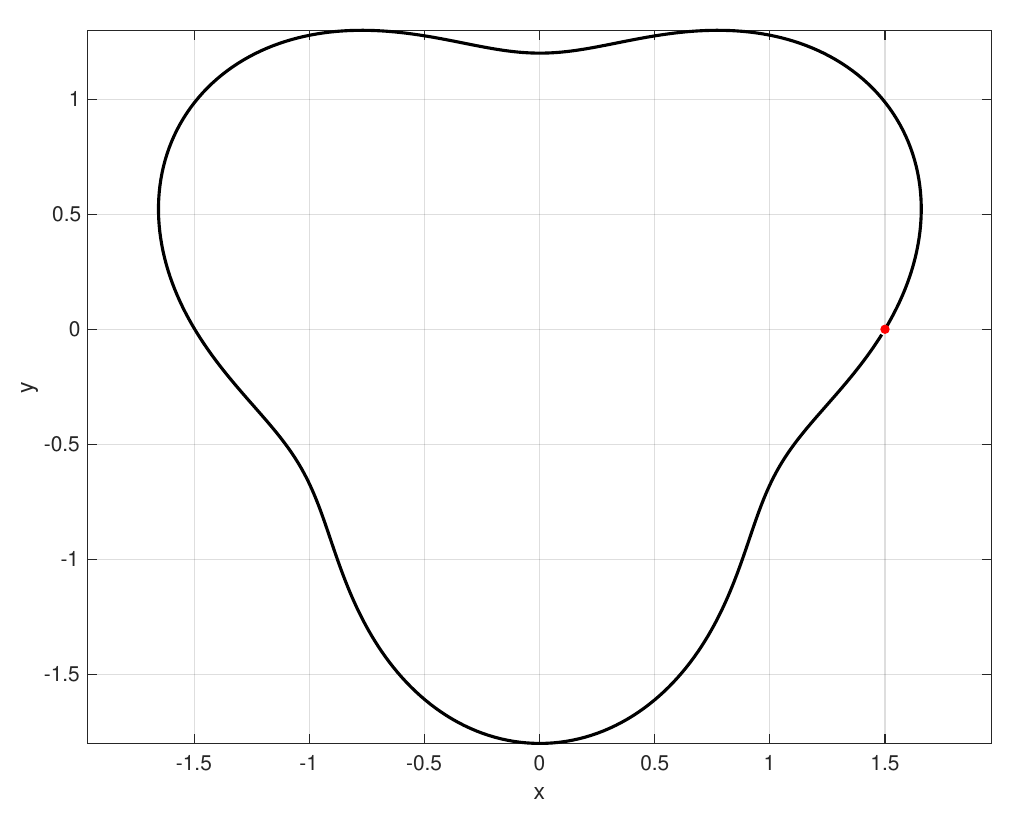}
        \caption{\small pear } 
        \label{pic:pear}
    \end{subfigure}%
    \hspace{0.5cm}
     \begin{subfigure}{0.3\textwidth}
  \includegraphics[width=\linewidth]{ 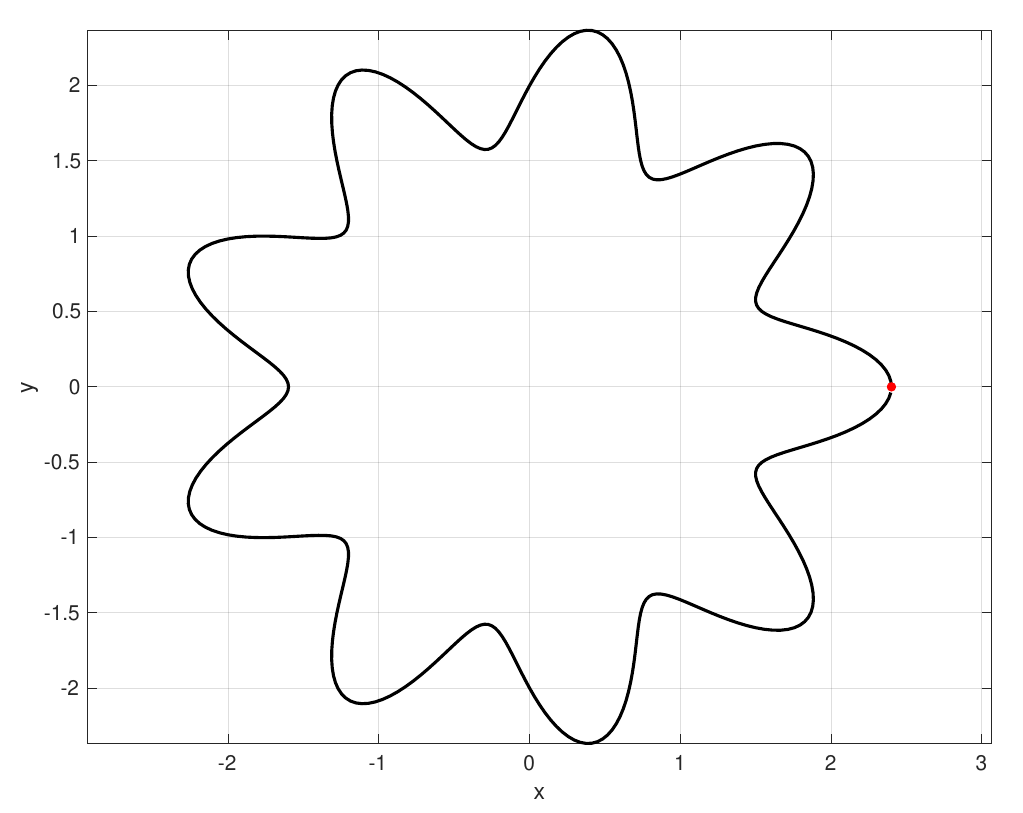}
  \caption{\small flower }
        \label{pic:flower}
    \end{subfigure}
    \caption{\small the geometry of the scatterer}
    \label{pic:scatterer}
    \end{figure}

\subsection{Experiment Setup} \label{experiment_setup}
Throughout this section we assume that all obstacles are star-shaped with respect to the origin $(0,0)$: $\partial D = \{ r(\theta)(\cos\theta, \sin\theta) : \theta \in [0, 2\pi] \}.$
The radial function is represented by a truncated Fourier expansion as given in \eqref{eq:param-radius} with a parameter dimension of $n_r=2N_r+1$. For boundary discretization, we employ a uniform angular grid: $\theta_i = 2\pi i / N_\theta, \  i = 0, \dots, N_\theta - 1$ with $N_\theta=400$. In the truncated KL expansion \eqref{trunKLexpan}, the selected modes are chosen by the threshold  $\lambda_j\geq 10^{-6}$ which yields a finite number of KL terms. The random variables are taken to be i.i.d. $\xi_0,\ \xi_{j,c},\  \xi_{j,s}(\omega)\sim\mathcal N(0,1)$ in \eqref{trunKLexpan}. For the forward (direct) problem, the sound-soft scattered field $u^s$ is formulated using a combined
acoustic double- and single-layer potential representation \cite{colton1998inverse} with:
\begin{equation}\label{us}
    u^s(x) = \int_{\partial D} \left\{ \frac{\partial \Phi(x, y)}{\partial \nu(y)} - i\eta \Phi(x, y) \right\} \phi(y) \, ds(y), \quad x \in \mathbb{R}^2 \setminus \partial D,
\end{equation}
where $\Phi(x,y)$ denotes the Green function for the two-dimensional Helmholtz equation and $\phi$ is the unknown  density function defined on $\partial D$.
 We set the coupling parameter to $\eta=k$ for each wavenumber $k$ and the boundary integral in \eqref{us} will be discretized by  Nystr\"om method.  We fix the incident direction to be $d=(-1,0)^\top$. The far-field pattern $u^\infty$  will be  computed via \cite{colton1998inverse}:\[u_\infty(\hat{x}) = \frac{e^{-i\frac{\pi}{4}}}{\sqrt{8\pi k}} 
\int_{\partial D} \left\{ k\, \nu(y) \cdot \hat{x} + \eta \right\} 
e^{-ik \hat{x} \cdot y} \, \phi(y) \, ds(y),\quad |\hat{x}|=1\] evaluated at angles $\{\hat x_m\}_{m=1}^M\subset\mathbb S^1$ and the number of observation directions is fixed at $M=200$, uniformly distributed over $[0,2\pi)$. For multifrequency data, we employ $k_j = j,\   j = 1, \dots, n_K$ where $ n_K = 8$. Noise-free far-field data $u^\infty$ 
are generated for each sample and each frequency. Subsequently, relative Gaussian noise is added: $u^{\infty,\delta} = u^\infty + \delta \frac{\|u^\infty\|}{\|\xi\|} \, \xi,
\  \xi \sim \mathcal{N}(0, I)$ with a uniform noise level of $\delta=5\%$. Unless otherwise specified, the parameter $N_{\mathrm{KL}}$ is fixed at 4 in all subsequent experiments. To validate the correctness of the far-field data, we have also included an additional experiment to verify the reciprocity relation. Specifically, we check whether the far-field data $u^\infty$ satisfy the following relation:
\begin{equation}\label{reciprocity}
    u_\infty(\hat{x}, d) = u_\infty(-d, -\hat{x}), \quad \hat{x}, d \in \mathbb{S}^1,
\end{equation}
Three plots in \autoref{pic:farfieldre} demonstrate that the data indeed satisfy this reciprocity condition. 
\begin{figure}[htbp]
    \centering
    \begin{subfigure}{0.3\textwidth}
        \includegraphics[width=\linewidth]{ 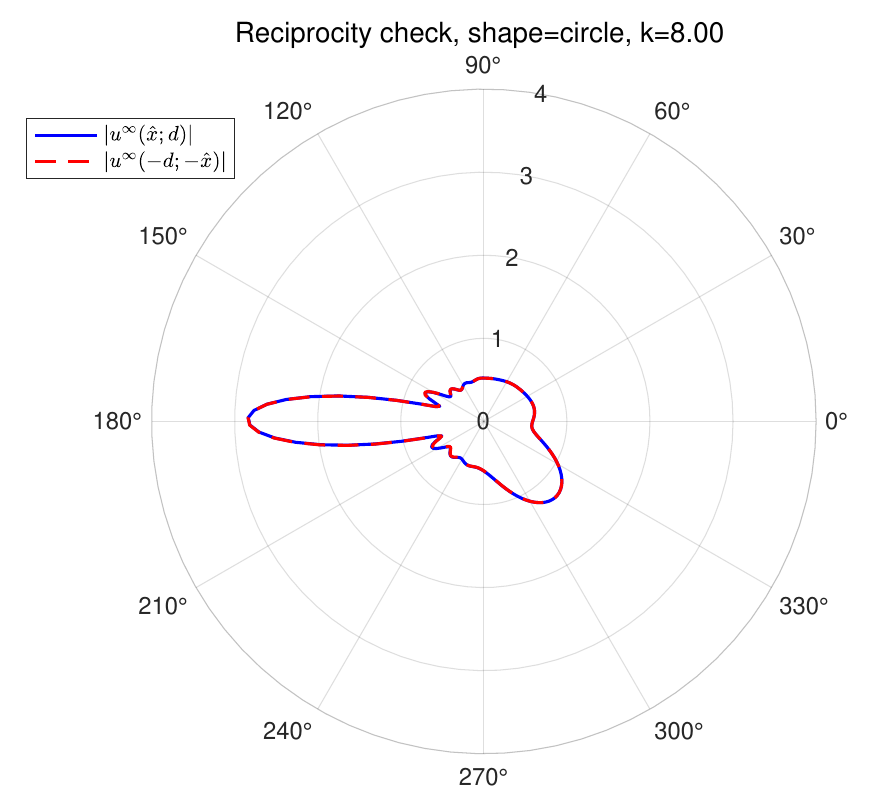}
        \caption{\small circle}
        \label{pic:far1}
    \end{subfigure}
    \hspace{0.5cm}
     \begin{subfigure}{0.3\textwidth}
        \includegraphics[width=\linewidth]{ 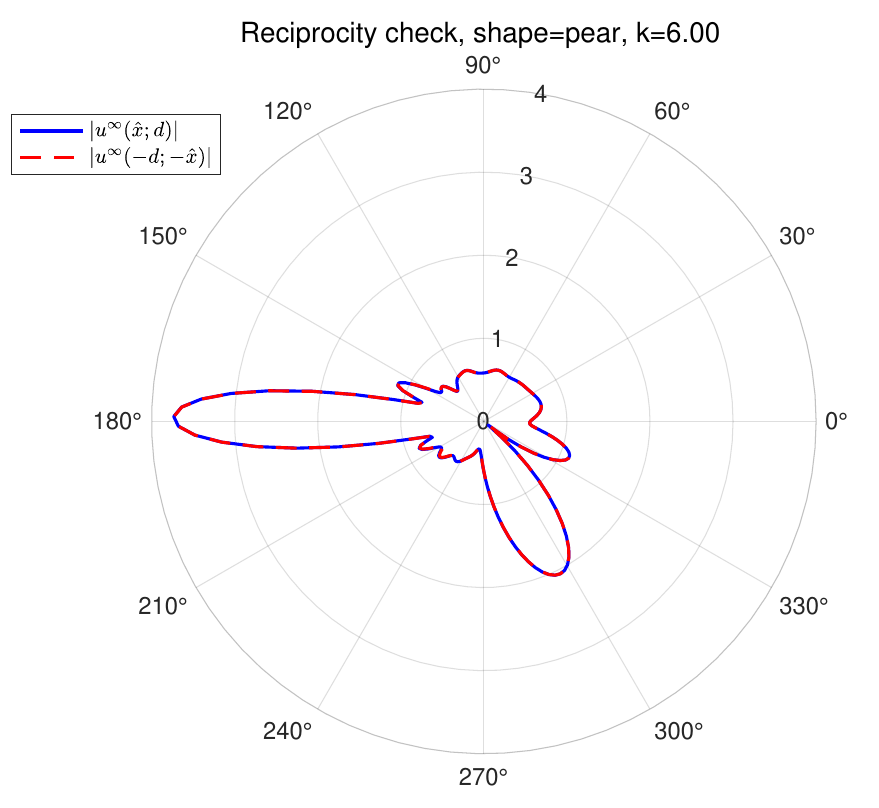}
        \caption{\small pear } 
        \label{pic:far2}
    \end{subfigure}%
    \hspace{0.5cm}
     \begin{subfigure}{0.3\textwidth}
  \includegraphics[width=\linewidth]{ 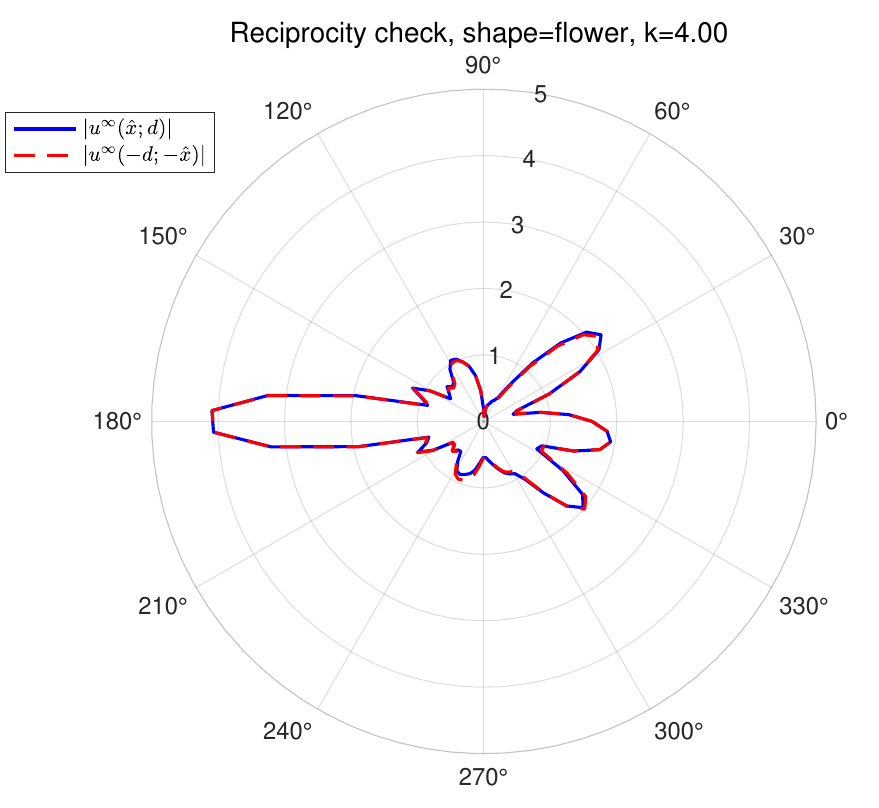}
  \caption{\small flower }
        \label{pic:far3}
    \end{subfigure}
    \caption{\small verification of reciprocity relation \eqref{reciprocity} plotted in polar coordinates. red dashed line: $|u^{\infty}(-d,-\hat{x})|$; blue solid line: $|u^{\infty}(\hat{x},d)|$. } 
    \label{pic:farfieldre}
    \end{figure}
    
For the low-frequency nonlinear  optimization problem \eqref{Tikhonov}, we employ the \texttt{fmincon} solver in MATLAB to solve it. All examples start with  the initial guess  by setting only the constant term $p_1\neq 0$ while setting all other coefficients to zero. The regularization parameter $\gamma=$ is $10^{-2}$ and the maximum number of iterations $N_{\mathrm{it}}$ is set to 50. The stopping criterion is $\|\Delta p\|/\|p\|\le 10^{-3}=\texttt{tol}$. The Fréchet derivative $\mathcal{F}'$ is approximated via a forward finite difference scheme with a step size of
$ \varepsilon_{\rm fd}=10^{-5}$.  For the high-frequency recursive linearization, the regularization parameter $\alpha$ is set to $0.1$. If certain samples exhibit significant shape inversion deviations (which will be specified in the numerical examples), this value is increased to $0.5$. Within each frequency $k_j$, a maximum of 5 iterations is allowed and early termination will be  triggered if the relative change in the step size $\Delta r_j$ falls below $10^{-3}$.  

    Since the inversion is carried out independently for each sample (the processes are mutually independent), the treatment of multi-sample data naturally lend itself to parallel computation. In our implementation, we utilize MATLAB's \texttt{parfor} to parallelize the operations across samples. However, as will be evident from the numerical results presented later, the proposed algorithm does not require a large amount of data, $20\sim50$ samples are sufficient to yield accurate results within a reasonable precision range.

 \subsection{Random Circle}
 \label{example:circle}
We first consider the scattering reconstruction for a circular random obstacle and begin by fixing the stochastic parameters as $\sigma=0.05$, $\ell=1$. The true circle has radius $R=1$ and is centered at the origin $(0,0)$. In our random circle reconstruction experiments, 20 samples are drawn for analysis. The random realizations of these samples are displayed in \autoref{fig:randomcircle} where the black solid line represents the unperturbed circular geometry and the red dashed lines depict the true geometries of the randomly perturbed samples. To visualize a random sample, \autoref{fig:samplecircle} shows the true scatterer geometry of the sixth sample and plots its far-field pattern at the corresponding frequency on a polar diagram.

\begin{figure}[htbp]
    \centering
    \begin{subfigure}{0.35\textwidth}
        \includegraphics[width=\linewidth]{ 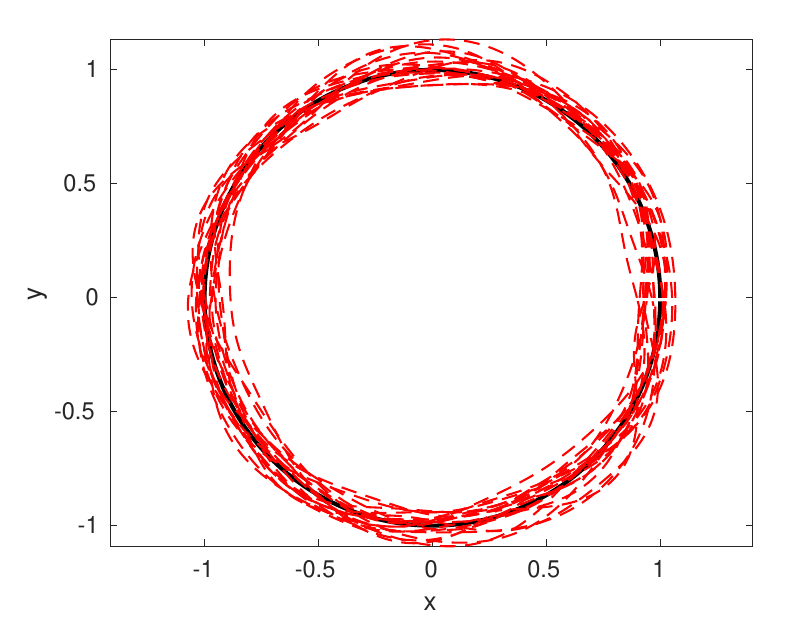}
        \caption{\small random realizations of circle }
        \label{fig:randomcircle}
    \end{subfigure}%
    \hspace{2cm} 
     \begin{subfigure}{0.35\textwidth}
        \includegraphics[width=\linewidth]{ 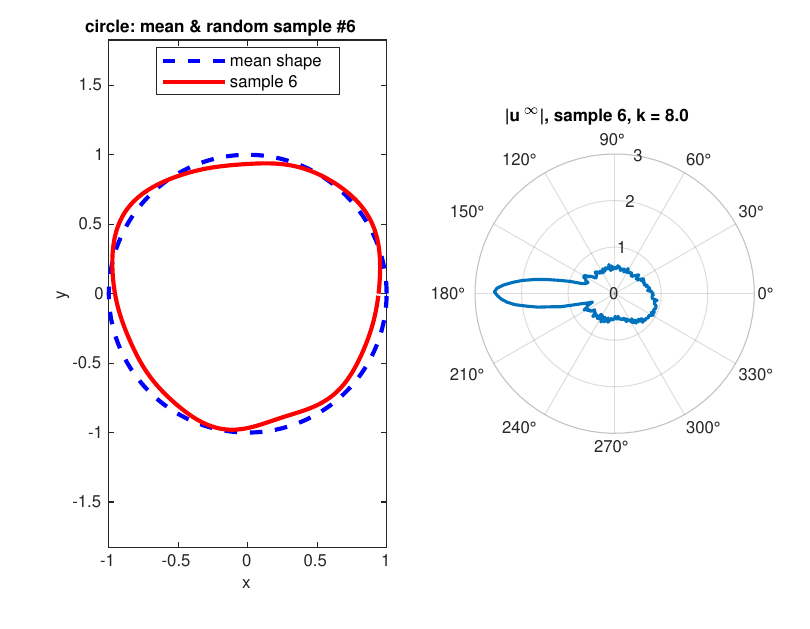}
        \caption{\small shape  and far-field data} 
        \label{fig:samplecircle}
    \end{subfigure}
    \caption{\small \subref{fig:randomcircle}: random realizations with 20 samples. dashed line: random circular scatterer; solid line: standard circular scatterer. \subref{fig:samplecircle}: geometric visualization of the $6$-th random sample and its corresponding far-field pattern data. left:  dashed line represents the circle; solid line represents the true scatterer shape. right: far-field pattern $u^\infty(\hat{x}_m,d,k=8)$ for sample $6$ plotted in polar coordinates.}
    \label{fig:KLco}
    \end{figure}

During the first-stage shape inversion, it can be observed from \autoref{cirrla} that the reconstruction at the lowest frequency is quite coarse. As the frequency increases, more geometric details are accurately captured until the reconstruction converges at the highest preset frequency.

    \begin{figure}[t]
\begin{center}
    \begin{overpic}[width=1\textwidth, tics=10]{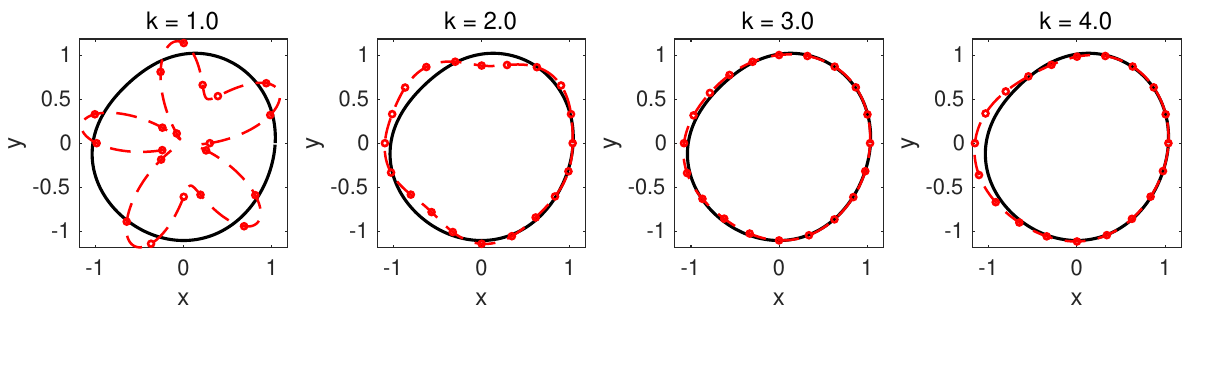}
        \put (30,30) {\scriptsize {\bf Geometric shape inversion at low frequency}}
    \end{overpic}
\end{center}
\vspace{-0.3cm}
 \begin{center}
    \begin{overpic}[width=1\textwidth, tics=10]{ 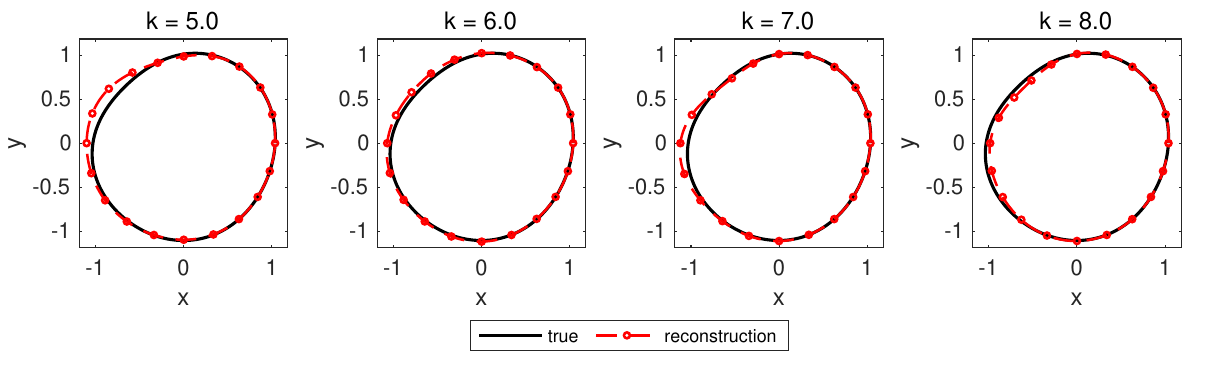}
        \put (30,30) {\scriptsize {\bf Geometric shape inversion at high frequency}}
    \end{overpic}
\end{center}
      \vspace{-0.2cm}
    \caption{\small visualization of Algorithm \ref{alg:mean-shape-random} executed for a single  circular sample.}
    \label{cirrla}
      \end{figure}


It is noteworthy that when inverting many realizations independently, a small number of reconstructions may behave as outliers: despite using the same algorithmic settings, they exhibit markedly poorer data consistency or noticeably less stable frequency continuation than the bulk of the samples, reflecting the intrinsic ill-posedness of the inverse problem. This ill-posedness may stem from an excessively large condition number of the coefficient matrix in the normal equations \eqref{normal} or from substantial deviations in the optimization results produced by \eqref{Tikhonov} or \eqref{RLATikhonov}. For such outlier samples, we can remediate these outliers by several methods such as increasing the regularization coefficient $\alpha$ and re-running the inversion for those chosen samples.  Once the corrected inversion results are in place, we simply replace the aberrant reconstructions with the successful ones. As shown in \autoref{pic:allcirsamp}, the reconstruction results for all 20 samples are displayed, with only sample $\#11$ exhibiting a significant deviation. To address this, sample $\#11$  is identified and re-inverted separately with an increased regularization parameter of $\alpha=0.5$. A comparison of the errors before and after this correction is presented in \autoref{fig:repair} below, demonstrating that increasing the regularization factor helps to mitigate ill-posed samples, and the shape comparison is shown in \autoref{fig:repair2}. The corrected results for samples $\#11\sim\#15$ are shown in \autoref{fig:11repair}.

  \begin{figure}[t]
    \begin{center}
    \begin{overpic}[width=1\textwidth,tics=10]{ 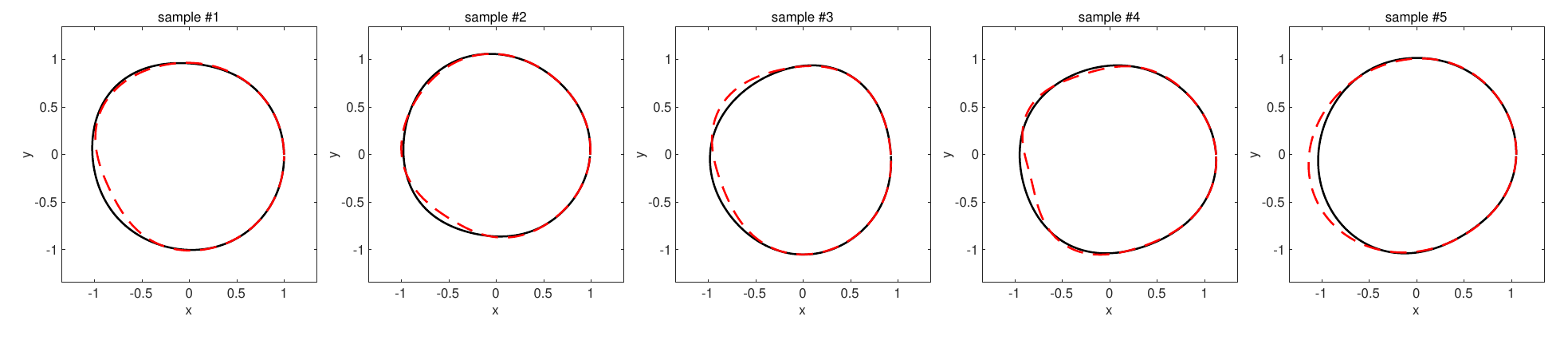}
    \end{overpic}
    \end{center}
    \vspace{-0.8cm}
    \begin{center}
        \begin{overpic}[width=1\textwidth,tics=10]{ 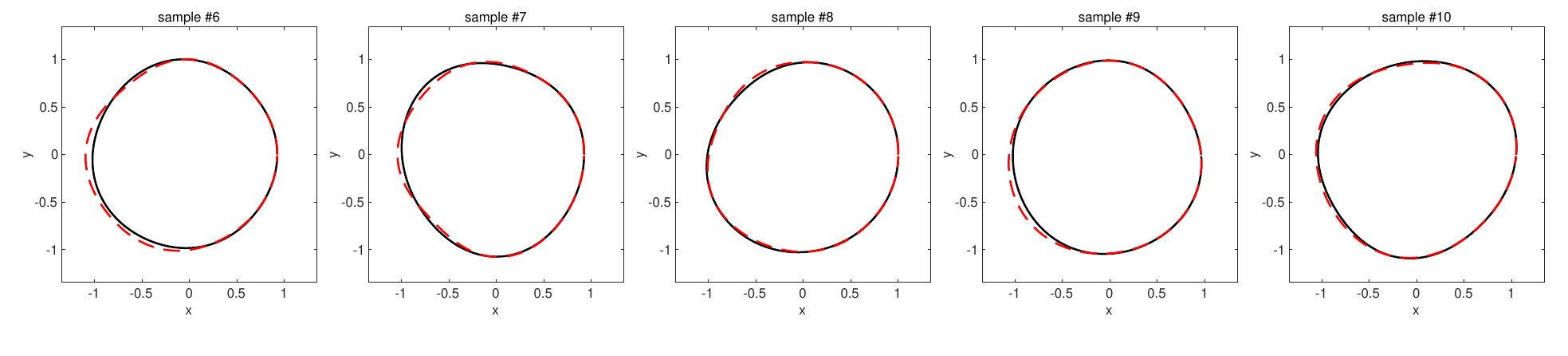}
      \end{overpic}
    \end{center}
    \vspace{-0.8cm}
\begin{center}
    \begin{overpic}[width=1\textwidth, tics=10]{ 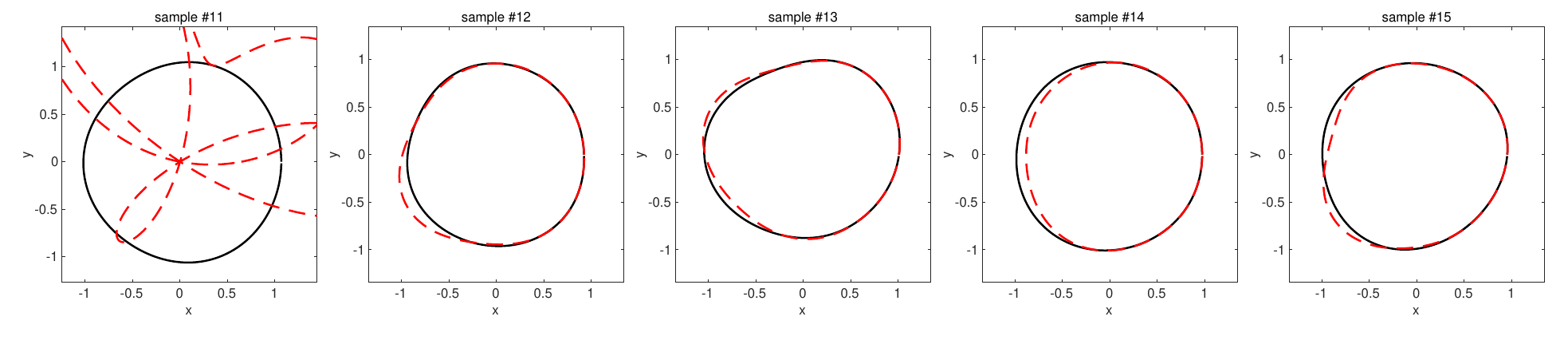}
    \end{overpic}
\end{center}
\vspace{-0.6cm}
 \begin{center}
    \begin{overpic}[width=1\textwidth, tics=10]{ 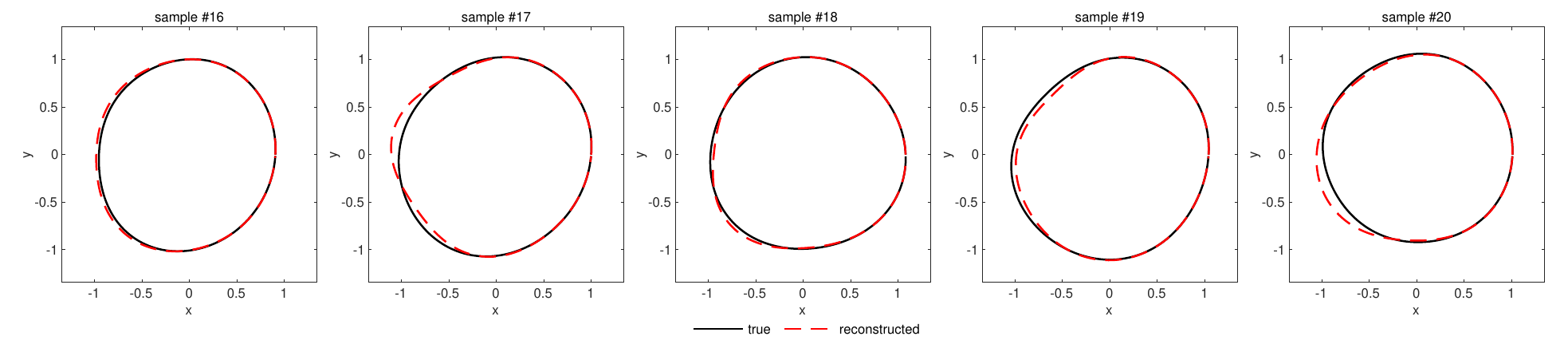}
    \end{overpic}
\end{center}
      \vspace{-0.2cm}
    \caption{ \small reconstruction results for all circular samples  without repairing.}
    \label{pic:allcirsamp}
      \end{figure}

      \begin{figure}[htbp]
    \centering
    \begin{subfigure}{0.35\textwidth}
        \includegraphics[width=\linewidth]{ 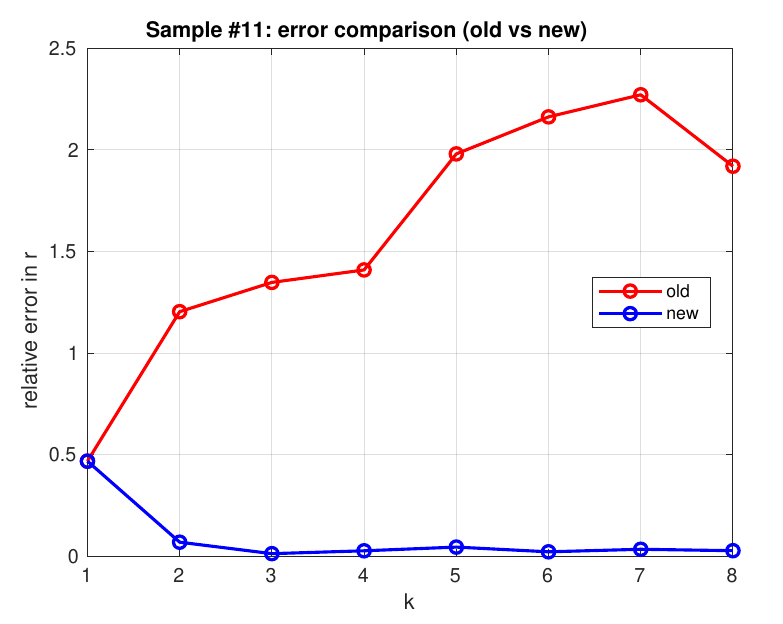}
        \caption{\small error comparison }
        \label{fig:repair}
    \end{subfigure}%
    \hspace{2cm} 
     \begin{subfigure}{0.35\textwidth}
        \includegraphics[width=\linewidth]{ 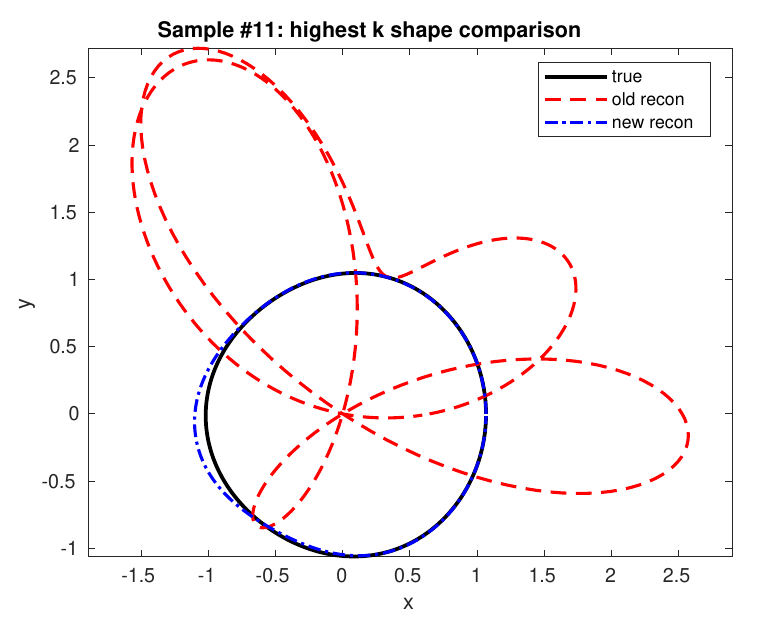}
        \caption{\small shape comparison} 
        \label{fig:repair2}
    \end{subfigure}
    \caption{\small \subref{fig:repair}: reconstruction error for sample $\#11$: initial vs. corrected inversion compared to true geometry. \subref{fig:repair2}: comparison of the old sand new shape inversion curves for sample $\#11$.}
    \label{fig:cirrepair}
    \end{figure}

\begin{figure}[htbp]
    \centering
        \includegraphics[width=\linewidth]{ 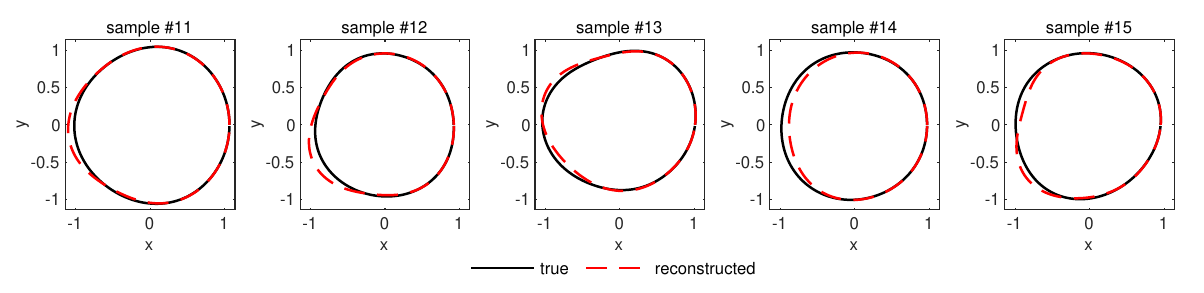}
        \caption{\small reconstruction results for $\#11\sim\#15$ circular samples after repairing.}
        \label{fig:11repair}
    \end{figure}
    
Combining \autoref{pic:allcirsamp} and  \autoref{fig:11repair} we can observe that the first-stage algorithm achieves a good recovery of the geometric shapes. The average value of the geometric radii across these samples is calculated to be 1.001, which aligns well with the true radius parameter of the circle. Upon reconstructing the shape of each sample, we execute the second-stage algorithm to invert the KL coefficients and plot the three comparison curves shown in \autoref{fig:finalcir}. According to the truncation rule set earlier, only the first six modes of eigenvalues are selected here. The blue curve represents the theoretical $\lambda_j$ values, the black curve corresponds to the KL coefficient curve obtained by inverting the true stochastic geometric samples via the two-stage algorithm in the forward problem and the red curve shows the final inversion results obtained from the far-field data through the two-stage inversion. It can be observed that the three curves exhibit a relatively high degree of agreement. Finally, the inversion results for 
$\sigma$ and $\ell$ are presented in \autoref{tab:circase1}. \autoref{tab:circase1} compares the preset (true) values, the values obtained from the true geometry (reference) and the final estimates from our two-stage algorithm. We also tested another case with parameters $\sigma=0.08$ and $\ell=0.7$. The reconstruction results for the eigenvalues 
$\lambda_j$ and the statistical parameters are presented in \autoref{fig:finalancir} and \autoref{tab:circase2}, respectively, both showing satisfactory agreement with the true values within the expected error margin.

\begin{table}[htbp]
\centering

\begin{subtable}[t]{0.48\textwidth}
\centering
\begin{tabular}{c|ccc}
\hline
 & 

True (Preset)  & Reference  & \bf{Estimated} \\
\hline
$\sigma$ & 0.05 &0.0507 & \bf{0.0571} \\
$\ell$   & 1 & 1.0771  & \bf{0.9586} \\
\hline
\end{tabular}
\caption{case 1}
\label{tab:circase1}
\end{subtable}
\hfill
\begin{subtable}[t]{0.48\textwidth}
\centering
\begin{tabular}{c|ccc}
\hline
& True (Preset)  & Reference  & \bf{Estimated}\\
\hline
$\sigma$ & 0.08 & 0.0746 & \bf{0.0874} \\
$\ell$   & 0.7  & 0.7791 & \bf{0.7484} \\
\hline
\end{tabular}
\caption{case 2}
\label{tab:circase2}
\end{subtable}
\caption{comparison of the preset (true) values, reference values computed from the true geometry (reference) and the final reconstructed values for parameters $\sigma$ and $\ell$ in the circular scatterer case.}
\label{tab:ciroverall_comparison}
\end{table}

 \begin{figure}[htbp]
    \centering
    \begin{subfigure}{0.45\textwidth}
        \includegraphics[width=\linewidth]{ 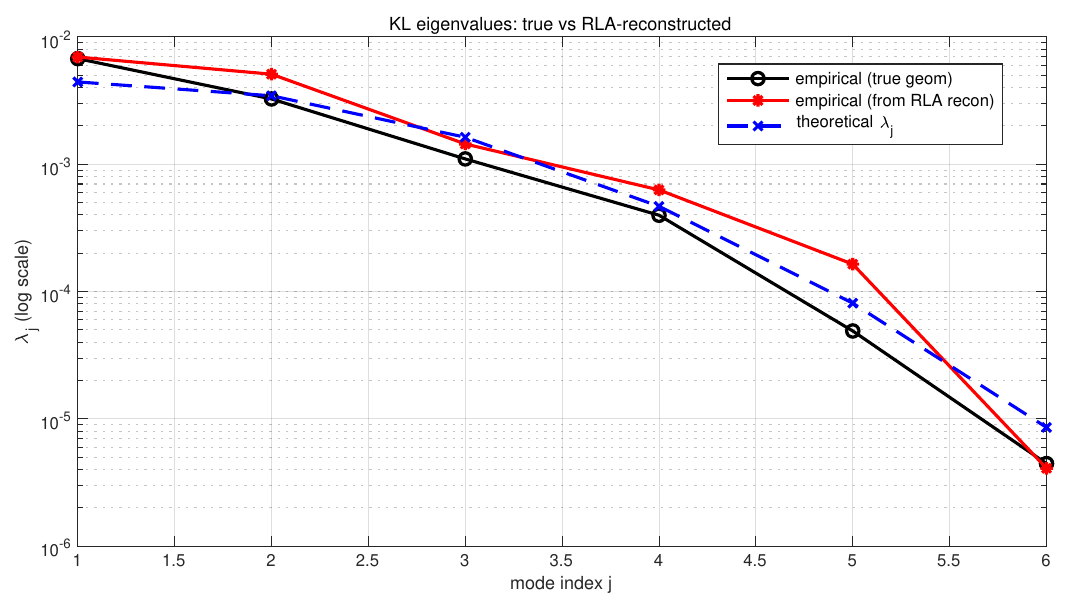}
        \caption{\small case 1: reconstructed $\lambda_j$ }
        \label{fig:finalcir}
    \end{subfigure}%
    \hspace{1cm} 
     \begin{subfigure}{0.45\textwidth}
        \includegraphics[width=\linewidth]{ 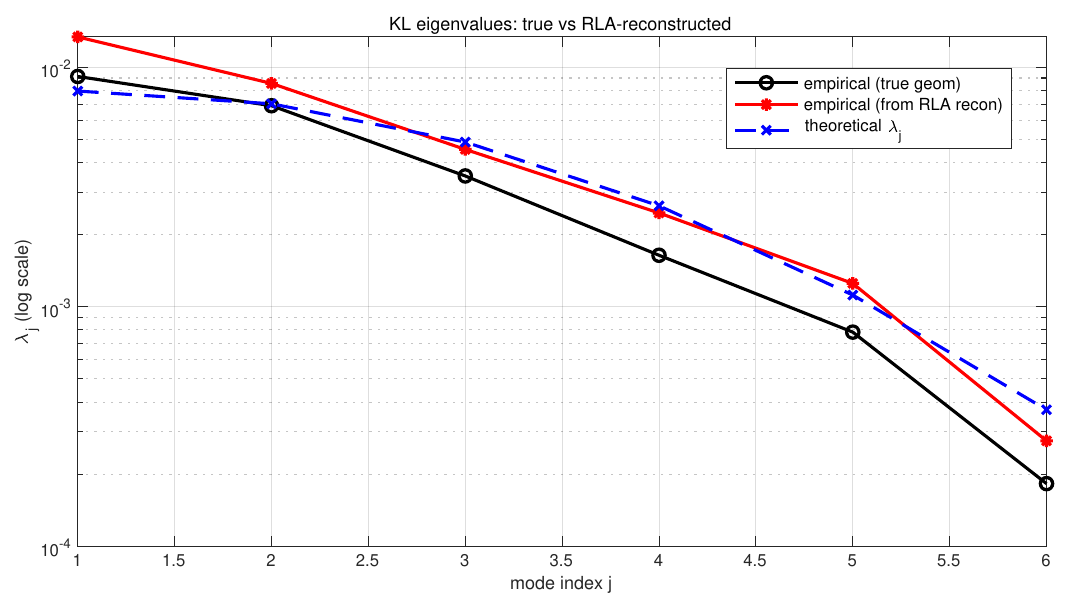}
        \caption{\small case 2: reconstructed $\lambda_j$ } 
        \label{fig:finalancir}
    \end{subfigure}
    \caption{\small random circle numerical examples: blue dashed line: theoretical $\lambda_j$ values; black solid line: eigenvalues of the discrete covariance matrix from the true stochastic fluctuations; red solid line: eigenvalues of the reconstructed discrete covariance matrix from the algorithm. \subref{fig:finalcir}: reconstructed $\lambda_j$ when $\sigma=0.05$, $\ell=1$. \subref{fig:finalancir}: reconstructed $\lambda_j$ when $\sigma=0.08$, $\ell=0.7$.}
    \label{fig:cirrepair}
    \end{figure}

\subsection{Random Pear}
We next consider the reconstruction of a  more complex geometry: the pear shape, along with its statistical parameters. The true geometry of the pear we consider here is described by
\[
r(\theta) = 1.5 + 0.3 \sin(3\theta), \quad
x(\theta) = r(\theta)\cos\theta,\quad
y(\theta) = r(\theta)\sin\theta,\quad 0 \leq \theta < 2\pi,
\]
In this experiment we also select total 20 samples for the computation and start with the setting $\sigma=0.05$, $\ell=1$. We choose $N_r=5$ in pear settings such that $n_r=2\times5+1=11$ in its Fourier representation as given in \eqref{fou}. \autoref{fig:randompear} displays the random sample realizations. The black solid line indicates the reference pear shape while the red dashed lines show the actual geometries of the randomly perturbed samples. For the pear-shaped scatterer, \autoref{fig:samplepear} visualizes a random sample by displaying the true geometry of the sixth instance and its far-field pattern on a polar diagram. \autoref{pearrla} shows the progression of the first-stage inversion for the pear shape at multiple frequencies. The reconstruction undergoes a stepwise refinement with increasing frequency.
\begin{figure}[htbp]
    \centering
    \begin{subfigure}{0.35\textwidth}
        \includegraphics[width=\linewidth]{ 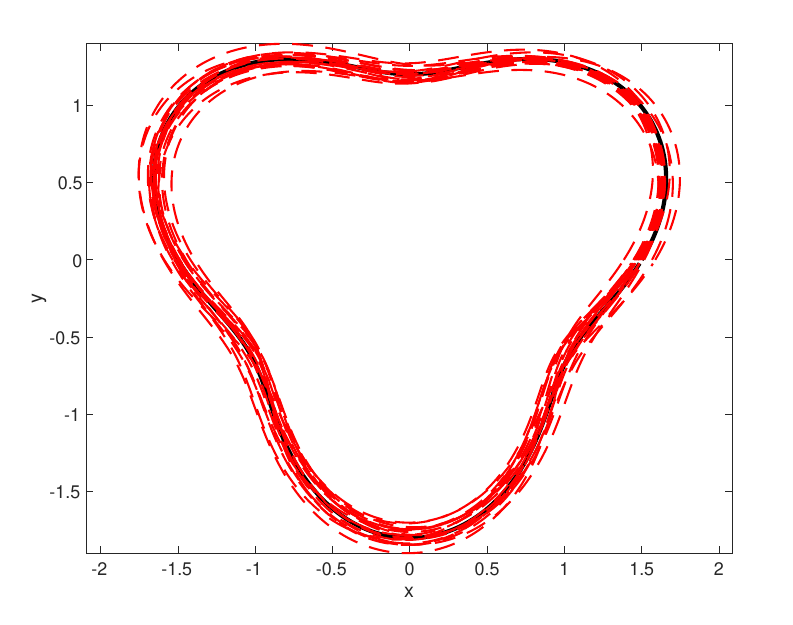}
        \caption{\small random realizations of pear}
        \label{fig:randompear}
    \end{subfigure}%
    \hspace{2cm} 
     \begin{subfigure}{0.35\textwidth}
        \includegraphics[width=\linewidth]{ 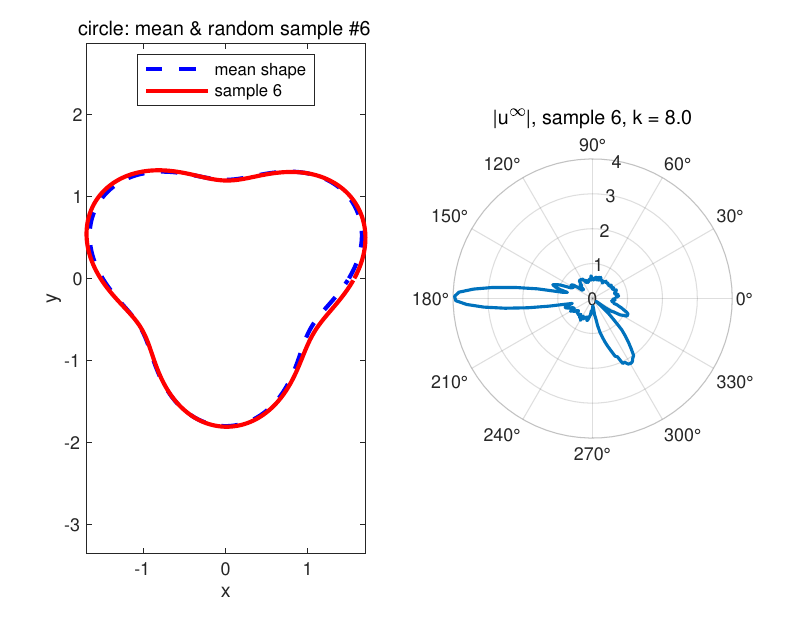}
        \caption{\small shape  and far-field data} 
        \label{fig:samplepear}
    \end{subfigure}
    \caption{\small \subref{fig:randompear}: random realization with 20 samples. dashed line: random pear-shaped scatterer; solid line: standard pear-shaped scatterer. \subref{fig:samplepear}: geometric visualization of the $6$-th random sample and its corresponding far-field pattern data. left:  dashed line represents the pear; solid line represents the true scatterer shape. right: far-field pattern $u^\infty(\hat{x}_m,d,k=8)$ for sample $6$ plotted in polar coordinates.}
    \label{fig:randomp}
    \end{figure} 

   \begin{figure}[t]
\begin{center}
    \begin{overpic}[width=1\textwidth, tics=10]{ 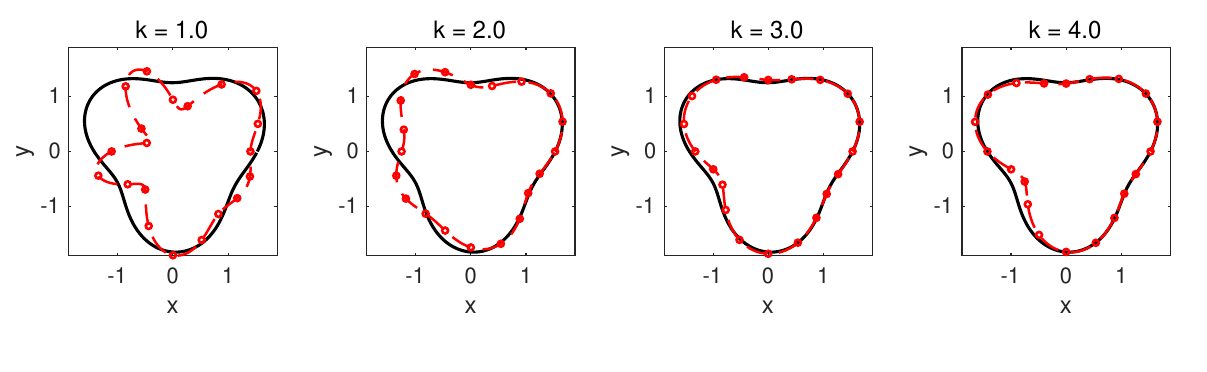}
        \put (30,31) {\scriptsize {\bf Geometric shape inversion at low frequency}}
    \end{overpic}
\end{center}
\vspace{-0.3cm}
 \begin{center}
    \begin{overpic}[width=1\textwidth, tics=10]{ 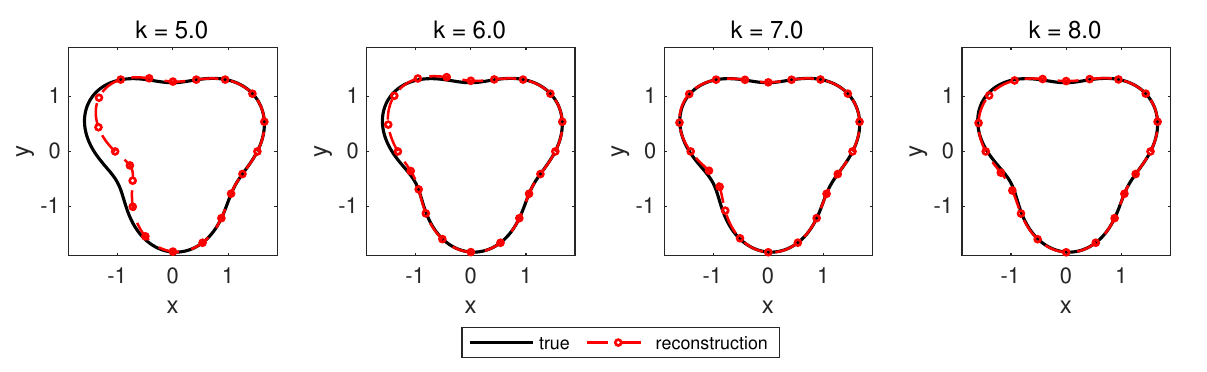}
        \put (30,31) {\scriptsize {\bf Geometric shape inversion at high frequency}}
    \end{overpic}
\end{center}
      \vspace{-0.4cm}
    \caption{\small visualization of Algorithm \ref{alg:mean-shape-random} executed for a single pear-shaped sample.}
    \label{pearrla}
      \end{figure}

Following the same workflow as given in the circular case (20 samples), we omit detailed sample-wise plots for brevity. For any sample exhibiting anomalous deviations, the inversion can be corrected by re-running it with an increased regularization parameter $\alpha$ or similar corrective measures. Once the shape inversion for all samples is completed, a comparison between the reconstructed mean and the true geometric parameters is presented, as shown in \autoref{fig:meanpear}. It is computed that the relative $L^2$ error between sample-mean radius and true radius is $7.827e-03$. For the second-stage inversion of statistical quantities, \autoref{fig:finalpear} presents the reconstructed KL coefficients for the pear shape, showing favorable reconstruction performance within acceptable tolerances of their theoretical values. Subsequently, the parameters $\sigma$ and $\ell$ are recovered, with the results displayed in \autoref{tab:pearcase1}. We also tested a configuration with $\sigma=0.08$ and $\ell=0.8$. The reconstruction results are shown in \autoref{fig:finalanpear} and the corresponding parameter estimates are listed in \autoref{tab:pearcase2}. From the numerical results in  \autoref{tab:pearoverall_comparison}, we can see that our proposed algorithm is effective for moderately complex geometries, yielding reconstructed parameters that all fall within a reasonable range across various experimental settings.



\begin{figure}[htbp]
    \centering
    \begin{subfigure}{0.35\textwidth}
        \includegraphics[width=\linewidth]{ 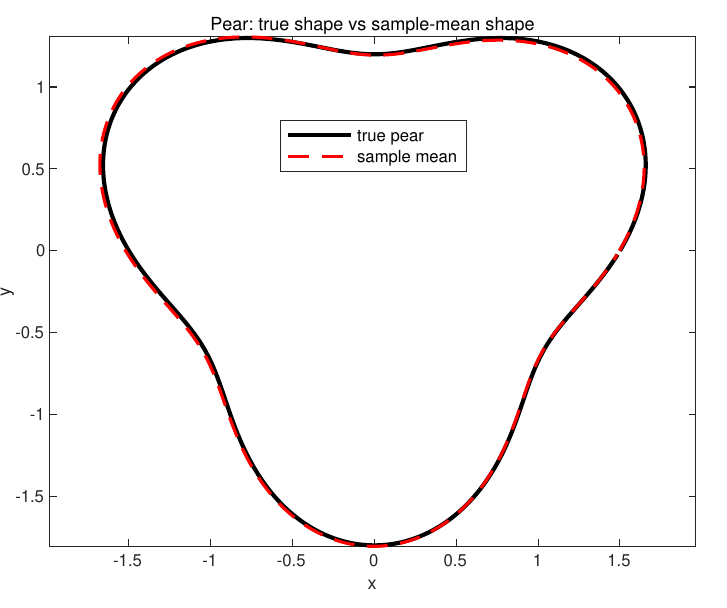}
        \caption{\small pear}
        \label{fig:meanpear}
    \end{subfigure}%
    \hspace{2cm} 
     \begin{subfigure}{0.35\textwidth}
        \includegraphics[width=\linewidth]{ 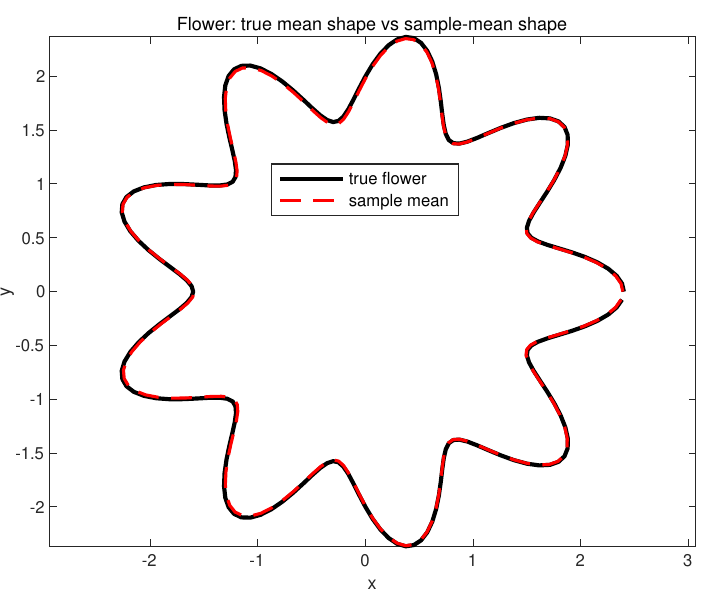}
        \caption{\small flower} 
        \label{fig:meanflo}
    \end{subfigure}
    \caption{\small \subref{fig:meanpear}: true shape vs sample-mean shape. \subref{fig:meanflo}: true shape vs sample-mean shape.}
    \label{fig:mean}
    \end{figure}

 \begin{figure}[htbp]
    \centering
    \begin{subfigure}{0.45\textwidth}
        \includegraphics[width=\linewidth]{ 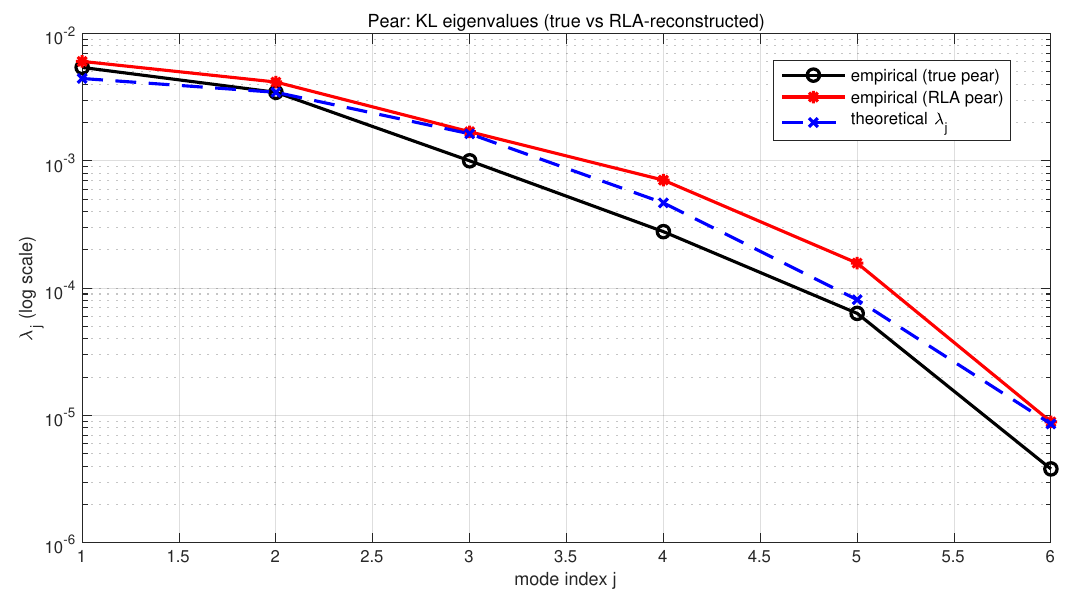}
        \caption{\small case 1: reconstructed $\lambda_j$ }
        \label{fig:finalpear}
    \end{subfigure}%
    \hspace{1cm} 
     \begin{subfigure}{0.45\textwidth}
        \includegraphics[width=\linewidth]{ 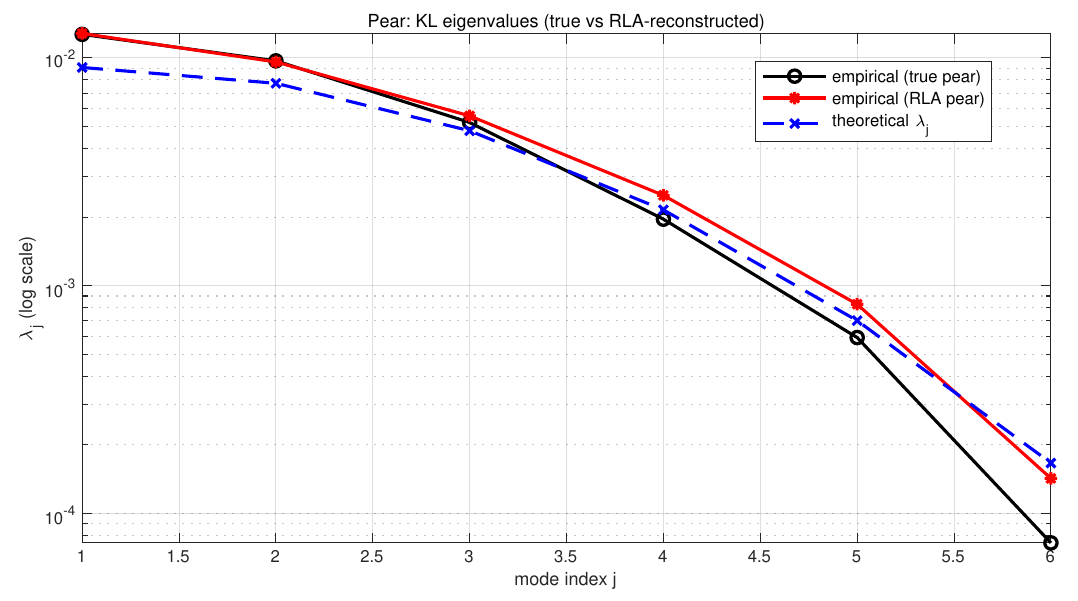}
        \caption{\small case 2: reconstructed $\lambda_j$ } 
        \label{fig:finalanpear}
    \end{subfigure}
    \caption{\small random pear numerical examples: blue dashed line: theoretical $\lambda_j$ values; black solid line: eigenvalues of the discrete covariance matrix from the true stochastic fluctuations; red solid line: eigenvalues of the reconstructed discrete covariance matrix from the algorithm.  \subref{fig:finalpear}: reconstructed $\lambda_j$ when $\sigma=0.05$, $\ell=1$. \subref{fig:finalanpear}: reconstructed $\lambda_j$ when $\sigma=0.08$, $\ell=0.8$.}
    \label{fig:pearrepair}
    \end{figure}

          \begin{table}[htbp]
\centering
\begin{subtable}[t]{0.48\textwidth}
\centering
\begin{tabular}{c|ccc}
\hline
 & 
True (Preset)  & Reference  & \bf{Estimated} \\
\hline
$\sigma$ & 0.05 &0.0473 & \bf{0.0555} \\
$\ell$   & 1 & 1.0478  & \bf{0.9408} \\
\hline
\end{tabular}
\caption{case 1}
\label{tab:pearcase1}
\end{subtable}
\hfill
\begin{subtable}[t]{0.48\textwidth}
\centering
\begin{tabular}{c|ccc}
\hline
& True (Preset)  & Reference  & \bf{Estimated}\\
\hline
$\sigma$ & 0.08 & 0.0872 & \bf{0.0896} \\
$\ell$   & 0.8  & 0.8729 & \bf{0.8191} \\
\hline
\end{tabular}
\caption{case 2}
\label{tab:pearcase2}
\end{subtable}
\caption{comparison of the preset (true) values, reference values computed from the true geometry and the final reconstructed values for parameters $\sigma$ and $\ell$ in the pear-shaped scatterer case.}
\label{tab:pearoverall_comparison}
\end{table}
 
 \subsection{Random Flower}
In this section, we attempt to apply the inversion algorithm proposed earlier to a scatterer with a more complex geometry—the flower shape. The flower boundary used in this work follows the parametric form:
\[
r(\theta) = c_1\bigl(1 + c_2 \cos(c_3 \theta)\bigr),\quad x(\theta) = r(\theta)\cos\theta, \quad
y(\theta) = r(\theta)\sin\theta,
\quad 0 \le \theta < 2\pi,
\] where $c_1 = 2,\ c_2 = 0.2,\ c_3 = 9$. The values of the statistical parameters are maintained as $\sigma=0.05$ and $\ell=1$. We choose $N_r=10$ in flower settings such that $n_r=2\times10+1=21$. Since the shape is more complex, $\alpha$ is set to 0.5 at first. The maximum wavenumber is kept consistent with the previous models at $k_{\max}=8$, i.e. $k_j = j,\   j = 1, \dots, n_K$ where $ n_K = 8$. A set of random realizations for the flower-shaped scatterer is presented in \autoref{fig:randomflo}. \autoref{fig:sampleflo} presents a random realization of the flower-shaped scatterer, featuring both its exact geometry and the corresponding far-field pattern on a polar plot. 
For the flower shape, we again first employ a set of 20 samples to examine the numerical outcomes. \autoref{flowerrla} depicts the inversion process of our algorithm applied to an individual flower-shaped sample. \autoref{pic:allflowsamp}, in turn, presents the ensemble of final reconstruction results for all 20 samples at the highest frequency.
\begin{figure}[htbp]
    \centering
    \begin{subfigure}{0.35\textwidth}
        \includegraphics[width=\linewidth]{ 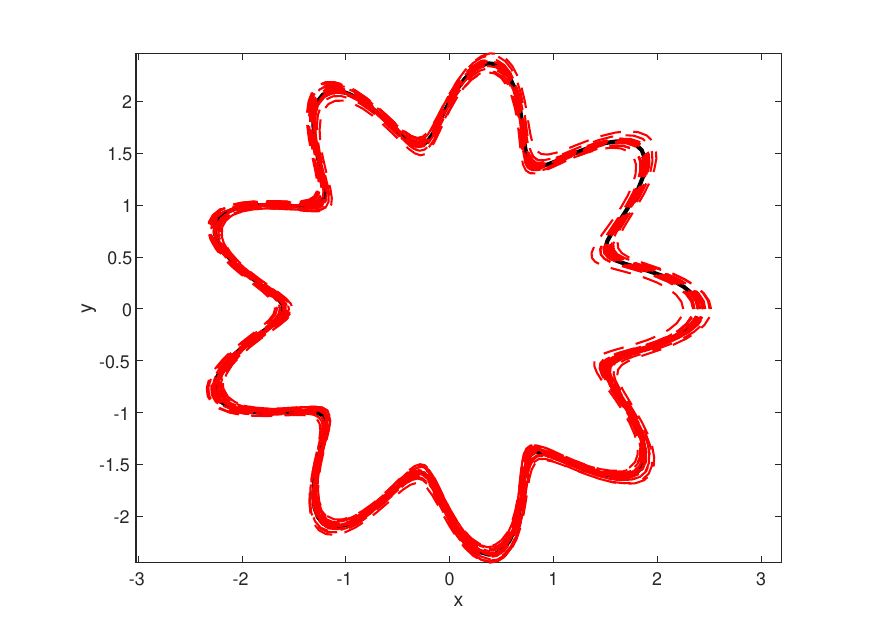}
        \caption{\small random realizations of flower}
        \label{fig:randomflo}
    \end{subfigure}%
    \hspace{2cm} 
     \begin{subfigure}{0.35\textwidth}
        \includegraphics[width=\linewidth]{ 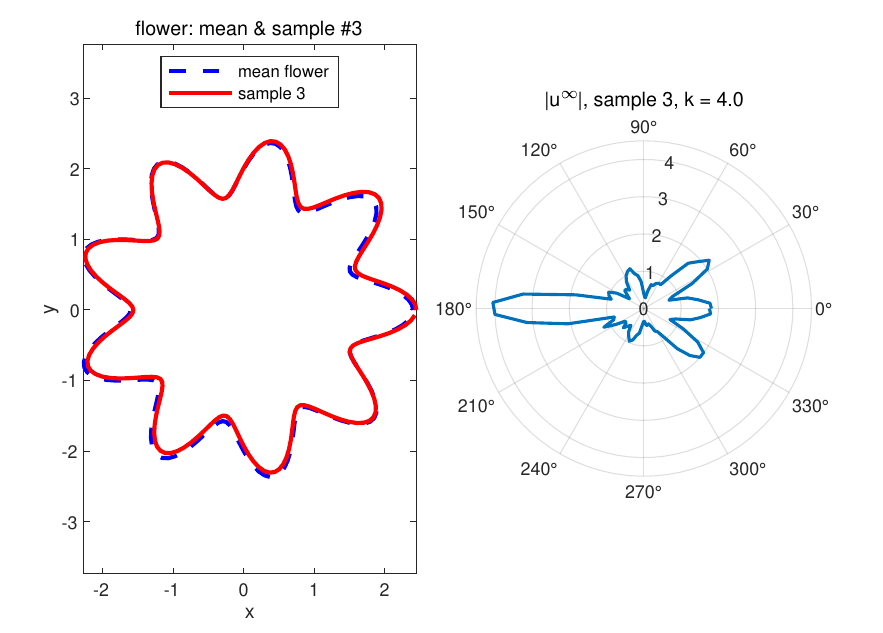}
        \caption{\small shape  and far-field data} 
        \label{fig:sampleflo}
    \end{subfigure}
    \caption{\small \subref{fig:randomflo}: random realization with 20 samples. dashed line: random flower-shaped scatterer; solid line: standard flower-shaped scatterer. \subref{fig:sampleflo}: geometric visualization of the $3$-th random sample and its corresponding far-field pattern data. left:  dashed line represents the circle; solid line represents the true scatterer shape. right: far-field pattern $u^\infty(\hat{x}_m,d,k=4)$ for sample $3$ plotted in polar coordinates.}
    \label{fig:randomf}
    \end{figure}

  \begin{figure}[t]
\begin{center}
    \begin{overpic}[width=1\textwidth, tics=10]{ 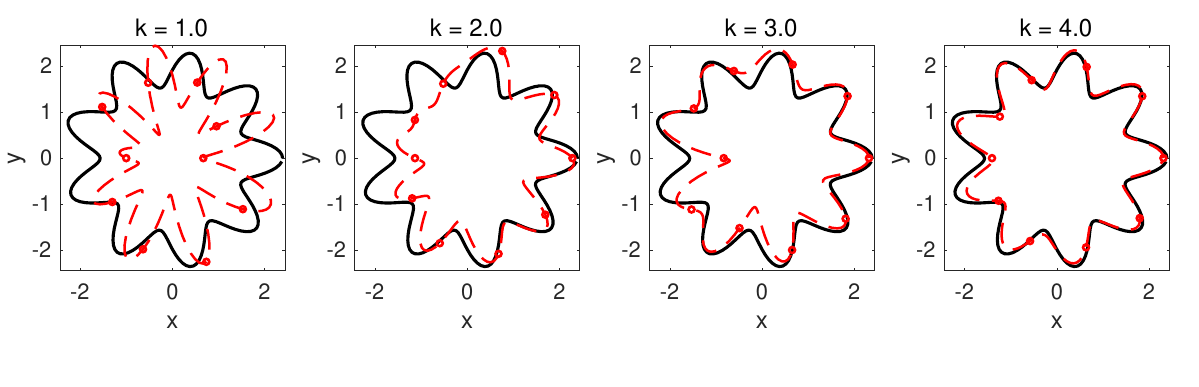}
        \put (30,32) {\scriptsize {\bf Geometric shape inversion at low frequency}}
    \end{overpic}
\end{center}
\vspace{-0.5cm}
 \begin{center}
    \begin{overpic}[width=1\textwidth, tics=10]{ 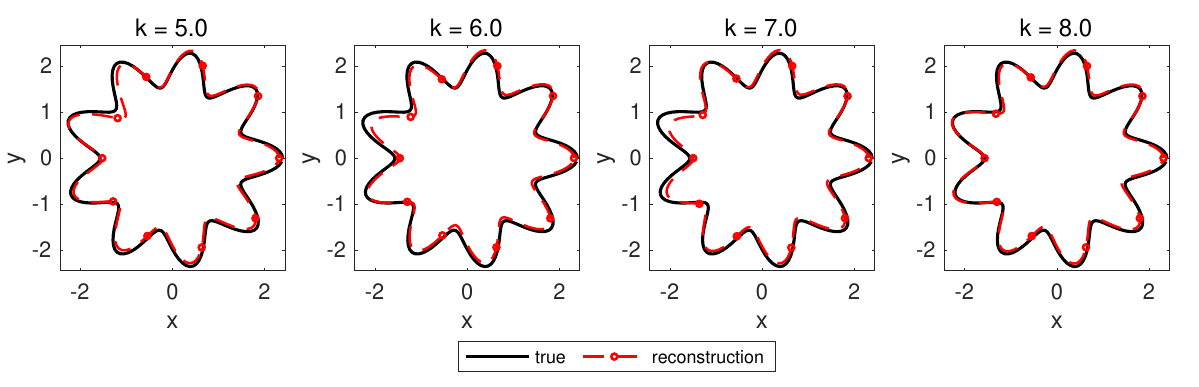}
        \put (30,32) {\scriptsize {\bf Geometric shape inversion at high frequency}}
    \end{overpic}
\end{center}
      \vspace{-0.4cm}
    \caption{visualization of Algorithm \ref{alg:mean-shape-random} executed for a single flower-shaped sample.}
    \label{flowerrla}
      \end{figure}

 \begin{figure}[t]
    \begin{center}
    \begin{overpic}[width=1\textwidth,tics=10]{ 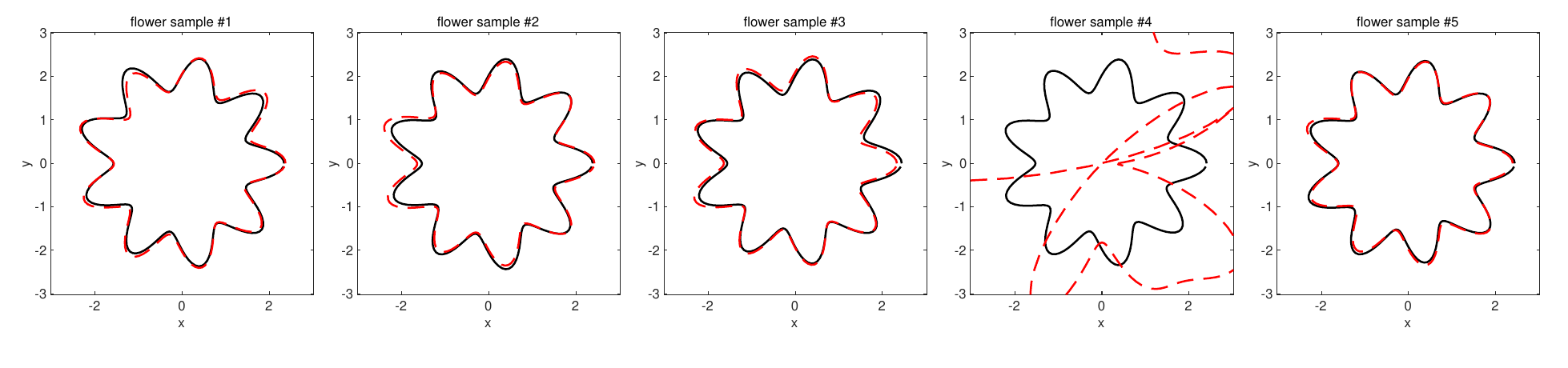}
    \end{overpic}
    \end{center}
    \vspace{-0.8cm}
    \begin{center}
        \begin{overpic}[width=1\textwidth,tics=10]{ 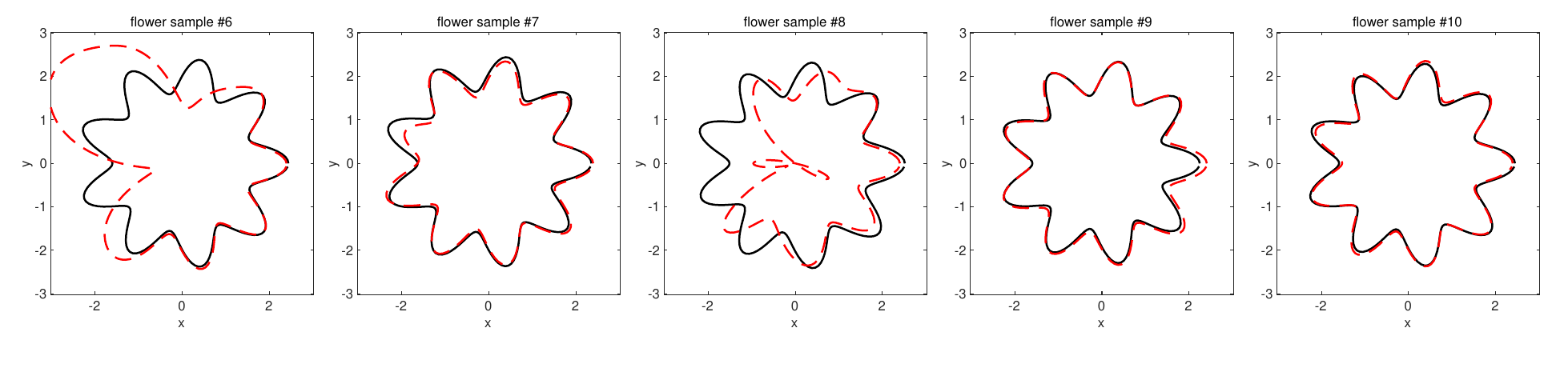}
      \end{overpic}
    \end{center}
    \vspace{-0.8cm}
\begin{center}
    \begin{overpic}[width=1\textwidth, tics=10]{ 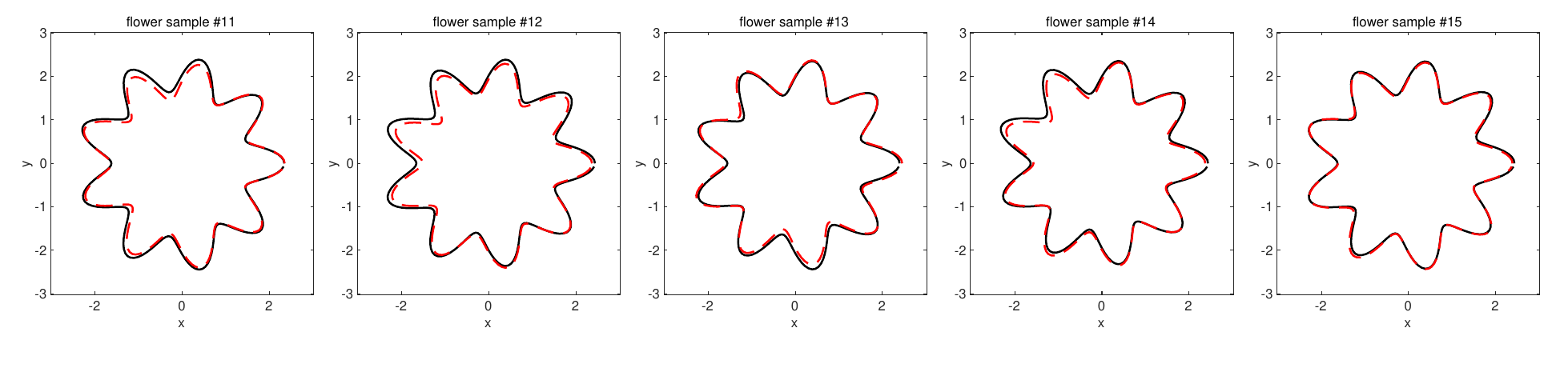}
    \end{overpic}
\end{center}
\vspace{-0.6cm}
 \begin{center}
    \begin{overpic}[width=1\textwidth, tics=10]{ 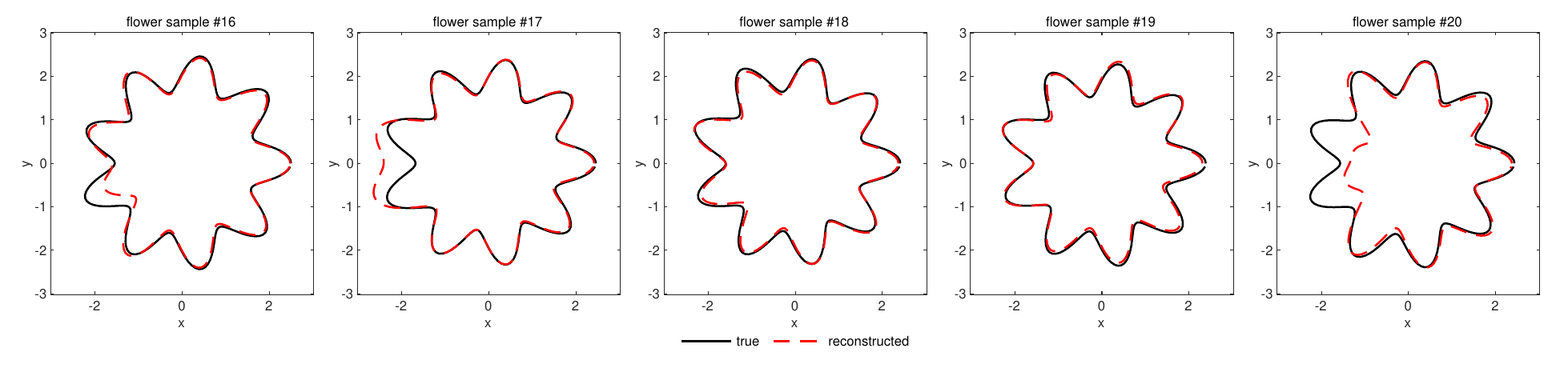}
    \end{overpic}
\end{center}
      \vspace{-0.2cm}
    \caption{ \small reconstructions of all 20 flower-shaped samples following the execution of Algorithm \ref{alg:mean-shape-random}.}
    \label{pic:allflowsamp}
      \end{figure}

By comparing the reconstruction results in \autoref{pic:allcirsamp} and \autoref{pic:allflowsamp}, it is clear that samples from complex geometries are more susceptible to failed reconstructions in the first shape inverting stage than those from simpler shapes, i.e., the number of anomalous samples rises with geometric complexity.  In spite of this growth, the proportion of anomalous samples constitutes a minority within the total set, indicating a relative robustness of the first stage in geometry recovery. While regularization-based correction (e.g., scaling up the penalty term $\alpha$ as we demonstrated in \autoref{example:circle} for the circular geometry) remains feasible for mitigating anomalous realizations, its effectiveness is markedly limited in the present flower-shaped configuration: sometimes even a larger regularization parameter ($\alpha=1 \ \text{or}\  5$) may also fail to recover a stable and physically consistent solution.

In light of the diminishing returns of aggressive regularization, we adopt an alternative strategy to mitigate occasional divergence in the multi-frequency continuation: rather than enforcing convergence for every realization, we perform an automated quality-control (QC) screening and restrict the subsequent statistical inversion to a subset of reliable reconstructions. Specifically, for each sample we quantify (i) the continuation stability across the last few frequency steps by the relative update 
\begin{equation}\label{QC}
    \eta_s^{(\max)} := \max_{j=n_K-m,\dots,n_K-1}\frac{\|r^{(s)}(k_{j+1})-r^{(s)}(k_j)\|_2}{\|r^{(s)}(k_{j+1})\|_2},
\end{equation}
and (ii) the data consistency by the final objective value $G^{(s)}_{\mathrm{final}}:=G^{(s)}_\gamma(k_{n_K})$, then we  discard outliers whose $\eta_s^{(\max)}$ or $G^{(s)}_{\mathrm{final}}$ lies in the upper tail of the empirical distribution. The stability metric $\eta_s^{(\max)}$ reflects the continuation principle that reconstructions should change mildly when the frequency is increased by a small increment while the final objective value $G^{(s)}_{\mathrm{final}}$ measures data consistency at the highest frequency, here ``data consistency" refers to the agreement between the measured far-field data and the model prediction associated with the reconstructed shape at $k_{n_K}$. For well-reconstructed samples, both $\eta_s^{(\max)}$ and $G^{(s)}_{\mathrm{final}}$ are expected to be small, hence this statistically defined screening rule selects a subset of reconstructions that are simultaneously stable across frequencies and consistent with the measured data.  Applying the above QC rule to the flower experiment showed in \autoref{pic:allflowsamp} with quantile thresholds chosen as $q_\eta=0.85$ and $q_G=0.90$ in that example,  the outcome leaves the well-behaved subset of samples $\#1$, $\#2$, $\#3$, $\#5$, $\#7$, $\#9$, $\#10$, $\#11$, $\#12$, $\#13$, $\#14$, $\#15$, $\#18$, $\#19$. Samples exhibiting extreme deviations (e.g.,$ \#4,\#6,\#8$) and mild but persistent inconsistencies (e.g., $\#16,\#17,\#20$) in \autoref{pic:allflowsamp} are excluded. Using the selected samples, we compute the empirical mean shape; the relative $L^2$ error between the resulting mean radius and the true radius is $6.649\times 10^{-3}$, see \autoref{fig:meanflo}. The corresponding reconstructed KL eigenvalues and the recovered covariance parameters $\sigma$ and $\ell$ are reported in \autoref{goodfinal1} and \autoref{tab:flower_val}, respectively. Finally, we observe that geometric complexity has a more pronounced impact on higher-order (smaller) KL modes: while the first few modes are captured with reasonable accuracy, noticeable discrepancies appear in later modes (e.g., the fifth and sixth), cf. \autoref{goodfinal1}. This is consistent with our analysis that the recovery of high-frequency spectral content is the most sensitive part of the pipeline; residual inaccuracies in these components propagate to the estimation of $\sigma$ and $\ell$, as reflected in \autoref{tab:flower_val}.

\begin{table}[htbp]
\centering
\begin{tabular}{c|cccc}   
\hline
 & True (Preset)  & Reference  & \bf{Estimated (4 modes)} & \bf{Estimated (6 modes)}  \\
\hline
$\sigma$  & 0.05 &0.0506 & \bf{0.0540} &  \bf{0.0524} \\
$\ell$  & 1  &  1.0634 & \bf{1.0107} & \bf{0.7660}\\
\hline
\end{tabular}
\caption{comparison of $\sigma$ and $\ell$ inversion results with different modal configurations.}
\label{tab:flower_val}
\end{table}

 \begin{figure}[htbp]
    \centering
    \begin{subfigure}{0.35\textwidth}
        \includegraphics[width=\linewidth]{ 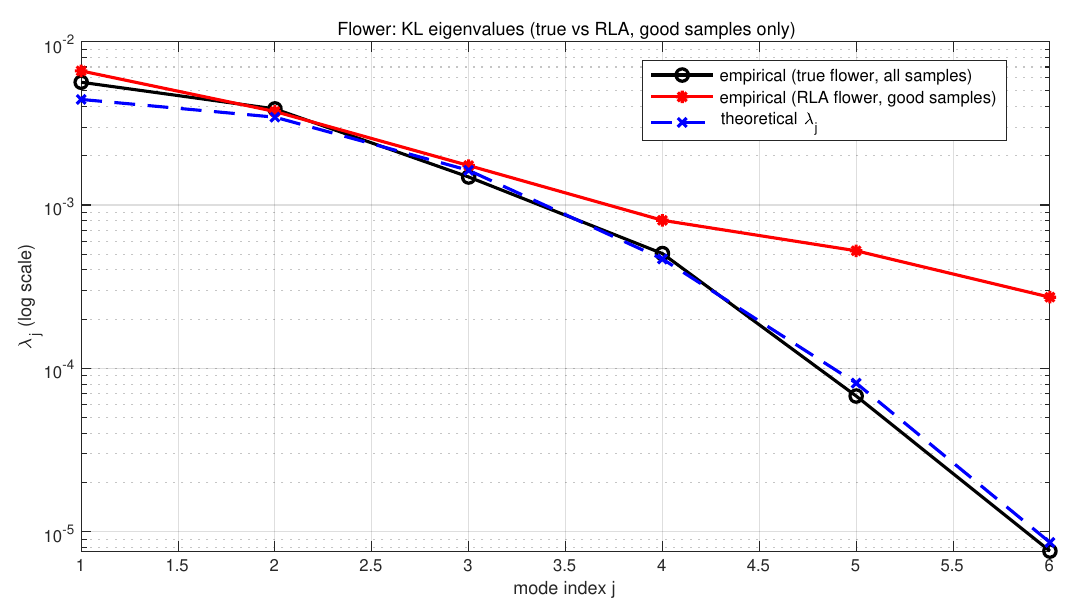}
        \caption{\small $\lambda_j$ from $14$  samples, $k_{n_K}=8$ }
        \label{goodfinal1}
    \end{subfigure}%
     \begin{subfigure}{0.35\textwidth}
        \includegraphics[width=\linewidth]{ 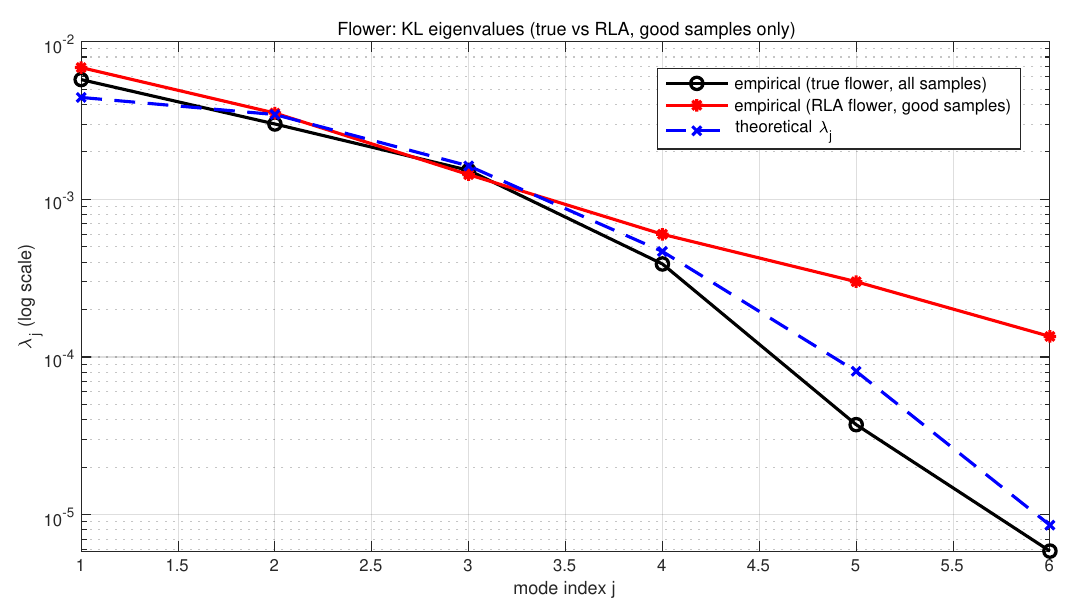}
        \caption{\small  $\lambda_j$  from 60 samples, $k_{n_K}=8$ } 
        \label{fig:goodfinal100}
    \end{subfigure}
     \begin{subfigure}{0.35\textwidth}
        \includegraphics[width=\linewidth]{ 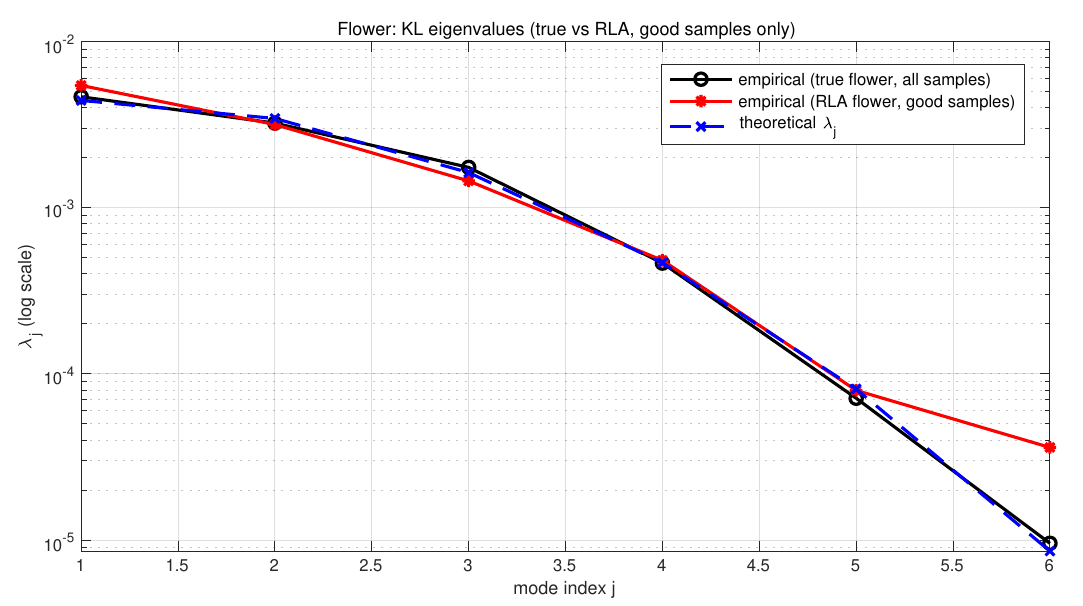}
        \caption{\small $\lambda_j$ from 20 samples, $k_{n_K}=20$}
        \label{goodfinalmoreK}
    \end{subfigure}%
    \caption{\small random flower numerical examples: blue dashed line: theoretical $\lambda_j$ values; black solid line: eigenvalues of the discrete covariance matrix from the true stochastic fluctuations; red solid line: eigenvalues of the reconstructed discrete covariance matrix from the algorithm. \subref{fig:finalcir}: reconstructed $\lambda_j$ from 14 selected samples among 20 inversions showed in \autoref{pic:allflowsamp}. \subref{fig:finalancir}: reconstructed $\lambda_j$ from 60 selected samples among 100 inversions under the same experimental setting as \subref{goodfinal1}. \subref{goodfinalmoreK}:  reconstructed $\lambda_j$ from 20 selected samples among 60 inversions under the setting: maximum wavenumber $k_{n_K}=20$, frequency index $N_r=14$, number of discrete iteration truncation terms $n_r=2N_r+1=29$ and regularization parameter $\alpha=0.8$.}
    \label{fig:cirrepair}
    \end{figure}

The reconstruction error appeared in high-frequency eigenvalues  in \autoref{goodfinal1} stems primarily from the insufficient capture of  high-frequency details in the first-stage geometric inversion rather than from high-frequency noise contamination due to stochastic fluctuations. To illustrate this point,  we repeat the experiment with 100 independent realizations under the same settings as  in \autoref{goodfinal1} and apply the same automated QC screening described above to identify a subset of reliable reconstructions. Ultimately, we retain 60 samples with small continuation updates $\&$ small final objective values (i.e., small $\eta_s^{(\max)}$ and $G^{(s)}_{\mathrm{final}}$ under the chosen quantile thresholds) and  recompute the empirical KL spectrum through these selected samples. The resulting eigenvalue comparison is shown in \autoref{fig:goodfinal100}. We observe that the discrepancy in the high-frequency range does not decrease appreciably, confirming our earlier assertion that the dominant source of error is the limited resolution of the first-stage geometric reconstruction at high frequencies rather than sample-to-sample stochastic noise.


In order to enhance the accuracy of the first-stage geometric inversion and capture more high-frequency details of the flower-shaped scatterer, we increase the maximum wavenumber $k_{n_K}$ from 8 to 20 and raised the discrete truncation iteration parameter $p$ from the original $2\times10+1=21$ terms to $2\times14+1=29$ terms. The corresponding parameter $\alpha$ is set to 0.8. We tested a total of 60 samples, among which 20 samples were justified by  QC screening as reliable reconstructions. A comparison of the eigenvalues computed from these 20 samples is presented in \autoref{goodfinalmoreK}. It can be clearly observed that the estimated eigenvalues converge toward their true values as the accuracy of geometric inversion improves. Among the first six eigenvalues, the first five are successfully reconstructed with high fidelity, only the smallest eigenvalue (the sixth one), which is on the order of magnitude between $10^{-4}\sim10^{-5}$ still exhibits a  slight deviation. With more accurate eigenvalue estimates, the precision in estimating the corresponding parameters $\sigma$ and $\ell$ is consequently improved, as presented in  \autoref{tab:flower_val1}. This result suggests that the proposed  inversion framework does not require a large sample size: in practice, about 20–30 samples are often sufficient to recover the covariance parameters with high accuracy. Such efficiency, however, hinges on  stage 1 reconstructions with sufficient fidelity and resolution to reliably capture the scatterer’s fine-scale features.  This observation naturally points to two  directions for the future work: developing more effective shape inversion techniques that improve high-frequency geometric fidelity and exploring alternative stochastic modeling and reconstruction frameworks for uncertainty quantification beyond the present Gaussian/KL setting. We leave these directions to the subsequent investigations.




\begin{table}[htbp]
\centering
\begin{tabular}{c|cccc}   
\hline
 & True (Preset)  & Reference  & \bf{Estimated (4 modes)} & \bf{Estimated (6 modes)}  \\
\hline
$\sigma$  & 0.05 &0.0503 & \bf{0.0506} &  \bf{0.0506} \\
$\ell$  & 1  &  1.0002 & \bf{1.0157} & \bf{1.0157}\\
\hline
\end{tabular}
\caption{inversion of $\sigma$ and $\ell$ with different modal configurations under $k_{n_K}=20$.}
\label{tab:flower_val1}
\end{table}

\begin{remark}
    \label{remark2}
    For samples exhibiting mild deviations (like $\#16, \#17,\#20$ in \autoref{pic:allflowsamp}),  these slight discrepancies primarily stem from the well–known  sensitivity of the inverse obstacle scattering problems with a single  incident direction, especially when the obstacle geometry is relatively complex.  Since this study employs a single fixed incident direction $d =(-1, 0)$ as stated earlier, the shape reconstruction quality is different between the illuminated (or front) side and the shadowed (or back) side, as evidenced by \autoref{pic:allcirsamp} and \autoref{pic:allflowsamp}. The recovery  on the illuminated side of the scatterer is slightly more accurate than that on the shadowed  side. Despite this, the overall shape is well-recovered within acceptable error bounds for most samples and, more importantly, for simpler geometries such as the circle and pear, these shadow-side inaccuracies  do not prevent accurate inversion of KL eigenvalues and the associated statistical parameters. 
    
    To further enhance the reconstruction accuracy, particularly for those scatterers with intricate geometry shapes, two natural extensions can be 
considered.  The first is to employ multi-angle illuminations which will in turn enrich the available scattered–field data and thereby improve the overall stability and resolution of the reconstruction, in particular on the shadowed side. This scenario can be further pursued  in future work. The second one, as noted by Sini et al. in \cite{sini2012inverse}, involves introducing a known reference scatterer (e.g., a circle of prescribed geometry) together  with the unknown obstacle, leading to a combined scattering configuration. This auxiliary scatterer acts as an intermediary whose scattered field helps compensate for the lack of information for the shadowed part of the unknown obstacle. However, the forward problem for such a coupled system requires solving a large boundary integral equation for the density function $\phi$ and the resulting system is typically more ill–conditioned 
due to multiple scattering interactions and geometric complexity. Resolving this numerical challenge constitutes another interesting and promising direction for future investigation.
\end{remark}
 

\section{Concluding remarks}
\label{sec:con}
In this work we proposed and analyzed a two-stage framework for inverse scattering by randomly perturbed obstacles, in which a multi-frequency shape reconstruction is first carried out for each realization and is then followed by a statistical inversion that estimates the covariance structure of the underlying geometric fluctuations. A central modeling ingredient is a covariance model on the angular domain that is both physically consistent (correlations decay with the minimal angular separation) and mathematically well-defined on the $2\pi$-periodic setting. This construction yields a tractable spectral characterization: the associated covariance operator is diagonalized by Fourier modes, leading to explicit formulas for the KL eigenvalues and enabling the recovery of covariance parameters such as $\sigma$ and $\ell$ from finitely many samples. On the theoretical side, we discuss the uniqueness and convergence aspects of the underlying inverse reconstruction problem, analyze how first-stage shape errors affect the second-stage KL eigenvalues and the recovered covariance parameters. Numerical experiments demonstrate that the proposed pipeline captures the dominant KL modes and produces stable parameter estimates across different geometries, while higher-order modes are more sensitive to geometric complexity and to the limited high-frequency resolution of the first-stage inversion. From an analysis perspective, it would be valuable to further integrate first-stage convergence results for multi-frequency shape reconstruction into the second-stage statistical guarantees, especially under the intrinsic ill-posed settings  of inverse scattering. From a modeling standpoint, extending the covariance class beyond isotropic, stationary kernels on the angular domain while preserving physical interpretability and spectral tractability would broaden applicability. On the computational side, developing adaptive frequency continuation and uncertainty-aware regularization strategies may improve the recovery of high-frequency KL modes and reduce the need for post hoc screening. We expect these directions to further strengthen the practical impact of the proposed two-stage stochastic inversion framework.



\appendix\label{appendix}
\section{\textbf{Proof of \autoref{allpo}}}\label{appen:Gaussexi}
\begin{proof}
    \ref{stmt:i}$\Longrightarrow$\ref{stmt:ii}: This is immediate from the definition of positive semi-definiteness in \eqref{pd}.

\ref{stmt:ii}$\Longrightarrow$\ref{stmt:iii}: Fix $ N_\theta\ge 1$ and let $\theta_m=2\pi m/N_\theta$. The matrix $K^{(N_\theta)}$ with entries
\[
K^{(N_\theta)}_{mn}=C(\theta_m-\theta_n)
\]
is circulant, hence it is diagonalized by the discrete Fourier vectors, and its eigenvalues are 
\[
\mu^{(N_\theta)}_j=\sum_{m=0}^{N_\theta-1} C(\theta_m)\,e^{-ij\theta_m},\qquad j=0,\dots,N_\theta-1.
\]
By (2), $K^{(N_\theta)}\succeq 0$ is positive semi-definite, so $\mu^{(N_\theta)}_j\ge 0$ for all $j$.
For each fixed $j\ge 0$, dividing by $N_\theta$ and letting $N_\theta\to\infty$, the above sum is a Riemann sum, hence
\[
\frac{1}{N_\theta}\mu^{(N_\theta)}_j \longrightarrow \frac{1}{2\pi}\int_0^{2\pi} C(t)e^{-ijt}\,dt=:c_j .
\]
Therefore $c_j\ge 0$. Since $C(t)$ is real and even, $c_j\in\mathbb{R}$ and
\[
c_0=a_0,\qquad c_j=\frac{a_j}{2}\ \ (j\ge 1),
\]
where $a_j$ are the cosine-series coefficients in \ref{Fourier}–\ref{coeff}. Thus $a_j\ge 0$ for all $j\ge 0$.
    \ref{stmt:iii}$\Longrightarrow$\ref{stmt:i}: Given $C(t) = a_0 + \sum_{j=1}^{\infty} a_j \cos(jt), \  a_j \geq 0, \  \sum_{j \geq 0} |a_j| < \infty$, we aim to prove for any finite set of points $\{\theta\}_{i=1}^n$ and any real coefficients $\{b_i\}_{i=1}^n$, the following inequality holds:\[
    S:= \sum_{i,k=1}^{n} b_i b_k \, C(\theta_i - \theta_k) \geq 0.
    \]
    Substitute $S$ into the cosine expansion of $C(t)$ yields:
    \[S  =\sum_{i, k} b_i b_k\left(a_0+\sum_{j=1}^{\infty} a_j \cos \left(j\left(\theta_i-\theta_k\right)\right)\right) 
 =a_0 \sum_{i, k} b_i b_k+\sum_{j=1}^{\infty} a_j \sum_{i, k} b_i b_k \cos \left(j\left(\theta_i-\theta_k\right)\right) .
\]
    Since $\sum_{i,k} b_i b_k = \left( \sum_i b_i \right)^2$,  the first term simplifies to  $a_0 \left( \sum_{i=1}^{n} b_i \right)^2$. For the second term, recall $\cos x = \frac{e^{ix} + e^{-ix}}{2} = \Re(e^{ix})$ where $\Re$ denotes the real part. This identity allows us to write $\cos(j(\theta_i - \theta_k)) = \Re\left(e^{ij(\theta_i - \theta_k)}\right)$, thus we have
    \[T_j := \sum_{i,k} b_i b_k \cos(j(\theta_i - \theta_k))
    = \sum_{i,k} b_i b_k \Re\left(e^{ij(\theta_i - \theta_k)}\right) 
    = \Re\left( \sum_{i,k} b_i b_k e^{ij(\theta_i - \theta_k)} \right).\]
Separate the exponential terms: $e^{ij(\theta_i - \theta_k)} = e^{ij\theta_i} e^{-ij\theta_k}$, then 
\[\sum_{i,k} b_i b_k e^{ij(\theta_i - \theta_k)} 
= \sum_{i,k} b_i b_k e^{ij\theta_i} e^{-ij\theta_k} \\
= \left( \sum_i b_i e^{ij\theta_i} \right) \left( \sum_k b_k e^{-ij\theta_k} \right).\]
    Denote $z_j := \sum_{i=1}^{n} b_i e^{ij\theta_i}$, With this notation, the preceding expression becomes $z_j\bar{z}_j$. Consequently, we have $  T_j = \Re(z_j \overline{z_j}) = \Re(|z_j|^2) = |z_j|^2$.
    Thus, we finally obtain:
    \[
    \sum_{i,k} b_i b_k \cos(j(\theta_i - \theta_k)) = \left| \sum_{i=1}^{n} b_i e^{ij\theta_i} \right|^2.
    \]
    Plugging the result into $S$ gives:
    \[S = a_0 \left( \sum_i b_i \right)^2 + \sum_{j=1}^{\infty} a_j \left| \sum_{i=1}^{n} b_i e^{ij\theta_i} \right|^2\geq0.\]
   Therefore, $C(t)$ is positive semi-definite.  \qedhere 
   
\end{proof}

\section{\textbf{Derivation of the relation between $\widetilde{a}_j$, $\lambda_j$ in \eqref{guanxi1} and \eqref{coe}}}\label{relationn}
In this part we derive the relationship between the Fourier expansion coefficients $c_n$ of a periodic (even) real stationary covariance function ${C}_{\delta r}(t)$ in \eqref{acc}, the even function expansion coefficients $\widetilde{a}_n$ and  the eigenvalues $\lambda_n$ of the operator $\mathcal{C}$ defined via ${C}_{\delta r}(t)$.

By definition, ${C}_{\delta r}(t)$ admits a Fourier expansion:
\begin{equation}\label{acc}
 {C}_{\delta r}(t) = \sum_{n \in \mathbb{Z}} c_n e^{i n t},\qquad  c_n = \frac{1}{2\pi} \int_0^{2\pi} {C}_{\delta r}(t) e^{-i n t} \, dt.    
\end{equation}
Since $C(t)$ is real and even, it can be deduced that
\[
c_{-n} = \overline{c_n} = c_n, \quad c_n \in \mathbb{R}, \quad \forall n \in \mathbb{Z}.
\]
Thus
\[
\begin{aligned}
{C}_{\delta r}(t) =\sum_{n \in \mathbb{Z}} c_n e^{i n t}=c_0+\sum_{j=1}^{\infty} c_j e^{i j t}+\sum_{j=1}^{\infty} c_{-j} e^{-i j t} 
 &=c_0+\sum_{j=1}^{\infty} c_j\left(e^{i j t}+e^{-i j t}\right) \\
&=c_0+2 \sum_{j=1}^{\infty} c_j \cos (j t) .
\end{aligned}
\]
If we direct expand ${C}_{\delta r}(t)$ as $ {C}_{\delta r}(t) = \widetilde{a}_0 + \sum_{j=1}^{\infty} \widetilde{a}_j \cos(j t), \  t \in [0, 2\pi],$ then we have
\[
\widetilde{a}_{0}=c_{0},\qquad \widetilde{a}_{j}=2c_{j}\quad(j\geq 1).
\]
For $\lambda_0$, we know it satisfies $\mathcal{C}\varphi_0=\lambda_0\varphi_0$. Easy to show that $\lambda_0=\int_0^{2\pi}{C}_{\delta r}(t)dt$ since $\varphi_0$ is a constant function, moreover $\widetilde{a}_0=\frac{1}{2\pi}\int_{0}^{2\pi} {C}_{\delta r}(t)dt$, hence 
    $\lambda_0=2\pi \widetilde{a}_0=2\pi c_0$.
For $\lambda_j$, we focus on the cosine mode $\varphi_{j, c}(\theta)=\frac{1}{\sqrt{\pi}} \cos (j \theta)$, the same derivation can be applied to the sine mode. By definition
\[
 (\mathcal{C} \phi_{j,c})(\theta) = \int_0^{2\pi} {C}_{\delta r}(|\theta - \theta'|) \frac{1}{\sqrt{\pi}} \cos(j \theta') \, d\theta'\xlongequal{\tau=\theta-\theta^\prime}  \frac{1}{\sqrt{\pi}} \int_0^{2\pi} {C}_{\delta r}(\tau) \cos\big(j(\theta - \tau)\big) \, d\tau.
\]
With the identity $  \cos(j(\theta - \tau)) = \cos(j\theta) \cos(j\tau) + \sin(j\theta) \sin(j\tau)$, we have
\[
\left(\mathcal{C} \phi_{j, c}\right)(\theta)=\frac{\cos (j \theta)}{\sqrt{\pi}} \int_0^{2 \pi} {C}_{\delta r}(\tau) \cos (j \tau) d \tau+\frac{\sin (j \theta)}{\sqrt{\pi}} \int_0^{2 \pi} {C}_{\delta r}(\tau) \sin (j \tau) d \tau.
\]
${C}_{\delta r}(\tau)$ is even and $\sin(j\tau)$ is odd, therefore $ \int_0^{2\pi} {C}_{\delta r}(\tau) \sin(j\tau) \, d\tau = 0$, which means
\[
(\mathcal{C} \phi_{j,c})(\theta) = \frac{\cos(j\theta)}{\sqrt{\pi}} \int_0^{2\pi} {C}_{\delta r}(\tau) \cos(j\tau) \, d\tau.
\]
Thus $\lambda_j = \int_0^{2\pi} {C}_{\delta r}(\tau) \cos(j\tau) \, d\tau, \  j \geq 1,$ yields\[\lambda_j=\pi \widetilde{a}_j=2\pi c_j,\quad j\geq 1.\]

\section{\textbf{Derivation of the Approximate Relation \eqref{relation}}}
\label{appen:relation}
In  a one-dimensional periodic setting, we focus on a single period cell $[0,2\pi)$. A (periodically) stationary covariance function means that the kernel  depends only on the relative displacement between two points and this relative displacement is meaningful only modulo $2\pi$: for any integer $m$, the displacements $x-y$ and $x-y+2\pi m$ are equivalent and represent the same periodic translation. On a $2\pi$-periodic domain, relative displacements are understood modulo $2\pi$, so one may choose any length-$2\pi$ interval to represent them.  For  convenience, we choose the symmetric representative interval $[-\pi,\pi)$ which corresponds to the shortest periodic displacement.  Assume the Gaussian   covariance kernel $C_f$ is periodically stationary, i.e., $C_f(x,y)=C_f(x-y)$ where $x-y:= x-y\ \text{mod}\  2\pi$ denotes the periodic relative displacement. In $[-\pi,\pi)$, $C_f$ takes the regular Gaussian form
\[
C_f(t)=\sigma^2 \exp\!\left(-\frac{t^2}{\ell^2}\right), \qquad 0<\ell\ll 2\pi.
\]
Then the associated covariance operator $\mathcal{C}_f$ is defined by:
\begin{equation}\label{axi}
    (\mathcal{C}_f \varphi)(x) = \int_0^{2\pi} C_f(x - y) \varphi(y) \, dy.
\end{equation}
\noindent  Accordingly, introducing the periodic displacement variable $t=y-x$ represented in $[-\pi,\pi)$, the operator in \eqref{axi} can be written as
\begin{equation}\label{axi2}
     (\mathcal{C}_f \varphi)(x) = \int_{-\pi}^{\pi} C_f(-t) \varphi(x+t) \, dt\xlongequal{C_f(t) =\ C_f(-t) \ \text{on}\  [-\pi,\pi)}\int_{-\pi}^{\pi} C_f(t) \varphi(x+t) \, dt.
\end{equation}
From \cite{xiu2010numerical} we know that the eigenfunctions towards the operator $\mathcal{C}_f$ is the Fourier basis functions: $\phi_j(x) = c e^{i j x}$,  $j \in \mathbb{Z}$
where $c$ is the orthonormalization constant for $\phi_j(x)$. Substituting $\phi_j$ into the operator $\mathcal{C}_f$ \eqref{axi2} yields
\[ 
 (\mathcal{C}_f\phi_j)(x)
  = \sigma^2\int_{-\pi}^{\pi} e^{-t^2/\ell^2} ce^{ij(x+t)}\,dt=
   c\sigma^2e^{ijx}\int_{-\pi}^{\pi} e^{-t^2/\ell^2} e^{ijt}\,dt.
\]
Therefore $\phi_j$ is an eigenfunction with eigenvalue
\begin{equation}\label{eq:lambda-j-periodic}
\lambda_j^{\mathrm{mod}}=\sigma^2\int_{-\pi}^{\pi} e^{-t^2/\ell^2}e^{ijt}\,dt .
\end{equation}
In particular, $\lambda_j$ depends only on $j$ and is independent of $x$. Based on the  assumption that the correlation length $\ell$ is much smaller than the
period, i.e.\ $\ell\ll 2\pi$, the Gaussian factor
$e^{-t^2/\ell^2}$ is exponentially small for $|t|\ge\pi$.
Writing
 $ I_j^{(\mathrm{per})}
  := \int_{-\pi}^{\pi} e^{-t^2/\ell^2} e^{ijt}\,dt$,
 $ I_j^{(\mathbb R)}
  := \int_{-\infty}^{\infty} e^{-t^2/\ell^2} e^{ijt}\,dt$,
we have
\[
  I_j^{(\mathbb R)} - I_j^{(\mathrm{per})}
  = \int_{|t|\ge\pi} e^{-t^2/\ell^2} e^{ijt}\,dt,
\]
and thus, by the triangle inequality,
  $\bigl|I_j^{(\mathbb R)} - I_j^{(\mathrm{per})}\bigr|
  \le \int_{|t|\ge\pi} e^{-t^2/\ell^2}\,dt
  = 2\int_{\pi}^{\infty} e^{-t^2/\ell^2}\,dt$,
the Gaussian tail admits the standard bound
\[
  \int_{\pi}^{\infty} e^{-t^2/\ell^2}\,dt
  = \ell\int_{\pi/\ell}^{\infty} e^{-s^2}\,ds
  \le \frac{\ell}{2(\pi/\ell)} e^{-(\pi/\ell)^2}
  = \frac{\ell^2}{2\pi} e^{-\pi^2/\ell^2},
\]
so we obtain the estimate
\begin{equation}\label{eq:Ij-error}
  \bigl|I_j^{(\mathbb R)} - I_j^{(\mathrm{per})}\bigr|
  \le C_{\mathrm{asym}}\,\ell^2 e^{-\pi^2/\ell^2},
\end{equation}
for some constant $C_{\mathrm{asym}}>0$ independent of $j$ (in particular, $C_{\mathrm{asym}}=1/\pi$ works).
Therefore, using \eqref{eq:lambda-j-periodic} and \eqref{eq:Ij-error}, we have
$\lambda_j^{\mathrm{mod}}
  = \sigma^2 I_j^{(\mathrm{per})}
  = \sigma^2 I_j^{(\mathbb R)}
    + \mathcal{O}\!\bigl(\sigma^2\ell^2 e^{-\pi^2/\ell^2}\bigr)$,
The  integral $I_j^{(\mathbb R)}$ can be evaluated explicitly:
\[
  I_j^{(\mathbb R)}
  = \int_{-\infty}^{\infty} e^{-t^2/\ell^2} e^{ijt}\,dt
  = \sqrt{\pi}\,\ell\,e^{-\ell^2 j^2/4}.
\]
Consequently we arrive at the asymptotic formula
\begin{equation}\label{eq:lambda-asymptotic}
  \lambda_j^{\mathrm{mod}}
  = \sqrt{\pi}\,\sigma^2\ell\,e^{-\ell^2 j^2/4}
    + \mathcal{O}\!\bigl(\sigma^2\ell^2 e^{-\pi^2/\ell^2}\bigr).
\end{equation}
In other words, for correlation lengths much smaller than the period
$2\pi$, the eigenvalues of the periodic covariance operator are well
approximated by
\[
  \lambda_j^{\mathrm{mod}} \approx \sqrt{\pi}\,\sigma^2\ell\,e^{-\ell^2 j^2/4},
\]
with an exponentially small error of order
$\mathcal{O}(\sigma^2\ell^2 e^{-\pi^2/\ell^2})$, which is the relation in \eqref{relation}.

\section{\textbf{Proof of Proposition \autoref{relatio}}}
\label{appen:Justification}
\begin{proof}
    Since $\Theta=[0,2\pi)\simeq\mathbb R/(2\pi\mathbb Z)$, our angular setting can also be viewed as a one-dimensional periodic domain analogous to the structure in  \autoref{appen:relation}. With the periodically stationary covariance function $C_{\mathrm{geod}}$, we have:
    \[
C_{\mathrm{geod}}(x,y)  \xlongequal{\text{stationary}}   C_{\mathrm{geod}}(x-y)\xlongequal{x-y=t}C_{\mathrm{geod}}(t).
    \] where $x-y$ is understood as the relative angular displacement modulo $2\pi$. Throughout this section, we adopt this convention.
Its corresponding operator is
    \[
    (\mathcal{C}_{\mathrm{geod}} f)(x) = \int_0^{2\pi} C_{\mathrm{geod}}(x - y ) f(y) \, dy.
    \]
    The Fourier basis functions continue to serve as the eigenfunctions for this operator
    \[
\mathcal{C}_{\mathrm{geod}}\phi_j=\lambda_j\phi_j, \qquad \phi_j(x) = c e^{i j x}, \quad j \in \mathbb{Z}.
\]
where $c$ is the orthonormalization constant for $\phi_j(x)$. 
Hence 
\begin{align*}
  (\mathcal{C}_{\mathrm{geod}}\phi_j)(x)
  &= \int_{0}^{2\pi} C_{\mathrm{geod}}(x-y)\,\phi_j(y)\,dy 
  = \int_{0}^{2\pi} C_{\mathrm{geod}}(x-y)\,
      ce^{ijy}\,dy  \\
  &= e^{ijx}c
     \int_{0}^{2\pi} C_{\mathrm{geod}}(x-y)\,e^{ij(y-x)}\,dy.
\end{align*}
therefore
\begin{equation}\label{eq:lambda-complex}
  \lambda_j
  = \int_{0}^{2\pi} C_{\mathrm{geod}}(x-y)\,e^{ij(y-x)}\,dy.
\end{equation}
View $C_{\mathrm{geod}}$ as a function on $\Theta$, i.e., identify $C_{\mathrm{geod}}(t)$ with its $2\pi$-periodic extension to $\mathbb R$. Since $j\in\mathbb Z$, the factor $e^{ijt}$ is also $2\pi$-periodic, hence the product $t\mapsto C_{\mathrm{geod}}(t)e^{ijt}$ is $2\pi$-periodic. With the change of variables $t = y-x$, \eqref{eq:lambda-complex} turns into
\begin{equation}\label{eq:lambda-fourier}
  \lambda_j
  = \int_{0}^{2\pi} C_{\mathrm{geod}}(t)\,e^{ijt}\,dt=\int_{-\pi}^{\pi} C_{\mathrm{geod}}(t)\,e^{ijt}\,dt=\int_{-\pi}^{\pi} C_{f}(t)\,e^{ijt}\,dt.
\end{equation}
which is the exact result showed in \eqref{eq:lambda-j-periodic}. Following the same process as described in \autoref{appen:relation} will yield the approximation relation. 

By construction, $C_{\delta r}$ is obtained from $C_{\mathrm{geod}}$ by adjusting  its Fourier/cosine coefficients so that all retained coefficients are nonnegative. Since both kernels are periodically stationary on $\Theta$, the associated covariance operators are convolution operators and are diagonalized by the same Fourier basis $\{\phi_j\}_{j\in\mathbb Z}$. In particular, for each $j=0,\ldots,N_{\mathrm{KL}}$, the discrepancy
$\lambda_j(C_{\delta r})-\lambda_j(C_{\mathrm{geod}})$ is exactly the discrepancy between the corresponding coefficients at index j. Therefore, if the correction is applied only to indices $j\ge N_{\mathrm{KL}}$, then the first $N_{\mathrm{KL}}$ eigenvalues are unchanged; otherwise, the perturbation of the retained eigenvalues is controlled directly by the perturbation of the retained coefficients. This shows that the asymptotic relation derived in \autoref{appen:relation} carries over to $C_{\delta r}$ for the modes used in the KL truncation.  \qedhere 

\end{proof}


\section{\textbf{Proof of  \autoref{convergence}}}\label{appen:proofconvergence}
\begin{proof}
    By definition, $\widehat{x}^{(s)} = x^{(s)} + \widetilde{e}^{(s)}$. Substitute it into $\widehat{\Sigma}_N^{\mathrm{rec}}$ yields:
    \[
\begin{aligned}
\widehat{\Sigma}_{N}^{\mathrm{rec}} & =\frac{1}{N-1}\sum_{s=1}^N(x^{(s)}+\widetilde{e}^{(s)})(x^{(s)}+\widetilde{e}^{(s)})^\top \\
 & =\frac{1}{N-1}\sum_{s=1}^N\left[x^{(s)}x^{(s)\top}+x^{(s)}\widetilde{e}^{(s)\top}+\widetilde{e}^{(s)}x^{(s)\top}+\widetilde{e}^{(s)}\widetilde{e}^{(s)\top}\right].
\end{aligned}
\]
Subtract $\widehat{\Sigma}_N=x^{(s)}x^{(s)\top}$ from both sides, we get
\[\widehat{\Sigma}_N^{\mathrm{rec}}-\widehat{\Sigma}_N=\frac{1}{N-1}\left(\sum_{s=1}^Nx^{(s)}\widetilde{e}^{(s)\top}+\sum_{s=1}^N\widetilde{e}^{(s)}x^{(s)\top}+\sum_{s=1}^N\widetilde{e}^{(s)}\widetilde{e}^{(s)\top}\right):=\frac{1}{N-1}\Big(A+B+C\Big),\]
here we introduce the notation $A:=\sum_{s=1}^Nx^{(s)}\widetilde{e}^{(s)\top}$, $B:=\sum_{s=1}^N\widetilde{e}^{(s)}x^{(s)\top}=A^{\top}$, $C:=\sum_{s=1}^N\widetilde{e}^{(s)}\widetilde{e}^{(s)\top}$ for brevity.
For  a matrix $M$, we know $\|M\|_2\leq\|M\|_{F}$, thus 
\[
\left\|\widehat{\Sigma}_N^\mathrm{rec}-\widehat{\Sigma}_N\right\|_2\leq\frac{1}{N-1}{\left(\|A\|_F+\|B\|_F+\|C\|_F\right)}.
\]
For $A$, 
\[
\|A\|_F=\left\|\sum_{s=1}^Nx^{(s)}\widetilde{e}^{(s)\top}\right\|_F\leq\sum_{s=1}^N\|x^{(s)}\widetilde{e}^{(s)\top}\|_F=\sum_{s=1}^N\|x^{(s)}\|_2\|\widetilde{e}^{(s)}\|_2.
\]
By Cauchy–Schwarz,
\[
\sum_{s=1}^N\|x^{(s)}\|_2\|\widetilde{e}^{(s)}\|_2\leq\left(\sum_{s=1}^N\|x^{(s)}\|_2^2\right)^{1/2}\left(\sum_{s=1}^N\|\widetilde{e}^{(s)}\|_2^2\right)^{1/2}.
\]
According to \eqref{assumption1}, $\sum_{s=1}^N\|x^{(s)}\|_2^2\leq NM_x,\  \sum_{s=1}^N\|\widetilde{e}^{(s)}\|_2^2\leq N\varepsilon^2$, then we derive that 
\[
\|A\|_F\leq\sqrt{NM_x}\sqrt{N\varepsilon^2}=N\sqrt{M_x}\mathrm{~}\varepsilon.
\]
For $B=\sum_{s=1}^N\widetilde{e}^{(s)}x^{(s)\top}$, applying the same computation  will lead to 
\[
\|B\|_F\leq N\sqrt{M_x}\mathrm{~}\varepsilon.
\]
As for $C$, we have the following estimate:
\[
\|C\|_F=\left\|\sum_{s=1}^N\widetilde{e}^{(s)}\widetilde{e}^{(s)\top}\right\|_F\leq\sum_{s=1}^N\|\widetilde{e}^{(s)}\widetilde{e}^{(s)\top}\|_F=\sum_{s=1}^N\|\widetilde{e}^{(s)}\|_2^2\leq N\varepsilon^2.
\]
Hence
\[
\begin{aligned}
\left\|\widehat{\Sigma}_{N}^{\mathrm{rec}}-\widehat{\Sigma}_{N}\right\|_{2}  \leq\frac{1}{N-1}\left(\|A\|_F+\|B\|_F+\|C\|_F\right) 
  &\leq\frac{1}{N-1}{\left(N\sqrt{M_x}\varepsilon+N\sqrt{M_x}\varepsilon+N\varepsilon^2\right)} \\
 & =\frac{N}{N-1}{\left(2\sqrt{M_x}\left.\varepsilon+\varepsilon^2\right).\right.}
\end{aligned}
\]
Finally, by a corollary to Weyl's inequality which is called Spectral Stability \cite{franklin2000matrix}: Let 
$A_1$, $B_1$  be Hermitian on inner product space  with dimension $n$, then
\[
|\mu_j(A_1+B_1)-\mu_j(A_1)|\leq\|B_1\|_2,\quad j=1,\ldots,n.
\]
Take $A_1=\widehat{\Sigma}_{N}$ and $B_1=\widehat{\Sigma}_{N}^{\mathrm{rec}}-\widehat{\Sigma}_{N}$, then we get
\[
\left|\widehat{\mu}_j^{\mathrm{rec}}-\widehat{\mu}_j\right|=\left|\mu_j(\widehat{\Sigma}_N^{\mathrm{rec}})-\mu_j(\widehat{\Sigma}_N)\right|=\left|\mu_j(A_1+B_1)-\mu_j(A_1)\right|\leq\|B_1\|_2=\left\|\widehat{\Sigma}_N^{\mathrm{rec}}-\widehat{\Sigma}_N\right\|_2. \qedhere 
\]
\end{proof}


\section{Proof of \autoref{bless5}}
\label{proofofbless5}
\begin{proof}
From \autoref{convergence} and Proposition \autoref{loglam}, it follows that there exists a constant $C^\prime$ such that for $j=0,\cdots,N_{\mathrm{KL}}$, \eqref{bless} and \eqref{bless2} hold where $x_j$ are given design points (e.g., $x_j=j^2$), and $A^\ast$, $B^\ast$ are the coefficients corresponding to the true parameters $(\sigma^\ast,\ell^\ast)$.
    Define \begin{equation}
    X=
\begin{pmatrix}
1 & -x_0 \\
1 & -x_1 \\
\vdots & \vdots \\
1 & -x_{N_{\mathrm{KL}}-1}
\end{pmatrix}\in\mathbb{R}^{N_\mathrm{KL}\times2},\quad\varrho=
\begin{pmatrix}
A \\
B
\end{pmatrix},\quad y=(y_0,\ldots,y_{N_\mathrm{KL}-1})^\top,\quad\eta=(\eta_0,\ldots,\eta_{N_\mathrm{KL}-1})^\top.
\end{equation}
 Let
\[
\varrho^{\ast}:=
\begin{pmatrix}
A^{\ast} \\
B^{\ast}
\end{pmatrix},\quad\widehat{\varrho}:=
\begin{pmatrix}
A_{\mathrm{fit}} \\
B_{\mathrm{fit}}
\end{pmatrix}.
\]
Then the relation \eqref{bless}  can be written in vector form as: \[
y=X\varrho^\ast+\eta,
\]
The least squares problem \eqref{fitting} and \eqref{AB} can be stated as to find $\widehat{\varrho}$ s.t.
\[
\widehat{\varrho}=\arg\min_{\varrho\in\mathbb{R}^2}\|y-X\varrho\|_2^2.
\]
Since $N_\mathrm{KL}\geq 2$ and $\{x_j\}$ are not all identical, thus the matrix $X$ has full column rank, therefore $X^\top X$ is invertible and the least squares solution satisfies the normal equation
\[
X^\top X\widehat{\varrho}=X^\top y.
\]
Substitute $y=X\varrho^\ast+\eta$ into the above equation yields
\[
X^\top y=X^{\top}X\widehat{\varrho}=X^{\top}(X\varrho^\ast+\eta)=X^{\top}X\varrho^{\ast}+X^{\top}\eta,
\]
and hence
\[
\widehat{\varrho}-\varrho^\ast=(X^\top X)^{-1}X^{\top}\eta.
\]
Taking the 2-norm of $\eta$ and using $\|\eta\|_\infty\leq C^\prime E_\lambda$, we obtain
\[
\|\eta\|_{2}\leq\sqrt{N_{\mathrm{KL}}}\left\|\eta\right\|_{\infty}\leq\sqrt{N_{\mathrm{KL}}}C^{\prime}E_{\lambda},
\]
and therefore
\[
\|\widehat{\varrho}-\varrho^\ast\|_2\leq\|(X^\top X)^{-1}X^\top\|_2\left\|\eta\right\|_2\leq\sqrt{N_{\mathrm{KL}}}\left\|(X^\top X)^{-1}X^\top\right\|_2C^{\prime}E_\lambda.
\]
In $\mathbb{R}^2$  the inequality $|u_1|+|u_2|\leq\sqrt{2}\left\|(u_1,u_2)\right\|_2$ holds, denote 
\[
\widetilde{C}_{1}:=\sqrt{N_{\mathrm{KL}}}\left\|(X^{\top}X)^{-1}X^{\top}\right\|_{2}C^{\prime},
\]
which depends only on $N_{\mathrm{KL}}$ and $\{x_j\}$, hence
\[
   |A_{\mathrm{fit}}-A^\ast|+ |B_{\mathrm{fit}}-B^\ast|\leq \sqrt{2}\|\widehat{\varrho}-\varrho^\ast\|\leq \sqrt{2}\widetilde{C}_1E_\lambda= C_1E_\lambda,
\]
where $C_1=C_1(N_{\mathrm{KL}},\{x_j\})>0$. For the second part, by definition \eqref{ell} and \eqref{sigma}, the covariance parameters and the regression coefficients satisfy the following explicit relation:
\begin{equation}\label{appp}    \ell=2\sqrt{B},\quad\sigma=\left(\frac{e^{A}}{\sqrt{\pi}\ell}\right)^{1/2}.
\end{equation}
Denote by $\mathcal{G}: (A,B) \mapsto (\sigma,\ell)$ in \eqref{appp} the mapping from $(A,B)$ to $(\sigma,\ell)$ which is $C^1$-smooth on the open set $\{B > 0\}$. Conversely, from the relations
\begin{equation}\label{pppa}    A=\log\left(\sqrt{\pi}\frac{\sigma^{2}}{\ell}\right),\quad B=\frac{\ell^{2}}{4}.
\end{equation}
we obtain the inverse mapping $\mathcal{G}^{-1}:(\sigma,\ell)\mapsto (A,B)$ in \eqref{pppa} which is also of class $C^1$ on $(0,\infty)^2$. Obviously, we have
\[
(A^\ast,B^\ast)=\mathcal{G}^{-1}(\sigma^\ast,\ell^\ast),\qquad(\sigma_{\mathrm{est}},\ell_{\mathrm{est}})=\mathcal{G}(A_{\mathrm{fit}},B_{\mathrm{fit}}).
\]
Under the assumptions of the \autoref{bless5}, the true parameters $(\sigma^\ast,\ell^\ast)$ lie in a compact set $I\subset(0,\infty)^2$. Since $\mathcal{G}^{-1}$ is continuous, $\mathcal{G}^{-1}(I)$ is a compact set contained in $\{B>0\}$ and contains $(A^\ast,B^\ast)$. On the other hand, we also assume that the pair $(A_{\mathrm{fit}},B_{\mathrm{fit}})$, computed via \eqref{fitting} and \eqref{AB}, is contained in some bounded set $I_{AB}\subset\{B>0\}$. We can then define $\widetilde{I}:=\mathcal{G}^{-1}(I) \cup I_{AB}\subset\{B>0\}$ which is a compact set in $\mathbb{R}^2$ and contains both $(A^\ast,B^\ast)$ and $(A_{\mathrm{fit}},B_{\mathrm{fit}})$. Since $\mathcal{G}$ is $C^1$ on $\{B > 0\}$, its Jacobian $D\mathcal{G}(A, B)$ is continuous and hence bounded on the compact set $\widetilde{I}$, i.e., there exists a constant $L:=\sup_{(A,B)\in\widetilde{I}}\|D\mathcal{G}(A,B)\|<+\infty$,  for any $(A, B), (A^\prime, B^\prime) \in \widetilde{I}$, we have
\[
\begin{vmatrix}
\mathcal{G}(A,B)-\mathcal{G}(A^{\prime},B^{\prime})
\end{vmatrix}\leq\mathrm{~}L
\begin{vmatrix}
(A,B)-(A^{\prime},B^{\prime})
\end{vmatrix},
\]
due to the Lipschitz property. Substituting $(A, B) = (A_{\mathrm{fit}}, B_{\mathrm{fit}})$ and $(A^\prime, B^\prime) = (A^{\ast}, B^{\ast})$ and expanding the vector norm into its coordinate form, we obtain
\[
\left|\sigma_{\mathrm{est}}-\sigma^{\ast}\right|+\left|\ell_{\mathrm{est}}-\ell^{\ast}\right|\mathrm{~\leq~}C_{2}{\big(\left|A_{\mathrm{fit}}-A^{\ast}\right|}+\big|B_{\mathrm{fit}}-B^{\ast}\big|),
\]
where $C_2 > 0$ depends only on the compact set $\widetilde{I}$ (and therefore only on the bounded regions mentioned in the theorem). Finally, combining the above with the estimate already established in the first part we get 
\[
     |\sigma_{\mathrm{est}}-\sigma^{\ast}|+|\ell_{\mathrm{est}}-\ell^{\ast}|\leq C_{2}{\left(|A_{\mathrm{fit}}-A^{\ast}|+|B_{\mathrm{fit}}-B^{\ast}|\right)}\leq C_1C_2E_\lambda.
\]
which is exactly formula \eqref{bless10}. This completes the proof. \qedhere

\end{proof}

\bibliographystyle{amsplain}

\begin{thebibliography}{99}


\bibitem{bao2001mathematical}
G. Bao, L. Cowsar, and W. Masters (eds.), \textit{Mathematical Modeling in Optical Science},
SIAM, Philadelphia, PA, 2001.



\bibitem{kuchment2013radon}
P. Kuchment, \textit{The Radon Transform and Medical Imaging}, SIAM, Philadelphia, PA, 2014.


\bibitem{mclean2000strongly}
W. C. H. McLean, \textit{Strongly Elliptic Systems and Boundary Integral Equations},
Cambridge University Press, Cambridge, 2000.


\bibitem{colton1998inverse}
D.~Colton and R.~Kress, \textit{Inverse Acoustic and Electromagnetic Scattering Theory},
4th ed., Applied Mathematical Sciences, vol.~93, Springer, Cham, Switzerland, 2019.


\bibitem{colton1986novel}
D. Colton and P. Monk, \textit{A novel method for solving the inverse scattering problem for time-harmonic acoustic waves in the resonance region II},
SIAM J. Appl. Math. \textbf{46} (1986), no.~3, 506--523.


\bibitem{colton1985novel}
D. Colton and P. Monk, \textit{A novel method for solving the inverse scattering problem for time-harmonic acoustic waves in the resonance region},
SIAM J. Appl. Math. \textbf{45} (1985), no.~6, 1039--1053.


\bibitem{yin2025physics}
Y. Yin and L. Yan, \textit{Physics-aware deep learning framework for the limited aperture inverse obstacle scattering problem},
SIAM J. Sci. Comput. \textbf{47} (2025), no.~2, C313--C342.


\bibitem{sini2012inverse}
M. Sini and N. T. Thanh, \textit{Inverse acoustic obstacle scattering problems using multifrequency measurements},
Inverse Probl. Imaging \textbf{6} (2012), no.~4, 749--773.


\bibitem{kress2003newton}
R. Kress, \textit{Newton's method for inverse obstacle scattering meets the method of least squares},
Inverse Problems \textbf{19} (2003), no.~6, S91--S104.



\bibitem{liu2019data}
X. Liu and J. Sun, \textit{Data recovery in inverse scattering: from limited-aperture to full-aperture},
J. Comput. Phys. \textbf{386} (2019), 350--364.


\bibitem{dou2022data}
F. Dou, X. Liu, S. Meng, and B. Zhang, \textit{Data completion algorithms and their applications in inverse acoustic scattering with limited-aperture backscattering data},
J. Comput. Phys. \textbf{469} (2022), 111550.


\bibitem{bao2015inverse}
G. Bao, P. Li, J. Lin, and F. Triki, \textit{Inverse scattering problems with multi-frequencies},
Inverse Problems \textbf{31} (2015), no.~9, 093001.


\bibitem{sullivan2015introduction}
T. J. Sullivan, \textit{Introduction to uncertainty quantification},
Texts in Applied Mathematics, vol.~63, Springer, Cham, 2015.


\bibitem{feng2018efficient}
X. Feng, J. Lin, and D. P. Nicholls, \textit{An efficient Monte Carlo-transformed field expansion method for electromagnetic wave scattering by random rough surfaces},
Commun. Comput. Phys. \textbf{23} (2018), no.~3, 685--705.


\bibitem{oksendal2013stochastic}
B. \O ksendal, \textit{Stochastic Differential Equations: An Introduction with Applications},
6th ed., Universitext, Springer, Heidelberg, 2013.


\bibitem{baudoin2012lecture5}
F. Baudoin, \textit{Lecture 5. The Daniell--Kolmogorov existence theorem},
online lecture notes, March 25, 2012, available at
\url{https://fabricebaudoin.wordpress.com/2012/03/25/lecture-5-the-daniell-kolmogorov-existence-theorem/}
(accessed December 4, 2025).



\bibitem{colton1983uniqueness}
 D. Colton and B. D. Sleeman, \textit{Uniqueness theorems for the inverse problem of acoustic scattering},
IMA J. Appl. Math. \textbf{31} (1983), no.~3, 253--259.


\bibitem{li2016inverse}
P. Li, Y. Wang, Z. Wang, and Y. Zhao, \textit{Inverse obstacle scattering for elastic waves},
Inverse Problems \textbf{32} (2016), no.~11, 115018.


\bibitem{kirsch1993domain}
A. Kirsch, \textit{The domain derivative and two applications in inverse scattering theory},
Inverse Problems \textbf{9} (1993), no.~1, 81--96.


\bibitem{gintides2005local}
D.~Gintides, \textit{Local uniqueness for the inverse scattering problem in acoustics via the Faber--Krahn inequality}, Inverse Problems \textbf{21} (2005), no.~4, 1195--1205.


\bibitem{chen1997inverse}
Y. Chen, \textit{Inverse scattering via Heisenberg's uncertainty principle},
Inverse Problems \textbf{13} (1997), no.~2, 253--282.


\bibitem{xiu2010numerical}
D. Xiu, \textit{Numerical Methods for Stochastic Computations: A Spectral Method Approach},
Princeton University Press, Princeton, NJ, 2010.


\bibitem{ramm1999multidimensional}
A. G. Ramm, \textit{Multidimensional Inverse Scattering Problems},
Pitman Monographs and Surveys in Pure and Applied Mathematics, \textbf{51},
Longman Scientific \& Technical, Harlow; copublished in the United States with
John Wiley \& Sons, Inc., New York, 1992.



\bibitem{cheng2004global}
J. Cheng and M. Yamamoto, \textit{Global uniqueness in the inverse acoustic scattering problem within polygonal obstacles},
Chinese Ann. Math. \textbf{25} (2004), no.~1, 1--6.



\bibitem{alessandrini2005determining}
G. Alessandrini and L. Rondi, \textit{Determining a sound-soft polyhedral scatterer by a single far-field measurement},
Proc. Amer. Math. Soc. \textbf{133} (2005), no.~6, 1685--1691.


\bibitem{xiu2007efficient}
D. Xiu and J. Shen, \textit{An efficient spectral method for acoustic scattering from rough surfaces},
Commun. Comput. Phys. \textbf{2} (2007), no.~1, 54--72. 

\bibitem{da2023gaussian}
N. Da Costa, C. Mostajeran, and J.-P. Ortega, \textit{The Gaussian kernel on the circle and spaces that admit isometric embeddings of the circle}, in \textit{Geometric science of information. Part I}, Lecture Notes in Comput. Sci., vol.~\textbf{14071}, Springer, Cham, (2023), 426--435.


\bibitem{franklin2000matrix}
J.~N.~Franklin, \textit{Matrix theory}, Dover Publications, Mineola, NY, 2000.


\bibitem{zhang2022direct}
D.~Zhang, Y.~Wu, Y.~Wang, and Y.~Guo, \textit{A direct imaging method for the exterior and interior inverse scattering problems}, Inverse Probl.\ Imaging \textbf{16} (2022), no.~5, 1299--1323.


\bibitem{dong2019inverse}
H.~Dong, J.~Lai, and P.~Li, \textit{Inverse obstacle scattering for elastic waves with phased or phaseless far-field data}, SIAM J.\ Imaging Sci.\ \textbf{12} (2019), no.~2, 809--838.


\bibitem{yue2019numerical}
J.~Yue, M.~Li, P.~Li, and X.~Yuan, \textit{Numerical solution of an inverse obstacle scattering problem for elastic waves via the Helmholtz decomposition}, Commun.\ Comput.\ Phys.\ \textbf{26} (2019), no.~3, 809--837.


\bibitem{li2019inverse}
P.~Li, J.~Wang, and L.~Zhang, \textit{Inverse obstacle scattering for Maxwell's equations in an unbounded structure}, Inverse Problems \textbf{35} (2019), no.~9, 095002.


\bibitem{cakoni2004electromagnetic}
F.~Cakoni, D.~Colton, and P.~Monk, \textit{The electromagnetic inverse-scattering problem for partly coated Lipschitz domains}, Proc.\ Roy.\ Soc.\ Edinburgh Sect.\ A \textbf{134} (2004), no.~4, 661--682.


\bibitem{audibert2022accelerated}
L.~Audibert, H.~Haddar, and X.~Liu, \textit{An accelerated level-set method for inverse scattering problems}, SIAM J.\ Imaging Sci.\ \textbf{15} (2022), no.~3, 1576--1600.


\bibitem{chang2024novel}
Y.~Chang, Y.~Guo, H.~Liu, and D.~Zhang, \textit{A novel Newton method for inverse elastic scattering problems}, Inverse Problems \textbf{40} (2024), no.~7, 075009, 34~pp..


\bibitem{colton1996simple}
D.~Colton and A.~Kirsch, \textit{A simple method for solving inverse scattering problems in the resonance region}, Inverse Problems \textbf{12} (1996), no.~4, 383--393.


\bibitem{ji2018direct}
X.~Ji, X.~Liu, and Y.~Xi, \textit{Direct sampling methods for inverse elastic scattering problems}, Inverse Problems \textbf{34} (2018), no.~3, 035008, 22~pp.


\bibitem{kirsch2007factorization}
A.~Kirsch and N.~Grinberg, \textit{The factorization method for inverse problems},
Oxford University Press, Oxford, 2007.


\bibitem{li2020extended}
Z.~Li, Z.~Deng, and J.~Sun, \textit{Extended-sampling-Bayesian method for limited aperture inverse scattering problems}, SIAM J.\ Imaging Sci.\ \textbf{13} (2020), no.~1, 422--444.


\bibitem{khoo2019switchnet}
Y.~Khoo and L.~Ying, \textit{SwitchNet: a neural network model for forward and inverse scattering problems}, SIAM J.\ Sci.\ Comput.\ \textbf{41} (2019), no.~5, A3182--A3201.


\bibitem{chen2024solving}
J.~Chen, B.~Jin, and H.~Liu, \textit{Solving inverse obstacle scattering problem with latent surface representations}, Inverse Problems \textbf{40} (2024), no.~6, 065013, 30~pp..


\bibitem{yin2025tddm}
Y.~Yin and L.~Yan, \textit{TDDM: a transfer learning framework for physics-guided 3D acoustic scattering inversion}, J.\ Comput.\ Phys.\ \textbf{539} (2025), 114211, 22~pp..


\bibitem{ning2025direct}
J.~Ning, F.~Han, and J.~Zou, \textit{A direct sampling method and its integration with deep learning for inverse scattering problems with phaseless data}, SIAM J.\ Sci.\ Comput.\ \textbf{47} (2025), no.~2, C343--C368.


\bibitem{bao2014inverse}
G.~Bao, S.-N.~Chow, P.~Li, and H.~Zhou, \textit{An inverse random source problem for the Helmholtz equation}, Math.\ Comp.\ \textbf{83} (2014), no.~285, 215--233.


\bibitem{li2020inverse}
J.~Li, T.~Helin, and P.~Li, \textit{Inverse random source problems for time-harmonic acoustic and elastic waves}, Comm.\ Partial Differential Equations \textbf{45} (2020), no.~10, 1335--1380.


\bibitem{li2024inverse}
P.~Li and X.~Wang, \textit{Inverse scattering for the biharmonic wave equation with a random potential}, SIAM J.\ Math.\ Anal.\ \textbf{56} (2024), no.~2, 1959--1995.


\bibitem{bao2020inverse}
G.~Bao, Y.~Lin, and X.~Xu, \textit{Inverse scattering by a random periodic structure}, SIAM J.\ Numer.\ Anal.\ \textbf{58} (2020), no.~5, 2934--2952.


\bibitem{bao2018robust}
G.~Bao, Y.~Cao, Y.~Hao, and K.~Zhang, \textit{A robust numerical method for the random interface grating problem via shape calculus, weak Galerkin method, and low-rank approximation}, J.\ Sci.\ Comput.\ \textbf{77} (2018), no.~1, 419--442.


\bibitem{chen2020review}
X.~Chen, Z.~Wei, M.~Li, and P.~Rocca, \textit{A review of deep learning approaches for inverse scattering problems (invited review)}, Progr.\ Electromagn.\ Res.\ \textbf{167} (2020), 67--81.


\bibitem{koltchinskii2017concentration}
V.~Koltchinskii and K.~Lounici, \textit{Concentration inequalities and moment bounds for sample covariance operators}, Bernoulli \textbf{23} (2017), no.~1, 110--133.


\end{thebibliography}

\end{document}